\documentclass[a4paper, 10pt, reqno, dvipsnames]{amsart}

\usepackage[utf8]{inputenc}

\usepackage{microtype}

\usepackage[english]{babel}

\usepackage{verbatim}
\usepackage{newcent}

\usepackage{xcolor}

\usepackage{longtable}
\usepackage{graphicx, amsmath, amsfonts, amssymb, amsthm, amscd, amsbsy, amstext, tikz,tikz-cd, mathrsfs, mathtools, enumerate, enumitem, tensor, a4wide,xfrac}   
\usetikzlibrary{decorations.pathreplacing}
\usetikzlibrary{calc}

\usepackage{latexsym,ifthen,stmaryrd,fancyhdr,empheq,verbatim, epigraph}

\usepackage{mathabx}
\usepackage{bbm}
\usepackage{pifont}
\usepackage{manfnt}

\usepackage{capt-of}

\usepackage{dynkin-diagrams}

\usepackage[colorlinks=true, allcolors=black]{hyperref}

 \usepackage{bbm}
 \usepackage{pifont}
 \usepackage{manfnt}

\usepackage{tikz-cd}
\usepackage{mathtools}

\makeatletter
\@namedef{subjclassname@2020}{%
  \textup{2020} Mathematics Subject Classification}
\makeatother

\usepackage[OT2, T1]{fontenc} 

\usepackage[textsize=scriptsize]{todonotes}

\usepackage{scalerel, stackengine}

\usepackage{graphicx}

\usepackage{todonotes} 

\mathtoolsset{
  showmanualtags}    

\hypersetup{anchorcolor=black, linktocpage=false,
colorlinks=true,
citecolor=RawSienna,
linkcolor=Mahogany,
urlcolor=black,
breaklinks=true,
plainpages=true
}

\usetikzlibrary{positioning,shapes,shadows,arrows}
\usepackage{tikz-qtree}

\usepackage[toc,page]{appendix}
\usepackage[mathscr]{euscript} 
\allowdisplaybreaks

\DeclareRobustCommand{\SkipTocEntry}[5]{}

\newtheorem{theorem}{Theorem}
\newtheorem{prop}[theorem]{Proposition}
\newtheorem{lemma}[theorem]{Lemma}
\newtheorem{coro}[theorem]{Corollary}

\newtheorem{corollary*}{Corollary}
\newtheorem*{theorem*}{Theorem}
\newtheorem*{proposition*}{Proposition}
\newtheorem*{conjecture*}{Conjecture}
\numberwithin{equation}{section}
\numberwithin{theorem}{section}

\theoremstyle{remark}

\newtheorem{example}[theorem]{Example}
\newtheorem{rmk}[theorem]{Remark}

\newtheorem{construction}[theorem]{Construction}

\theoremstyle{definition}
\newtheorem{deff}[theorem]{Definition}
\newtheorem*{deff*}{Definition}

\newcommand{\oo}{\mathcal{O}}
\newcommand{\cc}{\mathcal{C}}

\newcommand{{\Ae}}{{A^{\rm e}}}
\newcommand{{\Aop}}{{A^{\rm op}}}

\newcommand{\foo}{\scriptstyle}

\newcommand{\ga}{\alpha}

\newcommand{\gd}{\delta} 
\newcommand{\gD}{\Delta} 
\newcommand{\gve}{\varepsilon}

\newcommand{\gvr}{\varrho} 

\newcommand{\gs}{\sigma}

\newcommand{\cC}{{\mathcal C}}

\newcommand{\Aut}{\operatorname{Aut}}

\newcommand{\Hom}{\operatorname{Hom}}

\newcommand{{\bull}}{{\scriptscriptstyle{\bullet}}}

\newcommand{\ass}{\mathsf{Ass}}

\newcommand{\N}{\ensuremath{\mathbb{N}}}

\newcommand{\id}{\mathrm{id}}

\newcommand{\Tor}{\operatorname{Tor}}
\newcommand{\Ext}{\operatorname{Ext}}

\newcommand\pig[1]{\scalerel*[5pt]{\big#1}{%
  \ensurestackMath{\addstackgap[1.5pt]{\big#1}}}}

\newcommand{\Sum}{\textstyle\sum\limits}

\newcommand{{\op}}{{{\rm op}}}
\newcommand{{\coop}}{{{\rm coop}}}

\newcommand{\kmod}{k\mbox{-}\mathbf{Mod}}

\usepackage{pict2e}

\makeatletter
\newcommand{\adjunction}[4]{%
  #1\colon #2%
  \mathrel{\vcenter{%
    \offinterlineskip\m@th
    \ialign{%
      \hfil$##$\hfil\cr
      \longrightharpoonup\cr
      \noalign{\kern-.3ex}
      \smallbot\cr
      \longleftharpoondown\cr
    }%
  }}%
  #3 \noloc #4%
}
\newcommand{\longrightharpoonup}{\relbar\joinrel\rightharpoonup}
\newcommand{\longleftharpoondown}{\leftharpoondown\joinrel\relbar}
\newcommand\noloc{%
  \nobreak
  \mspace{6mu plus 1mu}
  {:}
  \nonscript\mkern-\thinmuskip
  \mathpunct{}
  \mspace{2mu}
}
\newcommand{\smallbot}{%
  \begingroup\setlength\unitlength{.15em}%
  \begin{picture}(1,1)
  \roundcap
  \polyline(0,0)(1,0)
  \polyline(0.5,0)(0.5,1)
  \end{picture}%
  \endgroup
}
\makeatother

\newcommand{\compactlist}[1]{\setlength{\itemsep}{0pt} \setlength{\parskip}{0pt} \setlength{\leftskip}{-0.#1em}}

\begin{document}

\title{Higher structures on homology groups}

\author{Niels Kowalzig}
\author{Francesca Pratali}

\begin{abstract}
We dualise the classical fact that an operad with multiplication leads to cohomology groups which form a Gerstenhaber algebra to the context of cooperads: as a result, a cooperad with comultiplication induces a homology theory that is endowed with the structure of a Gerstenhaber coalgebra, that is, it comes with a graded cocommutative coproduct which is compatible with a coantisymmetric cobracket in a dual Leibniz sense. As an application, one obtains Gerstenhaber coalgebra structures on Tor groups over bialgebras or Hopf algebras, as well as on Hochschild homology for Frobenius algebras.
\end{abstract}

\address{N.K.: Dipartimento di Matematica, Universit\`a di Roma Tor Vergata, Via della Ricerca Scientifica~1, 00133 Roma, Italy}

\email{kowalzig@mat.uniroma2.it}

\address{F.P.: Université Sorbonne Paris Nord, LAGA, CNRS, UMR 7539, 99 Avenue JB Clément,  F-93430, Villetaneuse, France}

\email{pratali@math.univ-paris13.fr}



\maketitle

\tableofcontents

\section*{Introduction}
When defining the cohomology or homology of a mathematical object by means of (co)chain complexes or derived
functors, a priori it is not evident how to obtain, if any, more structure on it other than that of a graded
module over some base ring. Hence, it is a natural endeavour to closer examine these (co)homology groups by investigating the possible existence of further (algebraic, geometric, or topological) structure and to understand in which sense these are connected to the features of the original object, a conceptual approach which leads to the study of {\em higher structures}.

This kind of study appears to be quite well-explored, see, among many others,
\cite{Ger:TCSOAAR, GerVor:HGAAMSO, McCSmi:ASODHCC, Men:BVAACCOHA, DolTamTsy:THGAOHCOARAIF, CalWil:TOTHFT, Wil:THBFM},
in the context of {\em cohomology} groups (or cochain complexes), such as, for example, those arising from operadic structures in presence of an operad multiplication,
leading to the notion of (homotopy) Gerstenhaber or even brace algebras (that is, algebras over various operads such as the (homotopy) Gerstenhaber operads $\mathcal{G}$ and $\mathcal{G_\infty}$, or the brace operad $\mathcal{B_\infty}$) plus stronger versions in the form of BV algebras. In striking contrast, it seems to be much less so for {\em homology} groups (or chain complexes),
apart from them being the receptacle for cohomological actions in the context of noncommutative differential calculi \cite{TamTsy:NCDCHBVAAFC}, or chain formality results \cite{Tsy:FCFC, Dol:AFTFHC}.

\addtocontents{toc}{\SkipTocEntry}
\subsection*{Aims and objectives}

The present article aims to contribute in this direction, {\em i.e.}, investigate higher structures on homology groups on their own. More precisely, while it is a well-known result that an operad with multiplication induces a cochain complex the cohomology of which carries a Gerstenhaber algebra structure, that is, the structure of a graded Lie bracket and a graded commutative product,
we are not aware of any explicit result that examines the case of a {\em cooperad with comultiplication}, leading to Gerstenhaber coalgebras that relate graded cocommutative coproducts and graded coantisymmetric cobrackets.
This already starts with the apparent lack of a detailed written-down  technical definition of a cooperad (by which we mean a definition by {\em partial} co-operations, see below)  with or without comultiplication, which will be the starting point in the introductory section.

\addtocontents{toc}{\SkipTocEntry}
\subsection*{Main results}

A planar or nonsymmetric {\em (counital) cooperad} will, in Definition \eqref{carciofo}, declared to be a sequence $\cC= \{\cC(n)\}_{n \in \N}$ of $k$-modules, for $k$ a commutative ring, endowed with $k$-linear (co)operations
$$
\gD_{p,q,i} \colon \cC(p+q-1) \to \cC(p) \otimes \cC(q)
$$
for $1 \leq i \leq p$, which correspond to degrafting a (possibly upside-down) tree at its $i^{\rm th}$ branch, subject to seemingly complicated sort of coassociativity relations, which, however, become quite clear if illustrated graphically; see the main text for all technical details and revealing figures.
A {\em cooperad with comultiplication} is then a cooperad with a cooperad morphism $\cC \to \ass^*$ to the dual of the unital associative operad, which can be, however, reformulated by saying that the $k$-linear dual $\cC^*$ is equipped with three special elements in
$
\cC(2)^*$, in $
\cC(1)^*$,
and in
$
\cC(0)^*$, respectively,
subject to certain relations that turn $\cC^*$ itself into an {\em operad} with multiplication. As we shall see, not only this suffices to equip $\cC$ with the structure of a simplicial $k$-module $(\cC_\bull, d_\bull, s_\bull)$ with corresponding differential $d$, 
but it also allows for a {\em cup} (or {\em Yoneda}) {\em coproduct} as a map
$$
\cup^c\colon \cC(n) \to \bigoplus_{j=1}^{n+1} \cC(j-1) \otimes \cC(n+1-j),
$$
which can be shown to be graded cocommutative up to homotopy terms defined by the differential $d$.
Indeed, summarising Propositions \ref{sugo}, \ref{zenzero}, and \ref{farro}, we are able to prove:

\begin{theorem*}
  A cooperadic comultiplication turns a cooperad $\cC$ into a simplicial $k$-module and at the same time equips it with a cup coproduct defining a dg coassociative (counital) coalgebra structure on $\cC$, whose respective homology groups $H(\cC_\bull, d)$ become a graded cocommutative coalgebra. 
  \end{theorem*}

The next step consists in constructing a coantisymmetric cobracket, obtained by the graded cocommutator $\{ - \}$ of the totalisation of all possible partial decompositions, {\em i.e.}, the sum
$$
[-]_n \colon \cc(n) \rightarrow \textstyle\bigoplus\limits_{p+q=n+1}\cc(p)\otimes \cc(q), \qquad [-]_n \coloneqq \Sum_{p+q=n+1} \Sum_{i=1}^p
(-1)^{\foo (q-1)(p-i)}
\Delta_{p,q,i},
$$
which we term, by analogy, a {\em cobrace}. A straightforward computation in Proposition \ref{lunapop} shows that these cobraces define a pre-Lie coalgebra structure, and we can therefore show  in Corollary \ref{sushi} that for the graded cocommutator $\{-\}$, shifting the degree by one, one has:

\begin{theorem*}
  For any comultiplicative cooperad $\cc$, its respective suspended complex $\cc_\bullet[1]$ can be given the structure of a dg Lie coalgebra.
  \end{theorem*}

The final objective is to harmonise these two structures in the sense of a {\em Gerstenhaber coalgebra} that dualises the notion of a Gerstenhaber algebra: a graded cocommutative co\-product together with a coantisymmetric cobracket such that a sort of Leibniz coderivation rule holds. More precisely:

\begin{deff*}
  A {\em Gerstenhaber coalgebra} is a graded
  cocommutative $k$-coalgebra $\big(\! \bigoplus_{n \in \N} V_n, \cup^c\big)$
endowed with a graded Lie cobracket $\{-\} \colon V_n \to \bigoplus_{p+q=n+1} V_p \otimes V_q$ of degree $+1$ such that the {\em coLeibniz} identity
\begin{equation*}
(\id \otimes \cup^c) \circ \{-\} = \big(\{-\} \otimes \id + (\gvr \otimes \id) \circ (\id \otimes \{-\})\big) \circ \cup^c
  \end{equation*}
holds, where $\gvr$ denotes the mixed tensor flip as mentioned below.
\end{deff*}

Combining the two constructions $\cup^c$ and $\{-\}$ of coproduct and cobracket from the preceding theorems yields, as expected and as is dually the case in the operadic context, a Leibniz compatibility only up to homotopic terms (more precisely, the coopposite cobrace {\em is} a Leibniz coderivation compatible even on chains, while the actual cobrace is so only up to homotopy). This allows us to prove our main result in Theorem \ref{asparagi}:

\begin{theorem*}
  The homology groups arising from a cooperad with comultiplication constitute a Gerstenhaber coalgebra, that is, they are endowed with a graded cocommutative coproduct and a graded coantisymmetric cobracket that are compatible in a (co)Leibniz way.
  \end{theorem*}

Examples include a comultiplicative cooperadic structure on the chain complex computing $\Tor^B_\bull(k,k)$ for a $k$-bialgebra $B$ as well as a comultiplicative cooperadic structure on the chain complex computing $\Tor^\bull_\Ae(A,A)$ for an associative $k$-algebra $A$ if some extra structure is present, such as $A$ being Frobenius. To sum up, we have in Propositions \ref{fave} \& \ref{limone}:
\begin{theorem*} \
  \begin{enumerate}
    \compactlist{99}
    \item
  If $B$ is a
  ($k$-projective) bialgebra, then the homology groups $\Tor_\bull^B(k,k)$
  define a Gerstenhaber coalgebra, that is, are equipped with a graded cocommutative coproduct that is a coderivation compatible with a graded coantisymmetric cobracket.
\item
  If $A$ is a ($k$-projective) Frobenius algebra with Nakayama automorphism $\gs$, then the homology groups $\Tor_\bull^\Ae({}_\gs A,A)$ define a Gerstenhaber coalgebra as above.
\end{enumerate}
  \end{theorem*}

Wrapping up,
while it is well-known that $\Ext^\bull_H(k,k)$
is always a Gerstenhaber {\em algebra} and, as just seen, $\Tor_\bull^H(k,k)$ is also always a Gerstenhaber {\em coalgebra}, in Hochschild theory the situation is not entirely symmetric: here again, $\Ext^\bull_\Ae(A,A)$ is  a Gerstenhaber algebra without exception, but $\Tor_\bull^\Ae(A,A)$ is {\em not} a Gerstenhaber coalgebra in all cases: for this one needs more structure as, for example, a Frobenius kind of property, or possibly a derived version of it such as (twisted) Calabi-Yau. This finds its simple explanation in the fact that Hopf algebra cohomology (with trivial coefficients) is precisely the linear dual of Hopf algebra homology, while this is not so in Hochschild theory, the dual theory being given by
$\Tor_\bull^\Ae(A^*,A)$.

Furthermore, 
as Hochschild theory is an example of {\em bialgebroid} theory as a generalisation of bialgebras to noncommutative base rings, from this observation one immediately deduces that the situation can neither be symmetric for bialgebroid cohomology and homology, that is, while $\Ext^\bull_U(A,A)$ for a (left) bialgebroid $(U,A)$ is again always a Gerstenhaber algebra, $\Tor_\bull^U(A,A)$, if defined, in general will not be a Gerstenhaber coalgebra.

\addtocontents{toc}{\SkipTocEntry}
\subsection*{Notation and conventions}
Throughout the text, $k$ denotes a commutative ring or a field, of characteristic zero if need be. As customary, unadorned tensor products or unadorned Homs are meant to be over $k$, and $\kmod$ denotes the respective monoidal category of $k$-modules, whereas $(-)^* = \Hom_k(-, k)$ the respective duality functor.
  For two (differential) graded $k$-modules $A$ and $B$, let us introduce four kinds of tensor flips $A \otimes B \to B \otimes A$, differing only by their sign:
  let  $|a|=n$ for $a\in A_n$, and set
  \begin{equation}
    \label{piadina}
\begin{array}{rcl}
  \gs(a \otimes b) &\coloneqq& b \otimes a,
  \\
  \tau(a \otimes b) &\coloneqq& (-1)^{|a|\cdot |b|} b \otimes a,
\\
  \rho(a\otimes b) &\coloneqq& (-1)^{(|a|-1)\cdot(|b|-1)} b \otimes a, 
\\
  \gvr(a\otimes b) &\coloneqq& (-1)^{|a|\cdot(|b|-1)} b \otimes a, 
\end{array}
  \end{equation}
  that is, the {\em ungraded} flip, the {\em graded} flip, the {\em (de)suspended} flip, as well as the {\em mixed flip}.

\section{Cooperads with comultiplication}

In this preliminary section, we briefly discuss the notion of operads and, in particular, of cooperads; see, for example, \cite[\S5.8]{LodVal:AO} or \cite{VdL:OHAACKD, MarShnSta:OIATAP} for an intrinsic definition and more details on cooperads.
However, as is the case in the theory of operads, at times the definition of a cooperad by means of {\em partial} decompositions (or partial {\em co-operations}) turns out to be more handy when it comes to explicit computations.
%
%
Such a ``partial'' definition seems to have appeared only recently in \cite[Def.~1.26]{Roc:COP} in a slightly different formulation,
which is why we include this in quite some detail here, see Definition \ref{carciofo}, and illustrate what is going on by a few diagrams, to conclude by the observation that a cooperad {\em with comultiplication} amounts to a multiplicative structure on the $k$-linear dual of the cooperad, see Remark \ref{pane}.

\begin{deff} Let $\mathbb{N}$ be the discrete category whose objects are the natural numbers. A \emph{$k$-linear $\mathbb{N}$-module} is a functor
  $$
  X \colon \mathbb{N} \rightarrow \kmod.
  $$
  We denote by $\mathbb{N}$-$\mathsf{mod}_k$ the category of {\em linear $\mathbb{N}$-modules} (or  {\em $\N$-graded $k$-modules}) where morphisms are natural transformations.
\end{deff}
There are two monoidal structures on $\mathbb{N}$-$\mathsf{mod}_k$. The first one, which we denote by $\boxtimes$, is the so-called {\em Day convolution} with respect to the sum of natural numbers and the tensor product of $k$-modules: for $X,Y \in \mathbb{N}\mbox{-}\mathsf{mod}_k$, their Day convolution $X\boxtimes Y$ is defined as
$$
(X\boxtimes Y) (n) := \bigoplus_{l+h=n} X(l)\otimes Y(h).
$$
The unit is given by $k$ concentrated in arity $0$. This monoidal structure is symmetric. 

The second monoidal structure is called the \emph{composition product}, and denoted by $\circ$.
For $X,Y \in \mathbb{N}\mbox{-}\mathsf{mod}_k$, we have
$$ (X\circ Y)(n)=\bigoplus_{h\geq 0}X(h)\otimes Y^{\boxtimes h}(n).
$$
The unit, denoted by $I$, is given by $k$ concentrated in arity $1$. Observe that the composition product is not symmetric.

\begin{deff}
\label{fragole}
  A \emph{planar unital operad} in $\kmod$ (which we just call \emph{operad}) is a monoid $(\mathcal{P},\gamma,\mathbbm{1})$ in the monoidal category $( \mathbb{N}\mbox{-}\mathsf{mod}_k, \circ, I)$. 
    A {\em morphism of operads} is a morphism of such monoids.
\end{deff}

It is sometimes more useful to specify an operad structure on a $k$-linear $\mathbb{N}$-module $\mathcal{P}$ by providing a unit operation $\mathbbm{1}_\mathcal{P}\colon k \to \mathcal{P}(1)$ and {\em partial composition} maps
$$ \circ_i\colon \mathcal{P}(n) \otimes \mathcal{P}(m) \rightarrow \mathcal{P}(n+m-1),
$$
for any $n,m \in \mathbb{N}, n>0$ and any $1 \leq i \leq n$, which have to satisfy some associativity relations (see, for example, \cite[\S5.3.7]{LodVal:AO}) and unitality relations, namely  $\mathbbm{1}_\mathcal{P} \circ_1 p = p$ and $p\circ_i \mathbbm{1}_\mathcal{P} = p$ for any $p\in \mathcal{P}(n)$ and $1 \leq i \leq n$.
In particular, given a monoid $(\mathcal{P},\gamma,\mathbbm{1}),$ the $i^{\text{th}}$ partial composition $\circ_i$ relative to $n$ and $m$ as above is obtained by placing the unit in all but the first and the $i^{\text{th}}+1$ position, namely as:
$$
\circ_i\coloneqq  \gamma(-, \mathbbm{1}_\mathcal{P},\dots, \mathbbm{1}_\mathcal{P}, -, \mathbbm{1}_\mathcal{P}, \dots, \mathbbm{1}_\mathcal{P})\colon \mathcal{P}(n)\otimes k \otimes \dots \otimes k \otimes \mathcal{P}(m) \otimes k \otimes \dots \otimes k \rightarrow \mathcal{P}(n+m-1).
$$

\begin{deff}
  The \emph{associative operad} $\ass$ is defined as the planar unital operad whose space of $n$-ary operations is the free $k$-module
  on one generator $ \nu_n$,
  where
  the partial compositions are defined as
    \begin{eqnarray*}
      \circ_i \colon\ass(n)\otimes \ass(m) & \rightarrow & \ass(n+m-1),
      \\
   \nu_n\otimes \nu_m &\mapsto& \nu_{n+m-1},
    \end{eqnarray*}
and where the unit $\mathbbm{1}_\ass$ sends $1_k \in k$ to $\nu_1$.
\end{deff}

\begin{deff}
    An \emph{operad with multiplication} is a unital operad $\oo$ together with a morphism of unital operads $\phi\colon \ass\to \oo $.
\end{deff}

\begin{rmk}
  The data of an operad with multiplication $(\oo,\phi)$ is equivalent to the choice of a binary multiplication $\mu\in \oo(2)$ and of a nullary operation $e\in \oo(0)$ such that
  \begin{align*}
    & \text{$\mu$ is associative:} & &\mu\circ_1\mu=\mu \circ_2 \mu,
    \\
       & \text{$\mu$ is unital with respect to $e$:} & &\mu\circ_1 e = \mu\circ_2 e= \mathbbm{1}_\oo,
    \end{align*} where $\mathbbm{1}_\oo\colon k \to \oo(1)$ is identified with $\mathbbm{1}_\oo(1)\in \oo(1)$.
\end{rmk}

Our main focus will be on a notion dual to Definition \ref{fragole}, that is, considering comonoidal structures rather than monoidal ones:

\begin{deff}
  A \emph{planar counital cooperad} in $\kmod$ (which we just call \emph{cooperad}) is a comonoid $(\mathcal{C},\Delta, \epsilon)$ in the monoidal category $(\mathbb{N}\mbox{-}\mathsf{mod}_k, \circ, I)$. 
\end{deff}

As said at the beginning of this section, for our explicit computations it turns out to be more convenient to use the following definition by {\em partial} decompositions:

\begin{deff}
  \label{carciofo}
  A {\em planar counital cooperad} (in $\kmod$) is a sequence $\cC = \{ \cC(n) \}_{n \in \N}$ of $k$-modules together with $k$-linear (co)operations
   \begin{eqnarray*}
     \epsilon\colon \cC(1) &\rightarrow& k,
     \\
\Delta_{p,q,i}\colon \cc(p+q-1) &\rightarrow& \cc(p)\otimes \cc(q),
\end{eqnarray*}
for any $p,q \in \mathbb{N}, p>0$ 
and
$1 \leq i \leq p$, declared to be zero if $p=0$, that describes 
the $i^{\text{th}}$ partial decomposition of an element $x \in \cc(p+q-1)$ as arising out of the {\em degrafting}
of a $q$-corolla from the $i^{\text{th}}$ leaf of a $p$-corolla, subject to the following relations:
\begin{eqnarray*}
(\epsilon\otimes \id)\circ \Delta_{1,n,1} &=& \id_{\cC(n)},
\\
  (\id\otimes \epsilon)\circ \Delta_{n,1,i} &=& \id_{\cC(n)},
  \end{eqnarray*}
for all $1\leq i \leq n$, as well as
\begin{equation}
  \label{orzo}
  (\Delta_{p,q, i} \otimes \id) \circ \Delta_{p+q-1,r,j}  =
\begin{cases}
  (\id \otimes \gs) \circ (\Delta_{p,r, j-q+1} \otimes \id)\circ \Delta_{p+r-1,q, i} &  \text{if} \ i \leq j-q, 
  \\ 
  (\id\otimes \Delta_{q,r,j-i+1})\circ \Delta_{p,q+r-1,i} &  \text{if} \
j-q < i \leq j,
  \\
  (\id \otimes \gs) \circ (\Delta_{p,r, j} \otimes \id)\circ \Delta_{p+r-1,q, i+r-1} &  \text{if} \
j < i \leq p,
    \end{cases}
\end{equation}
for $p,q,r \in \N$ as well as $ 1 \leq i \leq p$ and $1 \leq j \leq p+q-1$, which we will refer to as {\em coassociativity} relations.
\end{deff}

\begin{rmk}
  In case $p = 0$, by mere definition all identities reduce to zero, since we set by default $\Delta_{0,q,i}\equiv 0$. In case $q=0$ or $r=0$, the identities \eqref{orzo} have to be read as follows: for $r=0$, the
  coassociativity relations simply come out as
\begin{equation}
  \label{orzo2}
  (\Delta_{p,q, i} \otimes \id) \circ \Delta_{p+q-1,0,j} \ \, = \ \,
 \begin{cases}
  (\id \otimes \gs) \circ (\Delta_{p,0, j-q+1} \otimes \id)\circ \Delta_{p-1,q, i} &  \text{if} \ i \leq j-q,
  \\ 
  (\id\otimes \Delta_{q,0,j-i+1})\circ \Delta_{p,q-1,i} &  \text{if} \
j-q  < i \leq j,
  \\
  (\id \otimes \gs) \circ (\Delta_{p,0, j} \otimes \id)\circ \Delta_{p-1,q, i-1} &  \text{if} \ j < i \leq p.
    \end{cases}
\end{equation}
Reading this from right to left (and calling $q$ by $r$), this also contains the case $q=0$ in \eqref{orzo}, which when turning to the left-right reading direction has to be split into two equations:
\begin{equation}
  \begin{array}{rcl}
  \label{orzo3}
  (\Delta_{p,0, i} \otimes \id) \circ \Delta_{p-1,r,j}  &=&
\begin{cases}
  (\id \otimes \gs) \circ (\Delta_{p,r, j+1} \otimes \id)\circ \Delta_{p+r-1, 0, i} &  \text{if} \ i \leq j < p,
  \\
  (\id \otimes \gs) \circ (\Delta_{p,r, j} \otimes \id)\circ \Delta_{p+r-1,0, i+r-1} &  \text{if} \ j < i \leq p,
\end{cases}
\\[4mm]
 (\id \otimes \Delta_{r, 0, i}) \circ \Delta_{p,r-1,j} 
&=&  (\Delta_{p,r, j} \otimes \id) \circ \Delta_{p+r-1,0,i+j-1}, \qquad 1 \leq j \leq p, \  1 \leq i \leq r.
\end{array}
\end{equation}
Finally, if both $q=r=0$, we ignore the last equation in \eqref{orzo3}, and obtain:
\begin{equation}
  \label{orzo4}
  (\Delta_{p,0, i} \otimes \id) \circ \Delta_{p-1,0,j}  \ \, = \ \, 
\begin{cases}
  (\id \otimes \gs) \circ (\Delta_{p,0, j+1} \otimes \id)\circ \Delta_{p-1,0, i} &  \text{if} \ i \leq j < p,
  \\
  (\id \otimes \gs) \circ (\Delta_{p,0, j} \otimes \id)\circ \Delta_{p-1,0, i-1} &  \text{if} \ j < i \leq p,
    \end{cases}
\end{equation}
and observe that this is actually one line only.
  \end{rmk}

\pagebreak

\begin{example}
Let us illustrate the third coassociativity law in \eqref{orzo} for a specific case:

\vspace*{-.3cm}
  
  \begin{center}
\begin{tikzpicture}[scale=.85, yscale=1.15, line width=0.7pt] 



 \node(A) at (-6, -2) {};
\node(a) at (-6, -3) {$\bullet$};
\node(1) at (-7.25, -4) {};
\node(2) at (-6.75, -4) {};
\node(3) at (-6.25, -4) {};
\node(4) at (-5.75, -4) {};
\node(5) at (-5.25, -4) {};
\node(6) at (-4.75, -4) {};

\draw(A)--(a);
\draw(a)--(1);
\draw(a)--(2);
\draw(a)--(3);
\draw(a)--(4);
\draw(a)--(5);
\draw(a)--(6);

\node(x1) at (-5.25, -2.5) {};
\node(x2) at (-4, -1.75) {};
\draw[|->, thick](x1)--(x2);


\node() at (-4.95, -1.75) {$\gD_{4,3,2}$};

\node(y1) at (-5.25, -4.25) {};
\node(y2) at (-4, -5) {};
\draw[|->, thick](y1)--(y2);


\node() at (-4.95, -5) {$\gD_{5,2,5}$};

\node(B) at (-3, 0) {};
\node(b) at (-3, -1) {$\bullet$};

\node(7) at (-3.75, -2) {};
\node(8) at (-3.25, -2) {$\bullet$};
\node(9) at (-2.75, -2) {};
\node(10) at (-2.25, -2) {};

\draw(B)--(b);
\draw(b)--(7);
\draw(b)--(8);
\draw(b)--(9);
\draw(b)--(10);

\node(8a) at (-3.75, -3) {};
\node(8b) at (-3.25, -3) {};
\node(8c) at (-2.75, -3) {};

\draw(8)--(8a);
\draw(8)--(8b);
\draw(8)--(8c);

\node(x3) at (-2, -1) {$=$};

\node(C) at (-1, 0) {};
\node(c) at (-1, -1) {$\bullet$};

\node(11) at (-1.75, -2) {};
\node(12) at (-1.25, -2) {};
\node(13) at (-.75, -2) {};
\node(14) at (-0.25, -2) {};

\draw(C)--(c);
\draw(c)--(11);
\draw(c)--(12);
\draw(c)--(13);
\draw(c)--(14);

\node(x3) at (-0.25, -1) {$\otimes$};

\node(D) at (.5, 0) {};
\node(d) at (.5, -1) {$\bullet$};

\node(15) at (0.0, -2) {};
\node(16) at (0.5, -2) {};
\node(17) at (1.0, -2) {};

\draw(D)--(d);
\draw(d)--(15);
\draw(d)--(16);
\draw(d)--(17);

\node(x4) at (1.15, -1) {};
\node(x5) at (2.4, -1) {};
\draw[|->, thick](x4)--(x5)  node[above,midway] {$\gD_{3,2,3} \otimes \id$};

\node(E) at (3.25, 0) {};
\node(e) at (3.25, -1) {$\bullet$};

\node(18) at (2.75, -2) {};
\node(19) at (3.25, -2) {};
\node(20) at (3.75, -2) {$\bullet$};

\draw(E)--(e);
\draw(e)--(18);
\draw(e)--(19);
\draw(e)--(20);

\node(20a) at (3.25, -3) {};
\node(20b) at (4.25, -3) {};

\draw(20)--(20a);
\draw(20)--(20b);

\node(x6) at (4.125, -1) {$\otimes$};

\node(F) at (5, 0) {};
\node(f) at (5, -1) {$\bullet$};

\node(21) at (4.5, -2) {};
\node(22) at (5.0, -2) {};
\node(23) at (5.5, -2) {};

\draw(F)--(f);
\draw(f)--(21);
\draw(f)--(22);
\draw(f)--(23);

\node(x7) at (5.8375, -1) {$=$};

\node(G) at (6.75, 0) {};
\node(g) at (6.75, -1) {$\bullet$};

\node(24) at (6.25, -2) {};
\node(25) at (6.75, -2) {};
\node(26) at (7.25, -2) {};

\draw(G)--(g);
\draw(g)--(24);
\draw(g)--(25);
\draw(g)--(26);

\node(x8) at (7.5, -1) {$\otimes$};

\node(H) at (8.25, 0) {};
\node(h) at (8.25, -1) {$\bullet$};

\node(27) at (7.75, -2) {};
\node(28) at (8.75, -2) {};

\draw(H)--(h);
\draw(h)--(27);
\draw(h)--(28);

\node(x9) at (9, -1) {$\otimes$};

\node(I) at (9.75, 0) {};
\node(i) at (9.75, -1) {$\bullet$};

\node(29) at (9.25, -2) {};
\node(30) at (9.75, -2) {};
\node(31) at (10.25, -2) {};

\draw(I)--(i);
\draw(i)--(29);
\draw(i)--(30);
\draw(i)--(31);

\node(x10) at (8.25, -2.25) {};
\node(y10) at (8.25, -4.25) {};
\draw[|->, thick](x10)--(y10)  node[right,midway] {$\id \otimes \gs$};


\node(J) at (-3.75, -4.5) {};
\node(j) at (-3.75, -5.5) {$\bullet$};

 \node(32) at (-4.75, -6.5) {};
 \node(33) at (-4.25, -6.5) {};
 \node(34) at (-3.75, -6.5) {};
 \node(35) at (-3.25, -6.5) {};
 \node(36) at (-2.75, -6.5) {$\bullet$};

\draw(J)--(j);
\draw(j)--(32);
\draw(j)--(33);
\draw(j)--(34);
\draw(j)--(35);
\draw(j)--(36);

\node(36a) at (-3.25, -7.5) {};
\node(36b) at (-2.25, -7.5) {};

\draw(36)--(36a);
\draw(36)--(36b);

\node(y1) at (-2.25, -5.5) {$=$};

\node(K) at (-1, -4.5) {};
\node(k) at (-1, -5.5) {$\bullet$};

 \node(37) at (-2, -6.5) {};
 \node(38) at (-1.5, -6.5) {};
 \node(39) at (-1, -6.5) {};
 \node(40) at (-.5, -6.5) {};
 \node(41) at (0, -6.5) {};

\draw(K)--(k);
\draw(k)--(37);
\draw(k)--(38);
\draw(k)--(39);
\draw(k)--(40);
\draw(k)--(41);

\node(y2) at (0, -5.5) {$\otimes$};

\node(L) at (.75, -4.5) {};
\node(l) at (.75, -5.5) {$\bullet$};

 \node(42) at (.25, -6.5) {};
 \node(43) at (1.25, -6.5) {};

\draw(L)--(l);
\draw(l)--(42);
\draw(l)--(43);

\node(y3) at (1.35, -5.5) {};
\node(y4) at (2.6, -5.5) {};
\draw[|->, thick](y3)--(y4)  node[above,midway] {$\gD_{3,3,2} \otimes \id$};

\node(M) at (3.25, -4.5) {};
\node(m) at (3.25, -5.5) {$\bullet$};

 \node(44) at (2.75, -6.5) {};
 \node(45) at (3.25, -6.5) {$\bullet$};
 \node(46) at (3.75, -6.5) {};

\draw(M)--(m);
\draw(m)--(44);
\draw(m)--(45);
\draw(m)--(46);

 \node(45a) at (2.75, -7.5) {};
 \node(45b) at (3.25, -7.5) {};
 \node(45c) at (3.75, -7.5) {};

 \draw(45)--(45a);
\draw(45)--(45b);
\draw(45)--(45c);

\node(y5) at (4.125, -5.5) {$\otimes$};

\node(N) at (5, -4.5) {};
\node(n) at (5, -5.5) {$\bullet$};

\node(47) at (4.5, -6.5) {};
\node(48) at (5.5, -6.5) {};

\draw(N)--(n);
\draw(n)--(47);
\draw(n)--(48);

\node(y6) at (5.8375, -5.5) {$=$};

\node(O) at (6.75, -4.5) {};
\node(o) at (6.75, -5.5) {$\bullet$};

\node(49) at (6.25, -6.5) {};
\node(50) at (6.75, -6.5) {};
\node(51) at (7.25, -6.5) {};

\draw(O)--(o);
\draw(o)--(49);
\draw(o)--(50);
\draw(o)--(51);

\node(y7) at (7.5, -5.5) {$\otimes$};

\node(P) at (8.25, -4.5) {};
\node(p) at (8.25, -5.5) {$\bullet$};

\node(52) at (7.75, -6.5) {};
\node(53) at (8.25, -6.5) {};
\node(54) at (8.75, -6.5) {};

\draw(P)--(p);
\draw(p)--(52);
\draw(p)--(53);
\draw(p)--(54);

\node(y8) at (9, -5.5) {$\otimes$};

\node(Q) at (9.75, -4.5) {};
\node(q) at (9.75, -5.5) {$\bullet$};

\node(55) at (9.25, -6.5) {};
\node(56) at (10.25, -6.5) {};

\draw(Q)--(q);
\draw(q)--(55);
\draw(q)--(56);

 \end{tikzpicture}

\vspace*{-.3cm}

\captionof{figure}{The relation $(\gD_{3,2,3} \otimes \id) \circ \gD_{4,3,2} = (\id \otimes \gs) \circ (\gD_{3,3,2} \otimes \id) \circ \gD_{5,2,5}$.}
\label{rel}
 \end{center}

   \end{example}

\begin{deff}
    A \emph{cooperad with comultiplication} is a counital cooperad $\cc$ together with a morphism of counital cooperads $\psi\colon \cc \to \ass^*$.
\end{deff}

\begin{rmk}
  \label{pane}
The linear dual $\cC^*$ of a cooperad $\cC$ is automatically an operad. Its identity element consists of a linear map $\mathbbm{1}_{\cC^*}\colon k\to \cC^*(1)$, and identifying $\mathbbm{1}_{\cC^*}$ with the functional $\mathbbm{1}_{\cC^*}(1)$, we see that $\mathbbm{1}_{\cC^*}= \mathbbm{1}^c \in {\cC^*}(1)$. Since moreover $\ass^{**}\simeq \ass$, canonically once we fix the generating operations in each arity, if $(\cc,\psi)$ is a cooperad with comultiplication, then $(\cc^*,\psi^*)$ is an operad with multiplication, where the distinguished multiplication is $\mu^c = \psi^*_2(\nu_2) \in \cc^*(2)$ and the distinguished unit is $e^c= \psi^*_0(\nu_0) \in \cc^*(0)$.
Hence, a cooperad with comultiplication $(\cC, \psi)$ can be equivalently described by a quadruple $(\cC, \mu^c, \mathbbm{1}^c, e^c)$ as in Figure \ref{spe} right below.


\begin{center}
 \begin{tikzpicture}[scale=.75, yscale=.95, line width=0.7pt]
  \node(comult) at (-1,.5){$\mu^c \colon \cC(2) \to k$,};
 \node(1) at (-2,0){};
 \node(2) at (0,0){};
 \node(b) at (-1,-1){$\mu^c$};
\node(3) at (-1,-2){};
 
 \draw(b)--(2);
 \draw(b)--(1);
 \draw[[-](3)--(b);
 
 \node(coid) at (3,.5){$\mathbbm{1}^c \colon \cC(1) \to k$,};
 \node(c) at (3,0){};
 \node(5) at (3,-1){$\mathbbm{1}^c$};
 \node(6) at (3,-2) {};
 
 \draw(5)--(c);
 \draw[[-](6)--(5);
 
 \node(coun) at (7,.5) {$ e^c \colon \cC(0) \to k$,};
 \node(d) at (7,-1){$e^c$};
 \node(7) at (7,-2){};
 
 \draw[[-](7)--(d);
 \end{tikzpicture}

\vspace*{-.4cm}

\captionof{figure}{Special elements in a counital cooperad with comultiplication.}
\label{spe}

\end{center}

\smallskip

\noindent  Here, the output, as an element of $k$, is represented by a truncated line to distinguish it from elements in $\cC$. Then, (co)associativity of $\mu^c$ can be written as the equality,
\begin{equation}
\label{polenta}
(\mu^c\otimes \mu^c)\circ \Delta_{2,2,1}= (\mu^c\otimes \mu^c)\circ \Delta_{2,2,2},
\end{equation}
which we illustrate by

\vspace*{-.3cm}

\begin{center}
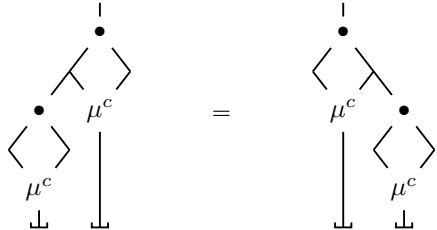

\begin{tikzpicture}[scale=.8, yscale=.65, line width=0.7pt]

\node(1) at (-2,0) {};
\node(2) at (-2,-1){$\bullet$};
\node(3) at (-2.5,-2){};
\node(4) at (-1.5,-2){};
\node(5) at (-3,-3){$\bullet$};
\node(6) at (-2,-3){$\mu^c$};
\node(7) at (-3.5,-4){};
\node(8) at (-2.5,-4){};
\node(9) at (-3,-5){$\mu^c$};
\node(10) at (-3,-6.25){};
\node(11) at (-2,-6.25){};

\draw(1)--(2);
\draw(3.center)--(2);
\draw(4.center)--(2);
\draw(3.center)--(6);
\draw(4.center)--(6);
\draw(3.center)--(5);
\draw(7.center)--(5);
\draw(8.center)--(5);
\draw(7.center)--(9);
\draw(8.center)--(9);
\draw[[-](10)--(9);
\draw[[-](11)--(6);

\node(uguale) at (0,-3.125){$=$};

\begin{scope}[yscale=1,xscale=-1]


  \node(1) at (-2,0) {};
\node(2) at (-2,-1){$\bullet$};
\node(3) at (-2.5,-2){};
\node(4) at (-1.5,-2){};
\node(5) at (-3,-3){$\bullet$};
\node(6) at (-2,-3){$\mu^c$};
\node(7) at (-3.5,-4){};
\node(8) at (-2.5,-4){};
\node(9) at (-3,-5){$\mu^c$};
\node(10) at (-3,-6.25){};
\node(11) at (-2,-6.25){};

\draw(1)--(2);
\draw(3.center)--(2);
\draw(4.center)--(2);
\draw(3.center)--(6);
\draw(4.center)--(6);
\draw(3.center)--(5);
\draw(7.center)--(5);
\draw(8.center)--(5);
\draw(7.center)--(9);
\draw(8.center)--(9);
\draw[[-](10)--(9);
\draw[[-](11)--(6);

\end{scope}
\end{tikzpicture}

\vspace*{-.4cm}

\nopagebreak

\captionof{figure}{Coassociativity in a cooperad with comultiplication.}
\label{coa}
\end{center}

\noindent In contrast, (co)unitality of $\mu^c$ with respect to $e^c$ corresponds to the following equations:
\begin{equation}\label{polentaunita}
(\mu^c\otimes e^c)\circ \Delta_{2,0,1}= (\mu^c\otimes e^c)\circ \Delta_{2,0,2}= \mathbbm{1}^c,
\end{equation}
which can be represented graphically as:
\begin{center}
\begin{tikzpicture}[scale=.7, yscale=.7, line width=0.7pt]

\node(1) at (-2,0) {};
\node(2) at (-2,-1){$\bullet$};
\node(3) at (-2.5,-2){};
\node(4) at (-1.5,-2){};
\node(5) at (-3,-3.85){$e^c$};
\node(a) at (-3,-3){$\bullet$};
\node(6) at (-2,-3){$\mu^c$};
\node(7) at (-3,-5){};
\node(8) at (-2,-5){};

\draw(1)--(2);
\draw(3.center)--(2);
\draw(4.center)--(2);
\draw(3.center)--(6);
\draw(4.center)--(6);
\draw(3.center)--(a);
\draw[[-](7)--(5);
\draw[[-](8)--(6);

\node(uguale) at (0,-2.125){$=$};

\begin{scope}[yscale=1,xscale=-1]


\node(1) at (-2,0) {};
\node(2) at (-2,-1){$\bullet$};
\node(3) at (-2.5,-2){};
\node(4) at (-1.5,-2){};
\node(5) at (-3,-3.85){$e^c$};
\node(a)at(-3,-3){$\bullet$};
\node(6) at (-2,-3){$\mu^c$};
\node(7) at (-3,-5){};
\node(8) at (-2,-5){};

\draw(1)--(2);
\draw(3.center)--(2);
\draw(4.center)--(2);
\draw(3.center)--(6);
\draw(4.center)--(6);
\draw(a)--(3.center);
\draw[[-](7)--(5);
\draw[[-](8)--(6);

\end{scope}

\node(uguale) at (4.5,-2){$=$};

\node(A) at (6,0){};
\node(Ba) at (6,-1.5){$\bullet$};
\node(B) at (6,-3){$\mathbbm{1}^c$};
\node(C) at (6,-5){};

\draw(A)--(Ba);
\draw(B)--(Ba);
\draw[[-](C)--(B);

\end{tikzpicture}

\vspace*{-.2cm}

\captionof{figure}{Counitality in a cooperad with comultiplication.}
\label{cou}

\end{center}

\medskip

\noindent Both (co)associativity and (co)unitality will be of multiple use in subsequent computations.
\end{rmk}

\section{Cobraces and cobrackets}

We start this section by describing how pre-Lie and Lie coalgebras are connected, see, for example, \cite{Mic:LC} for more details on this and related aspects, and afterwards we focus our attention on the situation of cooperads.

\begin{deff}
  \label{briciole1}
  A \emph{(dg) pre-Lie coalgebra} is a chain complex $(L,d)$ together with a morphism $[-]\colon L \to L\otimes L$, which we may call \emph{cobrace}, such that:
\begin{itemize}
  \compactlist{99}
\item $d$ is a coderivation for $[-]$, meaning that
   \[
    [-]\circ d =
\big(d\otimes \id + \tau \circ (d \otimes \id) \circ \tau\big)\circ [-]
\]
in any degree of $L$;
\item the cobrace satisfies the {\em pre-coJacobi identity}  
\begin{equation}\label{precoja}
 \big([-]\otimes \id - \id \otimes [-]\big)\circ [-] = \big(\id\otimes \tau\big)\circ \big([-]\otimes \id - \id \otimes [-] \big)\circ [-],
\end{equation} which means that the coassociator of $[-]$ is symmetric in the last two variables.
\end{itemize} 
\end{deff}
\noindent Observe that, with the Koszul sign rule in mind, the coderivation property could be equally expressed as
$
    [-]\circ d =
    (d\otimes \id + \id\otimes d)\circ [-],
    $
hiding the signs which, at times, has its advantages. 
    
\begin{deff}
  \label{coliedef}
A \emph{(dg) Lie coalgebra} is a chain complex $(L,d)$ together with a map of chain complexes $\{-\}\colon L \to L\otimes L$, called the \emph{cobracket}, such that:
\begin{itemize}
  \compactlist{99}
\item
 the differential $d$ is a coderivation with respect to the cobracket $\{-\}$;
\item the cobracket is coantisymmetric, which means
  \[
  \text{Im}(\{-\})\subseteq \text{Im}(\id-\tau),
  \]
  where $\tau$ is the graded flip as in \eqref{piadina};
\item the {\em coJacobi identity}
  \[
  \big(\id+ (\id\otimes \tau)\circ (\tau\otimes \id) + (\tau\otimes \id)\circ (\id\otimes \tau)\big)
  \circ \big(\{-\}\otimes \id \big)\circ \{-\} = 0
  \]
  holds.
\end{itemize}
\end{deff}

\begin{rmk}
  \label{baguette}
One could have chosen to formulate the coJacobi identity in Definition \ref{coliedef} by using $\big(\id\otimes \{-\}\big)\circ \{-\}$  instead of $\big( \{-\}\otimes \id\big)\circ \{-\}$, that is, by asking that 
$$
\big(\id+ (\id\otimes \tau)\circ (\tau\otimes \id) + (\tau\otimes \id)\circ (\id\otimes \tau)\big)
\circ \big(\id\otimes \{-\}\big) \circ \{-\} = 0.
$$
If the ground ring has odd characteristic and the cobracket is already known to be coantisymmetric, then
thanks to the identity
\begin{equation}
  \label{sinistradestra}
(\id\otimes \{-\})\circ \{-\}= (\tau\otimes \id)\circ (\id\otimes \tau)\circ (\{-\}\otimes \id)\circ \tau \circ \{-\}
\end{equation}
the two formulations of the coJacobi identity are equivalent.
\end{rmk}

The subsequent proposition is probably well-known, but for the reader's convenience we
include a detailed proof here:

\begin{prop}
  \label{pesce}
  If $\big(L,d,[-]\big)$ is a (dg) pre-Lie coalgebra, then $\big(L,d,\{-\}\big)$, where
  \begin{equation}
    \label{arancia2}
    \{-\}\coloneqq (\id-\tau)\circ [-],
    \end{equation}
  is a (dg) Lie coalgebra.
\end{prop}

\begin{proof}
We only need to check the coJacobi identity for $\{-\}$, and for this we denote: 
\begin{align*}
  & A\coloneqq ([-]\otimes \id)\circ [-], &
  B \coloneqq (\id\otimes [-])\circ [-],
  & &
  \xi \coloneqq (\id\otimes \tau )\circ (\tau\otimes \id). 
\end{align*}
Observe the following useful identities:
\begin{align}
  \label{eq3}
& \xi^2=(\tau\otimes \id)\circ (\id\otimes \tau), &  \xi^3=\id, & 
\end{align} 
and still
\begin{align}
  \label{eq2}
    & \xi \circ (\tau \otimes \id)= (\id\otimes \tau).
\end{align}
The pre-coJacobi identity then states that 
\begin{align}
  \label{eq1}
& A= B + (\id\otimes\tau) \circ A - (\id\otimes \tau) \circ B,
\end{align}
 while the coJacobi identity can be formulated as: 
 $$
 \big(\id+\xi + \xi^2\big)\big(A-\xi \circ B - (\tau\otimes \id) \circ A + (\tau\otimes \id)\circ \xi \circ B\big) =0.
 $$ 
With this, let us compute:
  \begin{equation*}
    \begin{array}{rcl}
      &&    
      \big(\id+\xi + (\tau\otimes \id)(\id\otimes\tau)\big)
          \circ \big(A-\xi \circ B - (\tau\otimes \id)\circ A + (\tau\otimes \id) \circ \xi \circ B \big) 
      \\[1mm]
      &
      \stackrel{\scriptscriptstyle{\eqref{eq1}}}{=}
      &
       B + (\id\otimes\tau) \circ A - (\id\otimes\tau) \circ B + \xi \circ A + (\tau\otimes \id) \circ (\id\otimes\tau) \circ A - \xi \circ B - \xi^2 \circ B - B        
        \\[1mm]
& &
 \
 - (\tau\otimes \id) \circ A - (\id\otimes\tau) \circ A - (\tau\otimes \id)\circ \xi \circ A + (\tau\otimes \id)\circ \xi \circ B
        \\[1mm]
& &
 \
 + (\tau\otimes \id) \circ B + (\tau\otimes 1) \circ \xi^2 \circ  B 
\\[1mm]
      &
      \stackrel{\scriptscriptstyle{\eqref{eq2}}}{=}
      &
      - (\id\otimes\tau) \circ B + \xi \circ A + (\tau\otimes \id)\circ (\id\otimes\tau) \circ A  - \xi \circ B - \xi^2 \circ B - (\tau\otimes \id) \circ A  
        \\[1mm]
& &
        \
        -(\tau\otimes \id) \circ \xi \circ A  
 + (\tau\otimes \id) \circ \xi \circ B + (\tau\otimes \id) \circ B + ( \tau\otimes \id) \circ \xi^2 \circ B
\\[1mm]
      & 
      \stackrel{\scriptscriptstyle{\eqref{eq1}}}{=} 
        &
      \xi \circ A + \xi^2 \circ A  - \xi \circ B - \xi^2 \circ B  - (\tau\otimes \id) \circ A  -(\tau\otimes \id) \circ \xi \circ A
        \\[1mm]
& &
        \
      + (\tau\otimes \id)\circ \xi \circ B + (\tau\otimes \id) \circ B 
\\[1mm]
      &
      =
      & (\xi+\xi^2) \circ (A-B) - (\tau\otimes \id)\circ A - (\tau\otimes \id)\circ \xi \circ A + (\tau\otimes \id)\circ \xi \circ B + (\tau\otimes \id) \circ B
      \\[1mm]
    &
    =
    &
    \xi \circ (\id\otimes\tau) \circ A + \xi^2 \circ (\id\otimes\tau)\circ A -(\tau\otimes \id)\circ A - (\tau\otimes \id)\circ \xi \circ A - \xi\circ (\id\otimes \tau)\circ B 
    \\[1mm]
& &
 \
 - \xi^2 \circ (\id\otimes \tau) \circ B + (\tau\otimes \id) \circ \xi \circ B + (\tau\otimes \id) \circ B
       \\[1mm]
     &
      \stackrel{\scriptscriptstyle{}}{=}
      &\big(\xi\circ (\id\otimes \tau) - (\tau\otimes \id)\circ \xi\big) \circ (A-B)
  \\[1mm]
     &
\stackrel{\scriptscriptstyle{\eqref{eq2}}}{=}
      &
0,
      \end{array}
   \end{equation*}
which concludes the proof of the coJacobi identity. One still needs to check that, if $d$ is a coderivation with respect to the cobrace $[-]$, it is still a coderivation with respect to its symmetrisation $\{-\}$. This is true since $(d\otimes \id)\circ \tau = \tau \circ (\id\otimes d)$, where the Koszul signs are always applied.
We have therefore shown that $\big(L,d,\{-\}\big)$ is a dg Lie coalgebra, as desired.
\end{proof}

\medskip
 \begin{center}
 * \qquad * \qquad *
   \end{center}
 \medskip

Turning to our objectives, 
in case of a cooperad, one has a canonical construction for a cobrace turning it into a pre-Lie coalgebra in the sense that the pre-coJacobi identity \eqref{precoja} is fulfilled. To this end, let us define:

\pagebreak

\begin{deff}
  \label{briciole2}
  Let
  $\cC=\bigoplus_{n \in \N} \cc(n)$ be a cooperad.
  \begin{enumerate}
    \compactlist{99}
\item
  For any $p>0, q\geq 0$,
the map
\[
    [-]_n \colon \cc(n) \rightarrow \textstyle\bigoplus\limits_{p+q=n+1}\cc(p)\otimes \cc(q)
    \]
given by the totalisation
\begin{equation}
      \label{sale}
            [x]_n \coloneqq \Sum_{p+q=n+1} \Sum_{i=1}^p
(-1)^{\foo (q-1)(p-i)}
\Delta_{p,q,i}(x)
    \end{equation} 
    of all partial decompositions
    $
    \Delta_{p,q,i}\colon \cc(p+q-1)\rightarrow \cc(p) \otimes \cc(q)
    $
    is called the ({\em canonical}\hskip .3pt) {\em cobrace} on $\cC(n)$. Since $\gD_{0, q, i} = 0$ for all $q,i$, this can be rewritten as
    \begin{equation}
      \label{sale2}
            [x]_n= \Sum^{n+1}_{p=1} \Sum_{i=1}^p
(-1)^{\foo (n-p)(p-i)}
\Delta_{p,n+1-p,i}(x).
    \end{equation} 
\item
  The degree $-1$ morphism
  \begin{equation}
    \label{arancia}
  \{-\} \coloneqq (\id-\rho)\circ [-]
\end{equation}
will be referred to as the ({\em canonical}\hskip .3pt) {\em cobracket} of the cooperad $\cC$. 
\end{enumerate}
\end{deff}

\noindent Observe the different signs for $\{-\}$ from Proposition \ref{pesce} in contrast to that of the canonical co- bracket just defined. We shall come back to this in more detail in Corollary \ref{sushi}.

\begin{prop}
  \label{lunapop}
  For any cooperad $\cc=\bigoplus_n \cc(n)$,
  the cobrace
 \eqref{sale} satisfies the pre-coJacobi identity of Definition \eqref{precoja}.
\end{prop}

\begin{rmk}
Prior to giving a rigorous proof of the pre-coJacobi identity, let us illustrate the idea behind it. Observe first that both in  $([-]\otimes \id)\circ [-]_n$ and $(\id\otimes [-])\circ [-]_n$ there are two types of three-vertex trees: those whose vertices lie in a same path from the root to one of the leaves, and those where the second and third vertices are independent, {\em i.e.}, the two types

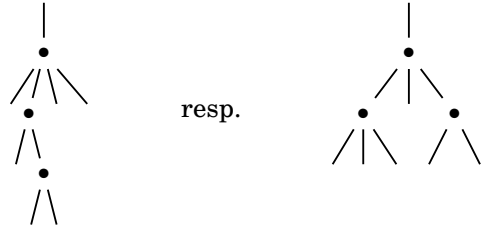
\vspace*{-.25cm}

\begin{center}
\begin{tikzpicture}[scale=.8, yscale=1, line width=0.7pt]

\node(B) at (-2.50, 0) {};
\node(b) at (-2.50, -1) {$\bullet$};
\node(7) at (-3.15, -2) {};
\node(c) at (-2.75, -2) {$\bullet$};
\node(8) at (-2.25, -2) {};
\node(9) at (-1.65,-2){};
\node(d) at (-2.50,-3) {$\bullet$};
\node(10) at (-3,-3){};
\node(11) at (-2.75,-4){};
\node(12) at (-2.25, -4){};

\draw(B)--(b);
\draw(b)--(7);
\draw(b)--(c);
\draw(b)--(8);
\draw(b)--(9);
\draw(c)--(d);
\draw(c)--(10);
\draw(d)--(11);
\draw(d)--(12);


\node(resp) at (0.25,-2){resp.};

\node(C) at (3.50,0){};
\node(e) at (3.50,-1){$\bullet$};
\node(13) at (3.50,-2){};
\node(f)at (2.75,-2) {$\bullet$};
\node(g) at (4.25,-2) {$\bullet$};
\node(14) at (3.40,-3){};
\node(15) at (2.15,-3){};
\node(16) at (2.75, -3){};
\node(17) at (3.75, -3) {};
\node(18) at (4.75, -3) {};

\draw(C)--(e);
\draw(e)--(13);
\draw(e)--(f);
\draw(e)--(g);
\draw(f)--(14);
\draw(f)--(15);
\draw(f)--(16);
\draw(g)--(17);
\draw(g)--(18);

\end{tikzpicture}
\end{center}
\vspace*{-.35cm}
%
%
%
%
%
The two different types of trees represent the different types of coassociativity illustrated in  \eqref{orzo}, where three equations are needed to take into account the possible permutations of the tensor factors. The coassociator consists of a sum of decompositions of the two types above, and by rearranging the trees in the coassociator, the ones of the type on the left hand side mutually cancel so that we are left with the difference between trees of the type on the right hand side, which, grouping together, yields an expression of the form:

\vspace*{-.2cm}

  \begin{center}
\begin{tikzpicture}[scale=.8, yscale=1, line width=0.7pt] 
  
 \node(A) at (-4, -1) {};
\node(a) at (-4, -2) {$\bullet$};
\node(1) at (-5.25, -3) {};
\node(2) at (-4.75, -3) {};
\node(3) at (-4.25, -3) {};
\node(4) at (-3.75, -3) {};
\node(5) at (-3.25, -3) {};
\node(6) at (-2.75, -3) {};

\draw(A)--(a);
\draw(a)--(1);
\draw(a)--(2);
\draw(a)--(3);
\draw(a)--(4);
\draw(a)--(5);
\draw(a)--(6);

\node(x1) at (-5.25, -2.5) {};
\node(x2) at (-4, -1.75) {};


\node() at (-2.50, -1.75) {$\mapsto$};

\node(y1) at (-1.25, -4.25) {};
\node(y2) at (-0, -5) {};

\node() at (-1.50, -1.75) {$\sum \pm$};

\node(C) at (0.50,0){};
\node(e) at (0.50,-1){$\bullet$};
\node(13) at (0.50,-2){};
\node(f)at (0,-2) {$\bullet$};
\node(g) at (1,-2) {$\bullet$};
\node(14) at (0,-3){};
\node(15) at (-0.5,-3){};
\node(16) at (0.5, -3){};
\node(17) at (0.75, -3) {};
\node(18) at (1.35, -3) {};

\draw(C)--(e);
\draw(e)--(13);
\draw(e)--(f);
\draw(e)--(g);
\draw(f)--(14);
\draw(f)--(15);
\draw(f)--(16);
\draw(g)--(17);
\draw(g)--(18);

\node() at (2, -1.75) {$-$};

\node(E) at (3.50,0){};
\node(m) at (3.50,-1){$\bullet$};
\node(25) at (3.50,-2){};
\node(n)at (3,-2) {$\bullet$};
\node(o) at (4,-2) {$\bullet$};
\node(26) at (2.75,-3){};
\node(27) at (3.25,-3){};
\node(28) at (3.5, -3){};
\node(29) at (4, -3) {};
\node(30) at (4.5, -3) {};

\draw(E)--(m);
\draw(m)--(25);
\draw(m)--(n);
\draw(m)--(o);
\draw(n)--(26);
\draw(n)--(27);
\draw(o)--(28);
\draw(o)--(29);
\draw(o)--(30);

 \end{tikzpicture}
  \end{center}

\vspace*{-1.9cm}
  
\noindent This is evidently symmetric in the last two tensor factors, as desired. After this conceptual explanation, let us prove the proposition rigorously.
\end{rmk}

\begin{proof}[Proof of Proposition \ref{lunapop}]
Let us prove that the pre-Jacobi identity holds, {\em i.e.}, that for any $n \in \N$ the coassociator
\begin{equation}
  \label{zero}
\big([-]\otimes \id - \id \otimes [-]\big) \circ [-]_n
\end{equation}
is symmetric in the last two tensor factors.
By doing so, in this proof (and only in this one) we shall be quite detailed
in how to explicitly apply the coassociativity relations \eqref{orzo}--\eqref{orzo4} to illustrate the idea, and much
less so in the following proofs of similar nature. Indeed, one can write the first term as follows:
\begin{equation}
  \label{prima}
([-]\otimes \id)\circ [-]_n= \sum_{p=1}^{n+1}\sum_{j=1}^p \sum_{a=1}^{p+1} \sum_{h=1}^a (-1)^{\ga(n,p,j,a,h)} (\Delta_{a,p+1-a,h}\otimes \id)\circ \Delta_{p,n+1-p,j},
\end{equation}
where the sign is given by $\ga(n,p,j,a,h)= (n-p)(p-j)+(p-a)(a-h)$.


Using the fact that $\sum_{p=1}^{n+1}\sum_{a=1}^p= \sum_{a=1}^{n+1}\sum_{p=a}^{n+1}$, renaming $p$ by $a$ (and vice versa), and calling $b=a+1-p$, we can rewrite the expression in \eqref{prima} as:
\begin{align*}
  \sum_{p=1}^{n+1}\sum_{h=1}^p\sum_{b=1}^{n+2-p} \sum_{j=1}^{b+p-1} (-1)^{\ga(n,b+p-1,j,p,h)} (\Delta_{p,b,h}\otimes \id)\circ \Delta_{b+p-1,n+2-b-p,j} &
  \\
  + \sum_{a=1}^{n+1}\sum_{j=1}^a\sum_{h=1}^{a+1} (-1)^{\ga(n,a,j,p,h)} (\Delta_{a+1,0,h}\otimes \id)\circ \Delta_{a,n+1-a,j}.& 
\end{align*} 
On the other hand, the second term of \eqref{zero} writes as:
\begin{equation}
  \label{seconda}
(\id\otimes [-])\circ [-]_n=\sum_{p=1}^{n+1}\sum_{i=1}^p \sum_{b=1}^{n+2-p}\sum_{k=1}^b (-1)^{\beta(n,p,i,b,k)} (\id\otimes \Delta_{b,n+2-p-b,k})\circ \Delta_{p,n+1-p,i},
\end{equation}
where the sign arises from $\beta(n,p,i,b,k)= (n-p)(p-i)+ (n+1-p-b)(b-k)$.

As said, we will now use repeatedly the coassociativity relations to rewrite the summands. The first application of the coassociativity laws will be performed explicitly, and afterwards we shall omit most of the details in the subsequent ones.  
The second
line in \eqref{orzo2} states that, for natural numbers $\underline{p}, \underline{q}, \underline{i}, \underline{r}, \underline{j}$, with $1 \leq \underline{i} \leq \underline{p}$ and $1 \leq \underline{j}\leq \underline{p}+\underline{q}-1$, whenever $\underline{j}-\underline{q}\leq \underline{i}\leq \underline{j}$, we have that
\[
(\Delta_{\underline{p},\underline{q},\underline{i}} \otimes \id )\circ \Delta_{\underline{p}+\underline{q}-1, \underline{r}, \underline{j}}= (\id \otimes \Delta_{\underline{q}, \underline{r}, \underline{j}-\underline{i}+1})\circ \Delta_{\underline{p}, \underline{q}+ \underline{r}-1, \underline{i}}.
\] 
If we now rename the indices in \eqref{seconda} as:
\begin{equation*}
  \underline{i}=i, \quad
  \underline{p}=p, \quad
  \underline{q}=b, \quad
  \underline{r}= n+2-p-b, \quad
\underline{j}-\underline{i}+1= k,
\end{equation*}
we see that the condition that $\underline{j}-\underline{q}\leq \underline{i}\leq \underline{j}$ is always verified, hence we can rewrite \eqref{seconda} as: 
\begin{align*}
  \sum_{p=1}^{n+1}\sum_{i=1}^p \sum_{b=1}^{n+2-p}\sum_{j=i}^{ b+i-1} (-1)^{\beta(n,p,i,b,j-i+1)} (\Delta_{p,b,i}\otimes \id)\circ \Delta_{p+b-1,n+2-p-b,j},
\end{align*} where we have set $j= k+i-1$.
Since $\ga(n,b+p-1,j,p,i) \equiv \beta(n,p,i,b,j-i+1)$ modulo $2$, we can erase part of the summands and write the coassociator as 
\begin{align*}
\sum_{p=2}^{n+1}\sum_{h=2}^p\sum_{b=1}^{n+2-p}\sum_{j=1}^{h-1}(-1)^{\ga(n,b+p-1,j,p,h)} (\Delta_{p,b,h}\otimes \id)\circ \Delta_{b+p-1,n+2-b-p,j} &
&
\\
+ \sum_{a=1}^{n+1}\sum_{j=1}^a\sum_{h=1}^{a+1} (-1)^{\ga(n,a,j,a+1,h)} (\Delta_{a+1,0,h}\otimes \id)\circ \Delta_{a,n+1-a,j} &
\\
\qquad + \sum_{p=2}^{n+1}\sum_{h=1}^{p}\sum_{b=1}^{n+2-p}\sum_{j=h+b}^{b+p-1}(-1)^{\ga(n,b+p-1,j,p,h)} (\Delta_{p,b,h}\otimes \id)\circ \Delta_{b+p-1,n+2-b-p,j}.
\tag{$\dagger$}
\end{align*}
Observe that $h$ cannot run until $p$ in the third term (as one might expect) since $j$ runs from $h+b$ to $b+p-1$. Therefore, also $p$ starts at $2$, not at $1$.

Calling $a$ by $p$ and re-indexing $p \mapsto p+1$, we can reunite part of the second sum with the first one and consequently rewrite the coassociator \eqref{zero} as
\begin{align}
      \label{crepe1}
    \Sum_{p=2}^{n+2}\Sum_{b=0}^{n+2-p} \Sum_{h=2}^p\Sum_{j=1}^{h-1}
    (-1)^{\ga(n,b+p-1,j,p,h)} (\Delta_{p,b,h}\otimes \id)\circ \Delta_{b+p-1,n+2-b-p,j} 
    &
    \\
    \label{crepe2}
    + \Sum_{p=2}^{n+2} \Sum_{h=1}^{p-1} \Sum^{p-1}_{j=h} (-1)^{\ga(n,p-1,j,p,h)} (\Delta_{p,0,h}\otimes \id)\circ \Delta_{p-1,n+2-p,j}&
   \\ 
   \label{crepe10}
   \qquad + \sum_{p=2}^{n+1}\sum_{h=1}^{p}\sum_{b=1}^{n+2-p}\sum_{j=h+b}^{b+p-1}(-1)^{\ga(n,b+p-1,j,p,h)} (\Delta_{p,b,h}\otimes \id)\circ \Delta_{b+p-1,n+2-b-p,j}.
\end{align}
We need to prove that this is symmetric in the last two tensor factors. By applying the relations \eqref{orzo} once again, we obtain for the expression \eqref{crepe1}:
\begin{align*}
& \Sum_{p=2}^{n+2}\Sum_{h=2}^p\Sum_{b=0}^{n+2-p}\Sum_{j=1}^{h-1}
(-1)^{\ga(n,b+p-1,j,p,h)} (\id\otimes \sigma)\circ (\Delta_{p,b,h}\otimes \id)\circ \Delta_{b+p-1,n+2-b-p,j} 
\\
& = 
\Sum_{p=2}^{n+2}\Sum_{h=2}^p\Sum_{b=0}^{n+2-p}\Sum_{j=1}^{h-1}
(-1)^{\ga(n, b+p-1,j,p,h)} (\Delta_{p,b,j}\otimes \id)\circ \Delta_{b+p-1,n+2-b-p,h+b-1}
\\
& = 
\Sum_{p=2}^{n+2}\Sum_{j=2}^p\Sum_{b=0}^{n+2-p}\Sum_{h=1}^{j-1}
(-1)^{\ga(n, b+p-1,h,p,j)} (\Delta_{p,b,h}\otimes \id)\circ \Delta_{b+p-1,n+2-b-p,j+b-1},
\end{align*}
where we re-indexed $b \mapsto n+2-b-p$ and called $h$ by $j$ (and vice versa) in the second resp.\ third step. Re-indexing $j \mapsto j+b-1$ and exchanging sums, one then has:
\begin{align}
\nonumber
& \Sum_{p=2}^{n+2}\Sum_{b=0}^{n+2-p} \Sum_{h=1}^{p-1} \Sum^{p}_{j=h+1}
(-1)^{\ga(n, b+p-1,h,p,j-b+1)} (\Delta_{p,b,h}\otimes \id)\circ \Delta_{b+p-1,n+2-b-p,j+b-1} 
\\
\nonumber
& = 
\Sum_{p=2}^{n+2}\Sum_{b=0}^{n+2-p} \Sum_{h=1}^{p-1} \Sum^{p+b-1}_{j=h+b}
(-1)^{\ga(n, b+p-1,h,p,j-b+1)}  (\Delta_{p,b,h}\otimes \id)\circ \Delta_{b+p-1,n+2-b-p,j} 
\\
  \label{crepe3}
& = 
\Sum_{p=2}^{n+2}\Sum_{b=1}^{n+2-p} \Sum_{h=1}^{p-1} \Sum^{p+b-1}_{j=h+b}
(-1)^{\ga(n, b+p-1,h,p,j-b+1)}  (\Delta_{p,b,h}\otimes \id)\circ \Delta_{b+p-1,n+2-b-p,j} 
\\
  \label{crepe4}
& \qquad
+ \Sum_{p=2}^{n+2} \Sum_{h=1}^{p-1} \Sum^{p-1}_{j=h}
(-1)^{\ga(n, p-1,h,p,j+1)} (\Delta_{p,0,h}\otimes \id)\circ \Delta_{p-1,n+2-p,j},
\end{align}
and one sees that the second sum \eqref{crepe4} is precisely the second sum in \eqref{crepe2}. Observe that \eqref{crepe3} corresponds to the tensor flip of $\eqref{crepe1}'$, where by $ \eqref{crepe1}'$ we denote the sub-summand of \eqref{crepe1} where we sum on those $b$ for which $b\leq n+1-p$.

Finally, applying $\id \otimes \gs$ to \eqref{crepe2} yields, by the usual exchanging of variable names and re-indexing:
\begin{align}
\nonumber
     &  \Sum_{p=2}^{n+2} \Sum_{h=1}^{p-1} \Sum^{p-1}_{j=h} (-1)^{\ga(n,p-1,j,p,h)} (\id \otimes \gs) \circ (\Delta_{p,0,h}\otimes \id)\circ \Delta_{p-1,n+2-p,j}
      \\
            \label{crepe5}
      & =
      \Sum_{p=2}^{n+2} \Sum_{h=2}^{p} \Sum^{h-1}_{j=1} (-1)^{ \ga(n,p-1,h-1,p,j)} (\Delta_{p,n+2-p,h}\otimes \id)\circ \Delta_{n+1,0,j},
      \end{align}
which cancels with \eqref{crepe1} in case $b = n+2-p$.
We can now conclude by observing that $\eqref{crepe3}= \eqref{crepe10}$ and, consequently, that the tensor flip of \eqref{crepe10} equals $\eqref{crepe1}'$. This proves that \eqref{zero} is symmetric in the last two tensor factors, as wanted.
\end{proof}

In the next section, we will explain how to obtain a differential in order to describe the full structure of a dg Lie coalgebra induced by a cooperad.

\section{The chain complex structure}

In case the cooperad is a (counital) cooperad with comultiplication, one obtains a richer structure on it, the first one of interest being that of a chain complex or, to be precise, the structure of a simplicial $k$-module. Indeed, one has:

\begin{prop}
\label{sugo}
Let $(\cC, \mu^c, e^c)$ be a cooperad with comultiplication. Setting
 \begin{eqnarray}
   \label{riso1}
   d_0 & := & (\mu^c \otimes \id) \circ \gD_{2,n-1,2},
   \\
   \label{riso2}
d_i & := & (\id \otimes \mu^c) \circ \gD_{n-1,2,i}, \qquad 1 \leq i \leq n-1,
\\
\label{riso3}
d_n & := & (\mu^c \otimes \id) \circ \gD_{2,n-1,1},
\\
\label{pasta}
s_i& := & (\id\otimes e^c)\circ \Delta_{n+1,0,i+1}, \quad 0 \leq i \leq n,
  \end{eqnarray} 
the triple $(\cC, d_\bull, s_\bull)$ defines a simplicial
  object in $\kmod$. In particular,  
  the map $d\colon \cc \to \cc[-1]$ given by the alternating sum
  $d = \sum^n_{i=0} (-1)^i d_i$ is a differential, that is, $d^2=0$.
\end{prop}

\begin{proof}
We need to check the cosimplicial identities, namely:
\begin{eqnarray}
  \label{cosimp1}
d_id_j &=& d_{j-1}d_i  \text{ for $i<j$},
\\
  \label{cosimp2}
s_is_j &=& s_js_{i-1} \text{ for $i>j$}, 
\\
  \label{cosimp3}
  d_is_j &=&
\begin{cases}
s_{j-1}d_i &  \text{if} \ \ i <j,
  \\
 \id &  \text{if}\  i\in \{j,j+1\},
 \\
 s_j d_{i-1} & \text{if} \ i>j+1.
\end{cases}
\end{eqnarray}
Regarding the relation \eqref{cosimp1} for the faces, one repeatedly uses coassociativity of the cocomposition in $\cc$ and the comultiplication $\mu^c$, and also the fact that $ (\id\otimes \mu^c\otimes \mu^c)\circ (\id\otimes \sigma)= (\id\otimes \mu^c\otimes \mu^c).$ Moreover, one needs to distinguish between external faces and internal ones. As an example, let us show two simplicial identities of the type in \eqref{cosimp1}.
%
\item
  For $j=1,i=0$, we have: 
\begin{eqnarray*}
  d_0d_0
 &      \stackrel{\scriptscriptstyle{\eqref{orzo}}}{=} &
  (\mu^c\otimes\mu^c\otimes \id)\circ(\id\otimes \Delta_{2,n-2,2})\circ\Delta_{2,n-1,2} \\ 
  &      \stackrel{\scriptscriptstyle{\eqref{polenta}}}{=} &
  (\mu^c\otimes\mu^c\otimes \id)\circ(\Delta_{2,2,2}\otimes \id)\circ\Delta_{3,n-2,3} \\
&      \stackrel{\scriptscriptstyle{\eqref{orzo}}}{=} &
  (\mu^c\otimes\mu^c\otimes \id)\circ(\Delta_{2,2,1}\otimes \id)\circ\Delta_{3,n-2,3} \\
  &      \stackrel{\scriptscriptstyle{}}{=} &
(\mu^c\otimes \id\otimes  \mu^c)\circ (\Delta_{2,n-2,2}\otimes \id)\circ \Delta_{n-1,2,1} \, = \,  d_0d_1. 
\end{eqnarray*}



\item For $1 \leq i \leq n-1$ and $2 \leq j\leq n-1$, as well as $i<j-1$, one computes
\begin{eqnarray*}
  d_id_j
   &      \stackrel{\scriptscriptstyle{\eqref{orzo}}}{=} &
  (\id\otimes \mu^c\otimes \mu^c)\circ (\Delta_{n-2,2,i}\otimes \id)\circ \Delta_{n-1,2,j} \\
   &      \stackrel{\scriptscriptstyle{}}{=} &
  (\id\otimes \mu^c\otimes \mu^c)\circ (\id\otimes \rho)\circ(\Delta_{n-2,2,j-1}\otimes\id)\circ\Delta_{n-1,2,i}
  \, = \, d_{j-1}d_i. 
\end{eqnarray*}
%
%
%
Identities of type \eqref{cosimp2} for the degeneracies also directly follow from the coassociativity relations \eqref{orzo}, 
\begin{eqnarray*}
s_is_j      
 &      \stackrel{\scriptscriptstyle{\eqref{pasta}}}{=} 
 & (\id\otimes e^c\otimes e^c)\circ (\Delta_{n+2,0,i+1}\otimes \id) \circ \Delta_{n+1,0,j+1} \\
&     \stackrel{\scriptscriptstyle{}}{=} 
 & (\id\otimes e^c\otimes e^c)\circ (\id\otimes \sigma) \circ (\Delta_{n+2,0,j+1}\otimes \id)\circ \Delta_{n+1,0,i} \, = \,  s_js_{i-1},
\end{eqnarray*}
%
as do 
the first and third type of identities in \eqref{cosimp3}: for example, if $0<i<n$ and $j>i$, we have:
\begin{eqnarray*}
  d_is_j
 &      \stackrel{\scriptscriptstyle{\eqref{orzo2}}}{=} &
  (\id \otimes \mu^c \otimes e^c)\circ (\Delta_{n,2,i}\otimes \id)\circ \Delta_{n+1,0,j+1} \\
  &      \stackrel{\scriptscriptstyle{}}{=} &
  (\id\otimes e^c\otimes \mu^c)\circ (\Delta_{n,0,j}\otimes \id)\circ \Delta_{n-1,2,i} \, = \,  s_{j-1}d_i.
\end{eqnarray*}
%
%
%
%
%
Finally, for the second type of identities in \eqref{cosimp3}, we exploit coassociativity and counitality of $e^c$ in $\cc^*$. For example:
\begin{eqnarray*}
  d_0s_0  &      \stackrel{\scriptscriptstyle{\eqref{orzo2}}}{=} &
  (\mu^c\otimes \id \otimes e^c)\circ (\Delta_{2,n,2}\otimes \id)\circ \Delta_{n+1,0,1}
  \\
   &      \stackrel{\scriptscriptstyle{}}{=} &
  (\mu^c\otimes e^c \otimes \id)\circ (\Delta_{2,0,1} \otimes \id)\circ \Delta_{1,n,1} \\
   &      \stackrel{\scriptscriptstyle{\eqref{polentaunita}}}{=} &
  ((\mu^c\circ_1 e^c)\otimes \id )\circ \Delta_{1,n,1}  \\
 &      \stackrel{\scriptscriptstyle{}}{=} &
  (\mathbbm{1}^c\otimes \id ) \circ \Delta_{1,n,1}= \id. 
\end{eqnarray*}
%
All other cases are similar.
We have therefore verified that the data of $(\cC_\bull, d_\bull,s_\bull)$ defines a simplicial object in $\kmod$, as wanted.
The statement about the map $d$ then follows from these simplicial identities in a standard way.
\end{proof}

   As a consequence, and dual to the case of multiplicative operads, the comultiplicative structure on a cooperad induces a homology theory, and we can define:

\begin{deff}
Let $\cC$ be a cooperad with comultiplication. The homology groups
  $$
  H_\bull(\cC) := H\big(\cC_\bull, d\big),
  $$
  where the differential $d$ arises as in Proposition \ref{sugo}, 
  will be termed the homology of $\cC$ {\em induced by its comultiplication}. 
\end{deff}

\quad
\
\quad

\begin{example}
Visually, one can represent the outer face $d_0$ on $\cC(4)$ via the following tree:

\vspace*{-.1cm}

  \begin{center}
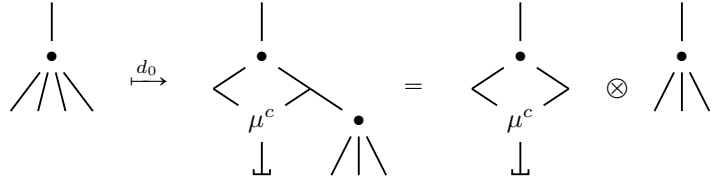

\begin{tikzpicture}[scale=.85, yscale=1, line width=0.7pt] 


\node(Z) at (-7.75,-1){};
\node(z) at (-7.75,-2){$\bullet$};
\node(1) at (-8.50,-3){};
\node(2) at (-8,-3){};
\node(4) at (-7.5,-3) {};
\node(5) at (-7,-3){};

\draw(Z)--(z);
\draw(z)--(1);
\draw(z)--(2);
\draw(z)--(4);
\draw(z)--(5);
 
\node(Freccia) at (-6.25,-2.25){$\xmapsto{ d_0 }$};

\node(A) at (-4.5, -1){};
\node(a) at (-4.5, -2){$\bullet$};
\node(6) at (-5.25,-2.5){};
\node(7) at (-3.75,-2.5){};
\node(7a) at (-3,-3){$\bullet$};
\node(c) at (-4.5,-3){$\mu^c$};
 \node(8) at (-4.5,-4){};
 \node(9) at (-3,-4){};
 \node(10) at (-3.5,-4){};
 \node(11) at (-2.5,-4){};

 \draw(A)--(a);
\draw(a)--(7.center);
\draw(a)--(6.center);
\draw(6.center)--(c);
 \draw(7.center)--(c);
 \draw[[-](8)--(c);
 \draw(7a)--(9);
 \draw(7a)--(10);
 \draw(7a)--(11);
 \draw(7a)--(7.center);

\node(Uguale) at (-2.15,-2.5){$=$};

\node(A) at (-.5, -1){};
\node(a) at (-.5, -2){$\bullet$};
\node(6) at (-1.25,-2.5){};
\node(7) at (0.25,-2.5){};
\node(c) at (-.5,-3){$\mu^c$};
\node(8) at (-.5,-4){};

\draw(A)--(a);
\draw(a)--(7.center);
\draw(a)--(6.center);
\draw(6.center)--(c);
\draw(7.center)--(c);
\draw[[-](8)--(c);

\node(Tensore) at (1,-2.5){$\otimes$};

\node(B) at (2,-1){};
\node(c) at (2,-2){$\bullet$};
\node(7) at (2,-3){};
\node(8) at (2.5,-3){};
\node(9) at (1.50,-3){};

\draw(B)--(c);
\draw(c)--(7);
\draw(c)--(8);
\draw(c)--(9);

 \end{tikzpicture}

\vspace*{-.2cm}

\captionof{figure}{The outer face $d_0$ on $\cC(4)$.}
\label{d0}

  \end{center}
 
\noindent Similar graphical representations can be drawn for the remaining faces on $\cC(4)$, and hence the differential $d$ can be pictured as an alternate sum of trees multiplied by a scalar:

\vspace*{-.1cm}

  \begin{center}
\begin{tikzpicture}[scale=.85, yscale=1, line width=0.7pt] 


\node(Z) at (-7.75,-1){};
\node(z) at (-7.75,-2){$\bullet$};
\node(1) at (-8.50,-3){};
\node(2) at (-8,-3){};
\node(4) at (-7.5,-3) {};
\node(5) at (-7,-3){};

\draw(Z)--(z);
\draw(z)--(1);
\draw(z)--(2);
\draw(z)--(4);
\draw(z)--(5);
 
\node(Freccia) at (-6.25,-2.25){$\xmapsto{\, d \, }$};

\node(A) at (-4.5, -1){};
\node(a) at (-4.5, -2){$\bullet$};
\node(6) at (-5.25,-2.5){};
\node(7) at (-3.75,-2.5){};
\node(7a) at (-3,-3){$\bullet$};
\node(c) at (-4.5,-3){$\mu^c$};
 \node(8) at (-4.5,-6){};
 \node(9) at (-3,-3.8){};
 \node(10) at (-3.5,-3.8){};
 \node(11) at (-2.5,-3.8){};

 \draw(A)--(a);
\draw(a)--(7.center);
\draw(a)--(6.center);
\draw(6.center)--(c);
 \draw(7.center)--(c);
 \draw[[-](8)--(c);
 \draw(7a)--(9);
 \draw(7a)--(10);
 \draw(7a)--(11);
 \draw(7a)--(7.center);

\node(meno1) at (-2.25,-2.5){$-$};

\node(Z1) at (-.75,-1){};
\node(z1) at (-.75,-2){$\bullet$};
\node(11) at (-1.50,-3.8){$\bullet$};
\node(21) at (-.75,-3.8){};
\node(41) at (0,-3.8) {};
\node(51) at (-2,-4.35){};
\node(61) at (-1,-4.35){};
\node(71) at (-1.5,-5){$\mu^c$};
\node(81) at (-1.5,-6){};

\draw(Z1)--(z1);
\draw(z1)--(11);
\draw(z1)--(21);
\draw(z1)--(41);
\draw(51.center)--(11);
\draw(61.center)--(11);
\draw(51.center)--(71);
\draw(61.center)--(71);
\draw[[-](81)--(71);

\node(piu1) at (0.25,-2.5){$+$};

\node(Z2) at (1.25,-1){};
\node(z2) at (1.25,-2){$\bullet$};
\node(12) at (1.25,-3.8){$\bullet$};
\node(22) at (.75,-3.8){};
\node(42) at (1.75,-3.8) {};

\node(52) at (.75,-4.35){};
\node(62) at (1.75,-4.35){};
\node(72) at (1.25,-5){$\mu^c$};
\node(82) at (1.25,-6){};

\draw(Z2)--(z2);
\draw(z2)--(12);
\draw(z2)--(22);
\draw(z2)--(42);

\draw(52.center)--(12);
\draw(62.center)--(12);
\draw(52.center)--(72);
\draw(62.center)--(72);
\draw[[-](82)--(72);

\node(meno2) at (2.25,-2.5){$-$};
\node(Z1) at (3.25,-1){};
\node(z1) at (3.25,-2){$\bullet$};
\node(11) at (4,-3.8){$\bullet$};
\node(21) at (3.25,-3.8){};
\node(41) at (2.5,-3.8) {};
\node(51) at (4.5,-4.35){};
\node(61) at (3.5,-4.35){};
\node(71) at (4,-5){$\mu^c$};
\node(81) at (4,-6){};

\draw(Z1)--(z1);
\draw(z1)--(11);
\draw(z1)--(21);
\draw(z1)--(41);
\draw(51.center)--(11);
\draw(61.center)--(11);
\draw(51.center)--(71);
\draw(61.center)--(71);
\draw[[-](81)--(71);

\begin{scope}[yscale=1,xscale=-1]

  \node(piu2) at (-4.9,-2.5){$+$};
  
\node(A) at (-7.25, -1){};
\node(a) at (-7.25, -2){$\bullet$};
\node(6) at (-8,-2.5){};
\node(7) at (-6.5,-2.5){};
\node(7a) at (-5.75,-3){$\bullet$};
\node(c) at (-7.25,-3){$\mu^c$};
 \node(8) at (-7.25,-6){};
 \node(9) at (-5.75,-3.8){};
 \node(10) at (-6.25,-3.8){};
 \node(11) at (-5.25,-3.8){};

 \draw(A)--(a);
\draw(a)--(7.center);
\draw(a)--(6.center);
\draw(6.center)--(c);
 \draw(7.center)--(c);
 \draw[[-](8)--(c);
 \draw(7a)--(9);
 \draw(7a)--(10);
 \draw(7a)--(11);
 \draw(7a)--(7.center);

 \end{scope}

\end{tikzpicture}

\vspace*{-.2cm}

\captionof{figure}{The differential $d$ on $\cC(4)$.}
\label{d}

 \end{center}

  \noindent
At first sight, one might be tempted to think that in Figure \ref{d} the second to fourth summand seem to produce trees in $\cC(2)$ instead of $\cC(3)$, but this is an illusion: the element $\mu^c$ kills the lower vertex resp.\ the lower tree but not the respective first, second, or third branch of the respective upper tree. 

\end{example}

\begin{rmk}
  We can express $d$ in terms of the cobracket from \eqref{arancia}, that is,
  \begin{equation}
    \label{olio}
    d \coloneqq 
    \overline{\mu}^c\circ \{-\}_n\colon \, \cc(n) \rightarrow \cc(n-1),
  \end{equation}
  where $\overline{\mu}^c$ is the degree $-2$ morphism $\overline{\mu}^c\colon \cc\otimes \cc\to \cc$ defined as 
\begin{equation}
  \label{pepe}
\overline{\mu}^c (x\otimes y) = \begin{cases}
 \mu^c(x)y & \text{if} \ \text{deg}(x)=2, \\
0 & \text{otherwise.}
\end{cases}
  \end{equation}
  %
  Combining Eqs.~\eqref{sale} and \eqref{pepe}, that is, the fact that expressions not in degree $2$ vanish, this can be explicitly rewritten as:
  \begin{equation}
    \label{olio1}
    d = (\mu^c \otimes \id) \circ \big(\gD_{2,n-1,2} + (-1)^n \gD_{2,n-1,1}\big)
    +
    \Sum^{n-1}_{i=1} (-1)^{i} (\id \otimes \mu^c) \circ \gD_{n-1,2,i}.
  \end{equation}
\end{rmk}
\noindent As a consequence, we obtain the following result:

\begin{coro}
  \label{sushi}
  For any comultiplicative cooperad $\cc$, its suspended complex $\big(\cc_\bullet[1], d, \{-\}\big)$ is a dg Lie coalgebra, where $\{-\}$ is given as in \eqref{arancia}. 
\end{coro}

\begin{proof}
  Observe first that suspending from $\cC$ to $\cC[1]$ transforms the flip $\tau$ into $\rho$,
  which explains the use of \eqref{arancia} instead of \eqref{arancia2} for the cobracket here, which is a question of signs.
  If we check  that the differential $d$ is a coderivation with respect to the cobracket $\{-\} = (\id - \rho) \circ [-]$, then we can directly conclude thanks to Propositions \ref{pesce} \& \ref{lunapop}.
One needs to check that 
\begin{equation*}
\{-\}\circ d = \big(d\otimes \id + \rho \circ (d \otimes \id) \circ \rho \big) \circ \{-\}.
\end{equation*}
on the suspended complex.
This follows from the coJacobi identity and the fact that
the cobracket is the coantisymmetrisation of the cobrace so that, in particular, $\rho \circ \{-\}= - \{-\}$.
Indeed, the left hand side can be written as:
\begin{equation*}
\{-\}\circ d  = (\mu^c\otimes \id\otimes \id)\circ \big(\id\otimes \{-\} \big)\circ \{-\}.
\end{equation*}
By using Eq.~\eqref{sinistradestra}, the first addendum on the right hand side writes as 
\begin{align*}
  (d\otimes \id)\circ \{-\} &= (\mu^c\otimes  \id^{\otimes 2})
  \circ (\{-\} \otimes \id ) \circ \{-\} \\
  &= (\mu^c\otimes  \id^{\otimes 2} )
\circ(\id\otimes\rho)\circ(\rho\otimes\id)\circ(\id\otimes\{-\})\circ\rho\circ\{-\}\\
  &=  - (\mu^c\otimes  \id^{\otimes 2})
  \circ(\id\otimes\rho)\circ(\rho\otimes\id)\circ(\id\otimes\{-\})
  \circ\{-\}.
\end{align*}
Similarly, the second addendum on the right can also be rewritten to give the third term of the coJacobi identity, {\em i.e.}:
\begin{equation*}
  \begin{split}
  \rho \circ (d\otimes \id) \circ \rho \circ \{-\} &
  = \rho \circ (\mu^c\otimes \id^{\otimes 2})\circ (\{-\} \otimes \id)\circ \rho \circ \{-\}
  \\
  & =
 (\mu^c\otimes \id^{\otimes 2})\circ (\id \otimes \rho) \circ (\{-\} \otimes \id)\circ \rho \circ \{-\}
\\
  & =
  (\mu^c\otimes \id^{\otimes 2})\circ (\rho \otimes \id) \circ (\id \otimes \{-\}) \circ \{-\}
  \\
  & =
  - (\mu^c\otimes \id^{\otimes 2} )\circ (\rho\otimes \id)\circ (\id\otimes \rho)\circ (\id\otimes \{-\})\circ \{-\},
%
%
\end{split}
  \end{equation*}
where we used \eqref{sinistradestra} in line three.
Thus, we have proven that the differential is a coderivation with respect to the cobracket, as claimed.
  \end{proof}

\section{The cup coproduct}

\reversemarginpar
The customary cup product arising in the operadic setting in presence of a multiplication, and which is shown to be equivalent to the Yoneda product in the corresponding context, can be dualised in the following way:

\begin{deff}
  Let $\cC$ be a cooperad with comultiplication.
  The map 
  $$
  \cup^c \colon
  \cC(n) \to \bigoplus_{j=1}^{n+1} \cC(j-1) \otimes \cC(n+1-j)
  $$
  given by
\begin{equation}
  \label{aceto}
\begin{split}
  \cup^c x &:=
\Sum^{n+1}_{j=1} (\mu^c \otimes \id^{\otimes 2}) \circ (\Delta_{2, j-1, 1} \otimes \id) \circ \Delta_{j, n+1-j, j} (x)
\end{split}
\end{equation}
for $x \in \cC(n)$ will be called the
{\em cup coproduct}.
\end{deff}

\begin{example}
  In the spirit of the figures above, the cup coproduct of an element $\cc(3)$ can be represented as follows:

\vspace*{-.1cm}
  
  \begin{center}
\begin{tikzpicture}[scale=.85, xscale=.85, line width=0.7pt] 

\node(Z) at (-7.5,.5){};
\node(z) at (-7.5,-.5){$\bullet$};
\node(a) at (-7.5,-2){};
\node(b) at (-8,-2){};
\node(c) at (-7,-2) {};

\draw(Z)--(z);
\draw(z)--(a.center);
\draw(z)--(b.center);
\draw(z)--(c.center);

  \node(Freccia) at (-6,-1){$\xmapsto{\, \cup^c}$};

\node(1) at (-4,1){};
\node(2) at (-4,0){$\bullet$};
\node(3) at (-4.5,-.5){};
\node(4) at (-3.5,-.5){};
\node(5) at (-5,-1) {};
\node(5a) at (-5,-1.5) {$\bullet$};
\node(6) at (-4,-1){$\mu^c$};
\node(7) at (-3,-1){};
\node(8) at (-3,-1.5){$\bullet$};
\node(10) at (-4,-2.25) {};
\node(11) at (-3.5,-2.25) {};
\node(12) at (-3,-2.25) {};
\node(13) at (-2.5,-2.25) {};

\draw(1)--(2);
\draw(2)--(3.center);
\draw(2)--(4.center);
\draw(3.center)--(5.center);
\draw(3.center)--(6);
\draw(4.center)--(6);
\draw(4.center)--(7.center);
\draw(8)--(7.center);
\draw(5a)--(5.center);
  \draw[[-](10.center)--(6);
    \draw(8)--(11.center);
    \draw(8)--(12.center);
    \draw(8)--(13.center);

    \node(piu1) at (-2,-1){$+$};

\node(1) at (-0,1){};
\node(2) at (-0,0){$\bullet$};
\node(3) at (-.5,-.5){};
\node(4) at (.5,-.5){};
\node(5) at (-1,-1) {};
\node(5a) at (-1,-1.5) {$\bullet$};
\node(6) at (-0,-1){$\mu^c$};
\node(7) at (1,-1){};
\node(8) at (1,-1.5){$\bullet$};
\node(9) at (-1,-2.25){};
\node(10) at (0,-2.25) {};
\node(11) at (.75,-2.25) {};
\node(13) at (1.25,-2.25) {};

\draw(1)--(2);
\draw(2)--(3.center);
\draw(2)--(4.center);
\draw(3.center)--(5.center);
\draw(3.center)--(6);
\draw(4.center)--(6);
\draw(4.center)--(7.center);
\draw(8)--(7.center);
\draw(9.center)--(5a);
\draw(5a)--(5.center);
\draw[[-](10.center)--(6);
    \draw(8)--(11.center);
    \draw(8)--(13.center);

    \node(piu2) at (2,-1){$+$};

\begin{scope}[yscale=1,xscale=-1]

\node(1) at (-8,1){};
\node(2) at (-8,0){$\bullet$};
\node(3) at (-8.5,-.5){};
\node(4) at (-7.5,-.5){};
\node(5) at (-9,-1) {};
\node(5a) at (-9,-1.5) {$\bullet$};
\node(6) at (-8,-1){$\mu^c$};
\node(7) at (-7,-1){};
\node(8) at (-7,-1.5){$\bullet$};
\node(9) at (-9,-2.25){};
\node(10) at (-8,-2.25) {};
\node(11) at (-7.5,-2.25) {};
\node(12) at (-7,-2.25) {};
\node(13) at (-6.5,-2.25) {};

\draw(1)--(2);
\draw(2)--(3.center);
\draw(2)--(4.center);
\draw(3.center)--(5.center);
\draw(3.center)--(6);
\draw(4.center)--(6);
\draw(4.center)--(7.center);
\draw(8)--(7.center);
\draw(5a)--(5.center);
  \draw[[-](10.center)--(6);
    \draw(8)--(11.center);
    \draw(8)--(12.center);
    \draw(8)--(13.center);

    \node(piu1) at (-6,-1){$+$};

\node(1) at (-4,1){};
\node(2) at (-4,0){$\bullet$};
\node(3) at (-4.5,-.5){};
\node(4) at (-3.5,-.5){};
\node(5) at (-5,-1) {};
\node(5a) at (-5,-1.5) {$\bullet$};
\node(6) at (-4,-1){$\mu^c$};
\node(7) at (-3,-1){};
\node(8) at (-3,-1.5){$\bullet$};
\node(9) at (-5,-2.25){};
\node(10) at (-4,-2.25) {};
\node(11) at (-3.25,-2.25) {};
\node(13) at (-2.75,-2.25) {};

\draw(1)--(2);
\draw(2)--(3.center);
\draw(2)--(4.center);
\draw(3.center)--(5.center);
\draw(3.center)--(6);
\draw(4.center)--(6);
\draw(4.center)--(7.center);
\draw(8)--(7.center);
\draw(9.center)--(5a);
\draw(5.center)--(5a);
  \draw[[-](10.center)--(6);
    \draw(8)--(11.center);
    \draw(8)--(13.center);

  \end{scope}

\end{tikzpicture}

\vspace*{-.2cm}


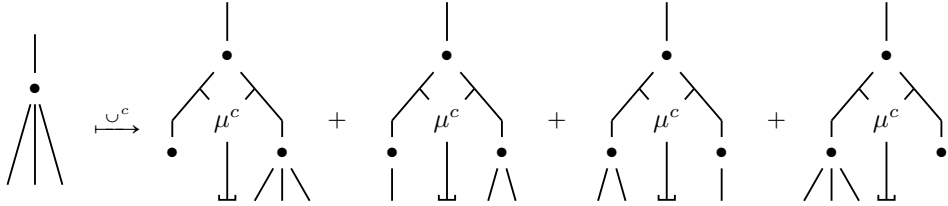
\captionof{figure}{The cup coproduct in terms of trees.}
\label{cup}

\end{center}

\noindent By writing the degrafting operation as a tensor product of trees, this can alternatively be depicted as:

\vspace*{-.1cm}

  \begin{center}
\begin{tikzpicture}[scale=.85, yscale=.9, line width=0.7pt] 

\node(Z) at (-7.5,-.5){};
\node(z) at (-7.5,-1.5){$\bullet$};
\node(a) at (-7.5,-2.75){};
\node(b) at (-8,-2.75){};
\node(c) at (-7,-2.75) {};

\draw(Z.center)--(z);
\draw(z)--(a.center);
\draw(z)--(b.center);
\draw(z)--(c.center);

  \node(Freccia) at (-6,-1.5){$\xmapsto{\, \cup^c}$};

\node(1) at (-5,-.5){};
\node(1a) at (-5, -2){$\bullet$};
\node(2) at (-3.75,-.5){};
\node(3) at (-2.5,-.5){};
\node(4) at (-4.5,-1.5){$\otimes$};
\node(4a) at (-3.75,-1){$\bullet$};
\node(5) at (-3,-1.5) {$\otimes$};
\node(5a) at (-2.5,-2){$\bullet$};
\node(6) at (-5,-2.75) {};
\node(7) at (-4.125,-1.5) {};
\node(8) at (-3.375,-1.5) {};
\node(9) at (-3,-2.75) {};
\node(10) at (-2.5,-2.75) {};
\node(11) at (-2,-2.75) {};
\node(12) at (-3.75,-2) {$\mu^c$};
\node(13) at (-3.75,-2.75) {};

\draw(1a)--(1.center);
  \draw(4a)--(2.center);
  \draw(4a)--(7.center);
   \draw(4a)--(8.center);
  \draw(12)--(7.center);
  \draw(12)--(8.center);
    \draw[[-](13.center)--(12);
      \draw(3.center)--(5a);
      \draw(9.center)--(5a);
      \draw(10.center)--(5a);
      \draw(11.center)--(5a);

    \node(piu1) at (-1.75,-1.5){$+$};

    \node(1) at (-1.25,-.5){};
    \node(1a) at (-1.25,-2){$\bullet$};
\node(2) at (0,-.5){};
\node(3) at (1.25,-.5){};
\node(4) at (-.75,-1.5){$\otimes$};
\node(4a) at (0,-1){$\bullet$};
\node(5) at (.75,-1.5) {$\otimes$};
\node(5a) at (1.25,-2){$\bullet$};
\node(6) at (-1.25,-2.75) {};
\node(7) at (-.375,-1.5) {};
\node(8) at (.375,-1.5) {};
\node(9) at (1,-2.75) {};
\node(11) at (1.5,-2.75) {};
\node(12) at (0,-2) {$\mu^c$};
\node(13) at (0,-2.75) {};

\draw(1a)--(1.center);
\draw(1a)--(6.center);
  \draw(4a)--(2.center);
  \draw(4a)--(7.center);
   \draw(4a)--(8.center);
  \draw(12)--(7.center);
  \draw(12)--(8.center);
    \draw[[-](13.center)--(12);
      \draw(3.center)--(5a);
      \draw(9.center)--(5a);
      \draw(11.center)--(5a);

    \node(piu2) at (1.875,-1.5){$+$};

\begin{scope}[yscale=1,xscale=-1]

  \node(1) at (-9,-.5){};
  \node(1a) at (-9, -2){$\bullet$};
\node(2) at (-7.75,-.5){};
\node(3) at (-6.5,-.5){};
\node(4) at (-8.5,-1.5){$\otimes$};
\node(4a) at (-7.75,-1){$\bullet$};
\node(5) at (-7,-1.5) {$\otimes$};
\node(5a) at (-6.5,-2){$\bullet$};
\node(6) at (-9,-2.75) {};
\node(7) at (-8.125,-1.5) {};
\node(8) at (-7.375,-1.5) {};
\node(9) at (-7,-2.75) {};
\node(10) at (-6.5,-2.75) {};
\node(11) at (-6,-2.75) {};
\node(12) at (-7.75,-2) {$\mu^c$};
\node(13) at (-7.75,-2.75) {};

\draw(1a)--(1.center);
  \draw(4a)--(2.center);
  \draw(4a)--(7.center);
   \draw(4a)--(8.center);
  \draw(12)--(7.center);
  \draw(12)--(8.center);
    \draw[[-](13.center)--(12);
      \draw(3.center)--(5a);
      \draw(9.center)--(5a);
      \draw(10.center)--(5a);
      \draw(11.center)--(5a);
 
\node(1) at (-5,-.5){};
\node(1a) at (-5,-2){$\bullet$};
\node(2) at (-3.75,-.5){};
\node(3) at (-2.5,-.5){};
\node(4) at (-3,-1.5){$\otimes$};
\node(4a) at (-3.75,-1){$\bullet$};
\node(5) at (-4.5,-1.5) {$\otimes$};
\node(5a) at (-2.5,-2){$\bullet$};
\node(6) at (-5,-2.75) {};
\node(7) at (-3.375,-1.5) {};
\node(8) at (-4.125,-1.5) {};
\node(9) at (-2.25,-2.75) {};
\node(11) at (-2.75,-2.75) {};
\node(12) at (-3.75,-2) {$\mu^c$};
\node(13) at (-3.75,-2.75) {};

\draw(1a)--(1.center);
\draw(6.center)--(1a);
  \draw(4a)--(2.center);
  \draw(4a)--(7.center);
   \draw(4a)--(8.center);
  \draw(12)--(7.center);
  \draw(12)--(8.center);
    \draw[[-](13.center)--(12);
      \draw(3.center)--(5a);
      \draw(9.center)--(5a);
      \draw(11.center)--(5a);

    \node(piu1) at (-5.625,-1.5){$+$};

     \end{scope}

 \end{tikzpicture}

\vspace*{-.2cm}


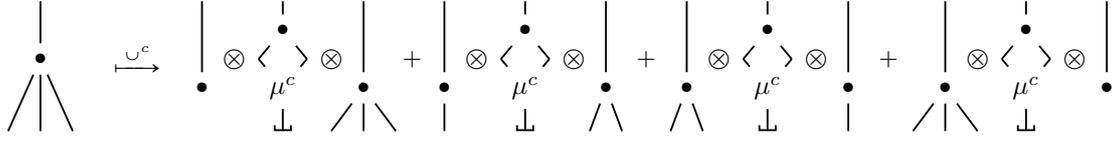
\captionof{figure}{The cup coproduct in terms of trees and tensor products.}
\label{cuptens}

\end{center}

\end{example}

\noindent Adding the
definition of the map
$ \overline{e}^c\colon \cC\to k $ by 
\begin{equation}
  \label{panforte}
\overline{e}^c(x)=
\begin{cases}
  e^c(x) & \text{if $x\in \cC(0)$,}
  \\
0 & \text{otherwise,}
\end{cases}
\end{equation}
for any cooperad with comultiplication, we can state:

\begin{prop}
  \label{zenzero}
  Let $(\cc,\psi)$ be a cooperad with comultiplication. Then
  $(\cc, \cup^c, d, \overline{e}^c)$ is
  a dg coassociative counital coalgebra.
\end{prop}

\begin{proof}
  We need to show that $\cup^c$ is coassociative, {\em i.e.},
  \begin{equation}
    \label{fagioli}
(\cup^c \otimes \id) \circ \cup^c = (\id \otimes \cup^c) \circ \cup^c, 
  \end{equation}
  and that $d$ is a coderivation of $\cup^c$, that is,
  \begin{equation}
    \label{ceci}
\cup^c \circ \, d = (d \otimes \id) \circ \cup^c  + \tau \circ (d \otimes \id) \circ \tau \circ \cup^c,
  \end{equation}
  where
  $\tau$ is the graded flip as in \eqref{piadina}. Observe again that, by applying the Koszul sign rule, this can be equally rewritten as 
  \begin{equation}
    \label{ceci2}
\cup^c \circ \, d = (d \otimes \id + \id \otimes d) \circ \cup^c,
  \end{equation}
which, as said, in some situations might be more convenient (but not in this proof).
  
 
  Proving Eq.~\eqref{fagioli} essentially hinges on reordering the partial decompositions $\gD_{p,q,i}$ first, so as to have a decreasing chain of decompositions appearing as leftmost as possible, applying afterwards the cooperadic identities \eqref{orzo} step-by-step from right to left, using then the associativity \eqref{polenta} for the multiplication element $\mu^c \in \cC^*(2)$, and finally applying the cooperadic identities \eqref{orzo} again from left to right. Indeed, on $\cC(n)$ one computes:
    \begin{equation*}
    \begin{array}{rcl}
      &&    
     (\id \otimes \cup^c) \circ \cup^c
      \\[1mm]
      &
      \stackrel{\scriptscriptstyle{\eqref{aceto}}}{=}
      &
      \Sum^{n+1}_{j=1} \Sum^{n+2-j}_{i=1} (\mu^c \otimes \id \otimes \mu^c \otimes \id^{\otimes 2}) \circ (\id^{\otimes 2} \otimes \gD_{2,i-1,1} \otimes \id)
      \circ (\id^{\otimes 2} \otimes \gD_{i,n+2-j-i,i})
 \\[1mm]
& &
 \hspace*{7cm}
      \circ (\gD_{2,j-1,1} \otimes \id) \circ \gD_{j,n+1-j,j}
      \\[1mm]
      &
      \stackrel{\scriptscriptstyle{}}{=}
      &
      \Sum^{n+1}_{j=1} \Sum^{n+2-j}_{i=1} (\mu^c \otimes \id \otimes \mu^c \otimes \id^{\otimes 2}) \circ (\gD_{2,j-1,1} \otimes \id^{\otimes 3})
      \circ (\id \otimes \gD_{2,i-1,1} \otimes \id)
 \\[1mm]
& &
 \hspace*{7cm}
      \circ (\id \otimes \gD_{i,n+2-j-i,i}) \circ \gD_{j,n+1-j,j}
      \\[1mm]
      &
      \stackrel{\scriptscriptstyle{\eqref{orzo}}}{=}
      &
            \Sum^{n+1}_{j=1} \Sum^{n+2-j}_{i=1} (\mu^c \otimes \id \otimes \mu^c \otimes \id^{\otimes 2}) \circ (\gD_{2,j-1,1} \otimes \id^{\otimes 3})
      \circ (\id \otimes \gD_{2,i-1,1} \otimes \id)
 \\[1mm]
& &
 \hspace*{7cm}
      \circ (\gD_{j,i,j} \otimes \id) \circ \gD_{j+i-1,n+2-j-i,j+i-1}
      \\[1mm]
      &
      \stackrel{\scriptscriptstyle{\eqref{orzo}}}{=}
      &
            \Sum^{n+1}_{j=1} \Sum^{n+2-j}_{i=1} (\mu^c \otimes \id \otimes \mu^c \otimes \id^{\otimes 2}) \circ (\gD_{2,j-1,1} \otimes \id^{\otimes 3})
      \circ (\gD_{j,2,j} \otimes \id^{\otimes 2})
 \\[1mm]
& &
 \hspace*{7cm}
      \circ (\gD_{j+1,i-1,j} \otimes \id) \circ \gD_{j+i-1,n+2-j-i,j+i-1}
      \\[1mm]
      &
      \stackrel{\scriptscriptstyle{\eqref{orzo}}}{=}
      &
      \Sum^{n+1}_{j=1} \Sum^{n+2-j}_{i=1} (\mu^c \otimes \id \otimes \mu^c \otimes \id^{\otimes 2})
\circ (\id \otimes \gs \otimes \id^{\otimes 2})  
      \circ (\gD_{2,2,2} \otimes \id^{\otimes 3})
      \circ (\gD_{3,j-1,1} \otimes \id^{\otimes 2})
 \\[1mm]
& &
 \hspace*{7cm}
      \circ (\gD_{j+1,i-1,j} \otimes \id) \circ \gD_{j+i-1,n+2-j-i,j+i-1}
                \end{array}
    \end{equation*}

    
       \begin{equation*}
    \begin{array}{rcl}
      &
      \stackrel{\scriptscriptstyle{\eqref{polenta}}}{=}
      &
      \Sum^{n+1}_{j=1} \Sum^{n+2-j}_{i=1} (\mu^c \otimes \mu^c \otimes \id^{\otimes 3})  
      \circ (\gD_{2,2,1} \otimes \id^{\otimes 3})
      \circ (\gD_{3,j-1,1} \otimes \id^{\otimes 2})
 \\[1mm]
& &
 \hspace*{7cm}
      \circ (\gD_{j+1,i-1,j} \otimes \id) \circ \gD_{j+i-1,n+2-j-i,j+i-1}
      \\[1mm]
      &
      \stackrel{\scriptscriptstyle{\eqref{orzo}}}{=}
      &
      \Sum^{n+1}_{j=1} \Sum^{n+2-j}_{i=1} (\mu^c \otimes \mu^c \otimes \id^{\otimes 3})  
      \circ (\id \otimes \gD_{2,j-1,1} \otimes \id^{\otimes 2})
      \circ (\gD_{2,j,1} \otimes \id^{\otimes 2})
 \\[1mm]
& &
 \hspace*{7cm}
      \circ (\gD_{j+1,i-1,j} \otimes \id) \circ \gD_{j+i-1,n+2-j-i,j+i-1}
      \\[1mm]
      &
      \stackrel{\scriptscriptstyle{\eqref{orzo}}}{=}
      &
       \Sum^{n+1}_{j=1} \Sum^{n+2-j}_{i=1} (\mu^c \otimes \mu^c \otimes \id^{\otimes 3})  
      \circ (\id \otimes \gD_{2,j-1,1} \otimes \id^{\otimes 2})
      \circ (\id \otimes \gD_{j,i-1,j} \otimes \id)
 \\[1mm]
& &
 \hspace*{7cm}
      \circ (\gD_{2,j+i-2,1} \otimes \id) \circ \gD_{j+i-1,n+2-j-i,j+i-1}
      \\[1mm]
      &
      \stackrel{\scriptscriptstyle{}}{=}
      &
       \Sum^{n+1}_{j=1} \, \Sum^{n+1}_{i=j} (\mu^c \otimes \mu^c \otimes \id^{\otimes 3})  
      \circ (\id \otimes \gD_{2,j-1,1} \otimes \id^{\otimes 2})
      \circ (\id \otimes \gD_{j,i-j,j} \otimes \id)
 \\[1mm]
& &
 \hspace*{7cm}
      \circ (\gD_{2,i-1,1} \otimes \id) \circ \gD_{i,n+1-i,i}
      \\[1mm]
      &
      \stackrel{\scriptscriptstyle{}}{=}
      &
       \Sum^{n+1}_{i=1} \, \Sum^{i}_{j=1} (\mu^c \otimes \mu^c \otimes \id^{\otimes 3})  
      \circ (\id \otimes \gD_{2,j-1,1} \otimes \id^{\otimes 2})
      \circ (\id \otimes \gD_{j,i-j,j} \otimes \id)
 \\[1mm]
& &
 \hspace*{7cm}
      \circ (\gD_{2,i-1,1} \otimes \id) \circ \gD_{i,n+1-i,i}
      \\[1mm]
      &
      \stackrel{\scriptscriptstyle{\eqref{aceto}}}{=}
      &
      (\cup^c \otimes \id) \circ \cup^c,
      \end{array}
   \end{equation*}
  where in the third line from below we used re-indexing $i \mapsto i-j+1$ and in the next step the usual double sum yoga.  

Next, let us check counitality, that is, 
\begin{equation*}
  (\overline{e}^c\otimes \id) \circ \cup^c
  = (\id\otimes \overline{e}^c) \circ \cup^c=\id,
\end{equation*}
which follows directly from counitality of $\mu^c$ with respect to $e^c$ given by the structure of a cooperad with comultiplication.
Indeed, one has:
\begin{eqnarray*}
(\overline{e}^c \otimes \id)\circ \cup^c &=& \Sum_{j=1}^{n+1}(\overline{\mu}^c\otimes \overline{e}^c\otimes \id)\circ (\Delta_{2,j-1,1}\otimes \id)\circ \Delta_{j,n+1-j,j} \\
& = & (\mu^c\otimes e^c\otimes \id)\circ(\Delta_{2,0,1}\otimes \id)\circ \Delta_{1,n,1} \\
&\stackrel{\scriptscriptstyle{\eqref{polentaunita}}}{=} & (\mathbbm{1}^c\otimes \id)\circ \Delta_{1,n,1} = \id.  
\end{eqnarray*}
Similarly, by making use of the coassociativity laws \eqref{orzo}, one computes
\begin{eqnarray*}
(\id\otimes \overline{e}^c)\circ \cup^c & = & (\mu^c\otimes \id \otimes e^c)\circ (\Delta_{2,n,1}\otimes \id)\circ \Delta_{n+1,0,n+1} \\
& \stackrel{\scriptscriptstyle{\eqref{orzo2}}}{=} & (\mu^c\otimes e^c \otimes )\circ (\Delta_{2,0,2}\otimes \id)\circ \Delta_{1,n,1} \\
& \stackrel{\scriptscriptstyle{\eqref{polentaunita}}}{=} & (\mathbbm{1}^c\otimes \id)\circ \Delta_{1,n,1}=\id.
\end{eqnarray*}
Counitality is therefore proven as well.

Finally,   
proving Eq.~\eqref{ceci} is done by decomposing the identity into simplicial pieces, using the face operators from Eqs.~\eqref{riso1}--\eqref{riso3}. 
To start with, for the zeroth face map on $\cC(n)$ one has:
   \begin{equation*}
    \begin{array}{rcl}
      &&    
      \cup^c \circ d_0
      \\[1mm]
      &
      \stackrel{\scriptscriptstyle{\eqref{riso1}, \eqref{aceto}}}{=}
      &
      \Sum^{n}_{j=1} (\mu^c \otimes \mu^c \otimes \id^{\otimes 2}) \circ (\id \otimes \gD_{2,j-1,1} \otimes \id) \circ ( \id \otimes \gD_{j,n-j,j}) \circ \gD_{2,n-1,2}
\\[1mm]
      &
      \stackrel{\scriptscriptstyle{\eqref{orzo}}}{=}
      &
      \Sum^{n}_{j=1} (\mu^c \otimes \mu^c \otimes \id^{\otimes 2}) \circ (\id \otimes \gD_{2,j-1,1} \otimes \id) \circ (\gD_{2,j,2} \otimes \id) \circ \gD_{j+1,n-j,j+1}
      \\[1mm]
      &
      \stackrel{\scriptscriptstyle{\eqref{orzo}}}{=}
      &
  \Sum^{n}_{j=1} (\mu^c \otimes \mu^c \otimes \id^{\otimes 2}) \circ (\gD_{2,2,2} \otimes \id^{\otimes 2}) \circ (\gD_{3,j-1,2} \otimes \id) \circ \gD_{j+1,n-j,j+1}
\\[1mm]
      &
      \stackrel{\scriptscriptstyle{\eqref{polenta}}}{=}
      &
  \Sum^{n}_{j=1} (\mu^c \otimes \mu^c \otimes \id^{\otimes 2}) \circ (\gD_{2,2,1} \otimes \id^{\otimes 2}) \circ (\gD_{3,j-1,2} \otimes \id) \circ \gD_{j+1,n-j,j+1}
  \\[1mm]
    &
      \stackrel{\scriptscriptstyle{\eqref{orzo}}}{=}
      &
  \Sum^{n}_{j=1} (\mu^c \otimes \mu^c \otimes \id^{\otimes 2}) \circ (\id \otimes \gD_{2,j-1,2} \otimes \id) \circ (\gD_{2,j,1} \otimes \id) \circ \gD_{j+1,n-j,j+1}
            \end{array}
    \end{equation*}

    
       \begin{equation*}
    \begin{array}{rcl}
     &
      \stackrel{\scriptscriptstyle{}}{=}
      &
  \Sum^{n+1}_{j=2} (\mu^c \otimes \mu^c \otimes \id^{\otimes 2}) \circ (\id \otimes \gD_{2,j-2,2} \otimes \id) \circ (\gD_{2,j-1,1} \otimes \id) \circ \gD_{j,n-j+1,j}
  \\[1mm]
     &
      \stackrel{\scriptscriptstyle{\eqref{aceto}, \eqref{riso1}}}{=}
      &
(d_0 \otimes \id) \circ \cup^c,
      \end{array}
   \end{equation*}
   where in the penultimate step we just used re-indexing and in the last one the fact that $d_0 \equiv 0$ on $\cC(0)$.
For the last face map on $\cC(n)$, we compute
   \begin{equation*}
    \begin{array}{rcl}
      &&    
      \cup^c \circ d_n
      \\[1mm]
      &
      \stackrel{\scriptscriptstyle{\eqref{riso3}, \eqref{aceto}}}{=}
      &
      \Sum^{n}_{j=1} (\mu^c \otimes \mu^c \otimes \id^{\otimes 2}) \circ (\id \otimes \gD_{2,j-1,1} \otimes \id) \circ ( \id \otimes \gD_{j,n-j,j}) \circ \gD_{2,n-1,1}
\\[1mm]
      &
      \stackrel{\scriptscriptstyle{\eqref{orzo}}}{=}
      &
          \Sum^{n}_{j=1} (\mu^c \otimes \mu^c \otimes \id^{\otimes 2}) \circ (\id \otimes \gD_{2,j-1,1} \otimes \id) \circ (\gD_{2,j,1} \otimes \id) \circ \gD_{j+1,n-j,j}
\\[1mm]
      &
      \stackrel{\scriptscriptstyle{\eqref{orzo}}}{=}
      &
          \Sum^{n}_{j=1} (\mu^c \otimes \mu^c \otimes \id^{\otimes 2}) \circ (\gD_{2,2,1} \otimes \id^{\otimes 2}) \circ (\gD_{3,j-1,1} \otimes \id) \circ \gD_{j+1,n-j,j}
\\[1mm]
      &
      \stackrel{\scriptscriptstyle{\eqref{polenta}}}{=}
      &
      \Sum^{n}_{j=1} (\mu^c \otimes \id \otimes \mu^c \otimes \id)
\circ (\id \otimes \gs \otimes \id) 
      \circ (\gD_{2,2,2} \otimes \id^{\otimes 2}) \circ (\gD_{3,j-1,1} \otimes \id) \circ \gD_{j+1,n-j,j}
\\[1mm]
      &
      \stackrel{\scriptscriptstyle{\eqref{orzo}}}{=}
      &
          \Sum^{n}_{j=1} (\mu^c \otimes \id \otimes \mu^c \otimes \id)
\circ (\gD_{2,j-1,1} \otimes \id^{\otimes 2}) \circ (\gD_{j,2,j} \otimes \id) \circ \gD_{j+1,n-j,j}
\\[1mm]
      &
      \stackrel{\scriptscriptstyle{\eqref{orzo}}}{=}
      &
          \Sum^{n}_{j=1} (\mu^c \otimes \id \otimes \mu^c \otimes \id)
\circ (\gD_{2,j-1,1} \otimes \id^{\otimes 2}) \circ (\id \otimes \gD_{2,n-j,1}) \circ \gD_{j,n+1-j,j}
\\[1mm]
      &
      \stackrel{\scriptscriptstyle{}}{=}
      &
          \Sum^{n}_{j=1} (\mu^c \otimes \id \otimes \mu^c \otimes \id)
\circ (\id^{\otimes 2} \otimes \gD_{2,n-j,1}) \circ (\gD_{2,j-1,1} \otimes \id) \circ \gD_{j,n+1-j,j}
\\[1mm]
     &
      \stackrel{\scriptscriptstyle{\eqref{aceto}, \eqref{riso3}}}{=}
      &
(\id \otimes d_{\rm last}) \circ \cup^c,
      \end{array}
   \end{equation*}
   as the last face map of $\id \otimes d$ on $\cup^c x \in  \bigoplus_{j} \cC(j-1) \otimes \cC(n+1-j)$ is in degree $n+1-j$, and where we used again that $d_{\rm last} \equiv 0$ on $\cC(0)$, hence the summand for $j=n+1$ vanishes.
   As far as the signs are concerned, this expression has to be rewritten by backwards applying the Koszul sign rule, that is, for the $j^{\rm th}$ summand $\cup^c_j$ of $\cup^c$ in the above sum multiplied by the sign $(-1)^n$ of the last face $d_n$ on $\cC(n)$, we obtain by using the notation for the graded and ungraded flips from \eqref{piadina}:
   \begin{equation}
     \label{bretzel}
     \begin{split}
     \cup^c_j \circ    (-1)^n  d_n &= (-1)^n (\id \otimes d_{n+1-j}) \circ \cup^c_j
\\
       &=
(-1)^{n+1-j} (-1)^{j-1} \sigma \circ (d_{n+1-j} \otimes \id) \circ \sigma
\circ \cup^c_j
\\
&=
   (-1)^{n+1-j} (-1)^{(n-j)(j-1)} \sigma \circ (d_{n+1-j} \otimes \id) \circ  (-1)^{(n+1-j)(j-1)} \sigma \circ \cup^c_j
\\
&=
    \tau \circ \big((-1)^{n+1-j} d_{n+1-j} \otimes \id \big) \circ \tau \circ \cup^c_j
     \end{split}
     \end{equation}
on $\cC(j-1) \otimes \cC(n+1-j)$, as desired.
  Similarly, one proves 
   $
(\id \otimes d_0) \circ \cup^c = (d_{\rm last} \otimes \id) \circ \cup^c,
   $
  which we omit.
  As for the intermediate faces, or rather the sum over the $d_i$ from \eqref{riso2}, we first
   write $\cup^c \circ \sum^{n-1}_{i=1} d_i$ on $\cC(n)$
by using Eqs.~\eqref{riso2} \& \eqref{aceto}, and decompose it afterwards 
   into two (double) sums: 
   \begin{equation}
     \label{carote}
\hspace*{-.6cm}{
     \begin{array}{rcl}
      &&    
       \cup^c \circ \Sum^{n-1}_{i=1}
       d_i
      \\[1mm]
      &
      \stackrel{\scriptscriptstyle{}}{=}
      &
      \Sum^{n}_{j=1} \Sum^{n-1}_{i=1}
      (\mu^c \otimes \id^{\otimes 2} \otimes \mu^c) \circ (\gD_{2,j-1,1} \otimes \id^{\otimes 2}) \circ (\gD_{j,n-j,j} \otimes \id) \circ \gD_{n-1, 2,i}
\\[1mm]
      &
      \stackrel{\scriptscriptstyle{}}{=}
      &
      \pig(\Sum^{n}_{j=2} \Sum^{j-1}_{i=1} + \Sum^{n}_{j=1} \Sum^{n-1}_{i=j}\pig)
      (\mu^c \otimes \id^{\otimes 2} \otimes \mu^c) \circ (\gD_{2,j-1,1} \otimes \id^{\otimes 2})
      \circ (\gD_{j,n-j,j} \otimes \id) \circ \gD_{n-1,2,i}.
      \end{array}
}
   \end{equation}
   We show now that these two double sums correspond to, respectively, $(\sum_{i=1}^{\rm max-1} 
   d_i \otimes \id) \circ \cup^c$ and $(\id \otimes \sum_{i=1}^{\rm max-1}
   d_i) \circ \cup^c$; the alternating sign will be thought of afterwards.

   Indeed,
   \begin{equation*}
     \begin{array}{rcl}
      &&    
       \pig(\Sum_{i=1}^{\rm max-1}
       d_i \otimes \id\pig) \circ \cup^c
       \\[1mm]
      &
      \stackrel{\scriptscriptstyle{\eqref{aceto}, \eqref{riso2}}}{=}
      &
      \Sum^{n+1}_{j=3} \Sum^{j-2}_{i=1}
      (\mu^c \otimes \id \otimes \mu^c \otimes \id) \circ (\id \otimes \gD_{j-2,2,i} \otimes \id) \circ (\gD_{2,j-1,1} \otimes \id) \circ \gD_{j,n+1-j,j}
\\[1mm]
      &
      \stackrel{\scriptscriptstyle{}}{=}
      &
      \Sum^{n}_{j=2} \Sum^{j-1}_{i=1}
      (\mu^c \otimes \id \otimes \mu^c \otimes \id) \circ (\id \otimes \gD_{j-1,2,i} \otimes \id) \circ (\gD_{2,j,1} \otimes \id) \circ \gD_{j+1,n-j,j+1}
\\[1mm]
  &
      \stackrel{\scriptscriptstyle{\eqref{orzo}}}{=}
      &
      \Sum^{n}_{j=2} \Sum^{j-1}_{i=1}
      (\mu^c \otimes \id \otimes \mu^c \otimes \id) \circ (\gD_{2,j-1,1} \otimes \id^{\otimes 2}) \circ (\gD_{j,2,i} \otimes \id) \circ \gD_{j+1,n-j,j+1}
\\[1mm]
  &
      \stackrel{\scriptscriptstyle{\eqref{orzo}}}{=}
      &
      \Sum^{n}_{j=2} \Sum^{j-1}_{i=1}
      (\mu^c \otimes \id \otimes \mu^c \otimes \id) \circ (\gD_{2,j-1,1} \otimes \id^{\otimes 2}) \circ
(\id \otimes \gs) \circ
      (\gD_{j,n-j,j} \otimes \id) \circ \gD_{n-1,2,i}
\\[1mm]
  &
      \stackrel{\scriptscriptstyle{}}{=}
      &
      \Sum^{n}_{j=2} \Sum^{j-1}_{i=1}
      (\mu^c \otimes \id^{\otimes 2} \otimes \mu^c) \circ (\gD_{2,j-1,1} \otimes \id^{\otimes 2}) \circ
      (\gD_{j,n-j,j} \otimes \id) \circ \gD_{n-1,2,i},
     \end{array}
   \end{equation*}
   where in the second step we simply re-indexed $j \mapsto j-1$, and which is the first double sum in Eq.~\eqref{carote}.
It is clear that introducing the sign $(-1)^i$ here and in the first double sum in \eqref{carote} does not change matters.
On the other hand,    
\begin{equation*}
  \begin{array}{rcl}
      &&    
\pig(\id \otimes \Sum_{i=1}^{\rm max-1} d_i \pig) \circ \cup^c
      \\[1mm]
      &
      \stackrel{\scriptscriptstyle{\eqref{aceto}, \eqref{riso2}}}{=}
      &
      \Sum^{n}_{j=1} \Sum^{n-j}_{i=1} (\mu^c \otimes \id^{\otimes 2} \otimes \mu^c) \circ (\id^{\otimes 2} \otimes \gD_{n-j,2,i}) \circ (\gD_{2,j-1,1} \otimes \id) \circ \gD_{j,n+1-j,j}
\\[1mm]
      &
      \stackrel{\scriptscriptstyle{}}{=}
      &
    \Sum^{n}_{j=1} \Sum^{n-j}_{i=1} (\mu^c \otimes \id^{\otimes 2} \otimes \mu^c) \circ (\gD_{2,j-1,1} \otimes \id^{\otimes 2}) \circ (\id \otimes \gD_{n-j,2,i}) \circ \gD_{j,n+1-j,j}
\\[1mm]
       &
       \stackrel{\scriptscriptstyle{\eqref{orzo}}}{=}
       &
    \Sum^{n}_{j=1} \Sum^{n-j}_{i=1} (\mu^c \otimes \id^{\otimes 2} \otimes \mu^c) \circ (\gD_{2,j-1,1} \otimes \id^{\otimes 2}) \circ (\gD_{j, n-j,j} \otimes \id) \circ \gD_{n-1,2,i+j-1}
\\[1mm]
      &
      \stackrel{\scriptscriptstyle{}}{=}
      &
    \Sum^{n}_{j=1} \Sum^{n-1}_{i=j} (\mu^c \otimes \id^{\otimes 2} \otimes \mu^c) \circ (\gD_{2,j-1,1} \otimes \id^{\otimes 2}) \circ (\gD_{j, n-j,j} \otimes \id) \circ \gD_{n-1,2,i},
  \end{array}
 \end{equation*}
where in the second step we just reordered operations and in the fourth we re-indexed by $i \mapsto i+j-1$. This is the second double sum in  Eq.~\eqref{carote}. Introducing the sign $(-1)^i$, due to the re-indexing of $i$ just mentioned, we obtain from this for the $j^{\rm th}$ summand $\cup^c_j$ that
\begin{equation*}
  \begin{split}
    \cup^c_j \circ \Sum^{n-1}_{i=j} (-1)^i d_i
    &= \big( \id \otimes \Sum^{n-j}_{i=1} (-1)^{i+j-1} d_i \big) \circ \cup^c_j
    \\
    &= (-1)^{j-1} \gs \circ  \big( \Sum^{n-j}_{i=1} (-1)^{i} d_i \otimes \id \big) \circ \gs \circ \cup^c_j
    \\
    &= (-1)^{(n-j)(j-1)} \gs \circ  \big( \Sum^{n-j}_{i=1} (-1)^{i} d_i \otimes \id \big) \circ (-1)^{(n+1-j)(j-1)} \gs \circ \cup^c_j
    \\
    &= \tau \circ \big(\Sum^{n-j}_{i=1} (-1)^{i} d_i \otimes \id\big) \circ \tau \circ \cup^c_j,
    \end{split}
  \end{equation*}
similar to what was done in \eqref{bretzel}, so that the signs turn out correctly.
This finishes the proof that $d$ is a coderivation with respect to $\cup^c$, and hence concludes the proof of the proposition.
  \end{proof}

Next, let us show that the cup coproduct is graded cocommutative up to homotopy. More precisely, we have:

\begin{prop}
  \label{farro}
 With respect to the totalisation of all partial decompositions as in \eqref{sale} and the differential from \eqref{olio}, the homotopy formula
  \begin{equation}
    \label{radicchio}
    \cup^c_\coop - \,\, \cup^c =
[-] \circ d
+   
  \tau \circ \big( d \otimes \mathrm{id} + (-1)^n \mathrm{id} \otimes d 
    \big) \circ \tau \circ [-]
%
  \end{equation}
  holds on $\cC(n)$ for any $n \in \N$, where $\cup^c_\coop \coloneqq  \tau \circ \cup^c$ denotes the coopposite cup coproduct.
  \end{prop}

\begin{proof}  This is a somewhat lengthy but nonetheless quite direct computation, in which we will compute all terms and compare them one by one. To start with, in degree $n$ on $\cC(n)$, one has:
\begin{equation*}
  \begin{array}{rcl}
      &&    
(-1)^{\foo n} \, \tau \circ (\id \otimes d) \circ \tau \circ [-]_n
      \\[1mm]
      &
      \stackrel{\scriptscriptstyle{\eqref{sale}, \eqref{olio}}}{=}
      &
      \Sum_{p+q=n+1} \Sum^{p}_{i=1}
(-1)^{\foo (q-1)(p-i)+1}
      (\mu^c \otimes \id^{\otimes 2}) \circ (\gD_{2,p-1,1} \otimes \id) \circ \gD_{p,q,i}
  \\[1mm]
      &&
\quad
      + \Sum_{p+q=n+1} \Sum^{p}_{i=1} (-1)^{\foo pn + i(q-1)+1} (\mu^c \otimes \id^{\otimes 2}) \circ (\gD_{2,p-1,2} \otimes \id) \circ \gD_{p,q,i}
      \\[1mm]
      &&
\quad
+ \Sum_{p+q=n+1} \Sum^{p}_{i=1} \Sum^{p-1}_{k=1}
(-1)^{\foo (q-1)(p-i)+k+1}
(\id \otimes \mu^c \otimes \id) \circ (\gD_{p-1,2,k} \otimes \id) \circ \gD_{p,q,i}
            \\[4mm]
      &
      \stackrel{\scriptscriptstyle{}}{=:}
      &
(1) + (2) + (3).
  \end{array}
 \end{equation*}
We have
\begin{equation*}
  \begin{array}{rcl}
    (1)
    &
      \stackrel{\scriptscriptstyle{}}{=}
      &
      \Sum_{p+q=n+1} \Sum^{p-1}_{i=1}
(-1)^{\foo (q-1)(p-i)+1}
(\mu^c \otimes \id^{\otimes 2}) \circ (\gD_{2,p-1,1} \otimes \id) \circ \gD_{p,q,i}
  \\[1mm]
      &&
\quad
+ \Sum_{p+q=n+1}
- (\mu^c \otimes \id^{\otimes 2}) \circ (\gD_{2,p-1,1} \otimes \id) \circ \gD_{p,q,p}
        \\[4mm]
      &
      \stackrel{\scriptscriptstyle{}}{=}:
      &
(4) + (5),
  \end{array}
 \end{equation*}
and by putting $q = n+1-p$, one sees that
\begin{equation*}
  \begin{array}{rcl}
    (5) &=& \Sum_{p+q=n+1}
  -  (\mu^c \otimes \id^{\otimes 2}) \circ (\gD_{2,p-1,1} \otimes \id) \circ \gD_{p,q,p}
    \\[1mm]
    &
      \stackrel{\scriptscriptstyle{}}{=}
      &
-    \Sum^{n+1}_{p=1} (\mu^c \otimes \id^{\otimes 2}) \circ (\gD_{2,p-1,1} \otimes \id) \circ \gD_{p,n+1-p,p}
\\[1mm]
    &
      \stackrel{\scriptscriptstyle{\eqref{aceto}}}{=}
      &
 - \,  \cup^c,
  \end{array}
 \end{equation*}
hence the cup coproduct. In a similar way, we deal with the term $(2)$. Indeed,
\begin{equation*}
  \begin{array}{rcl}
    (2)
    &
      \stackrel{\scriptscriptstyle{}}{=}
      &
      \Sum_{p+q=n+1} \Sum^{p}_{i=2} (-1)^{\foo pn + i(q-1)+1} (\mu^c \otimes \id^{\otimes 2}) \circ (\gD_{2,p-1,2} \otimes \id) \circ \gD_{p,q,i}
  \\[1mm]
      &&
\quad
      + \Sum_{p+q=n+1} (-1)^{\foo (p-1)q} (\mu^c \otimes \id^{\otimes 2}) \circ (\gD_{2,p-1,2} \otimes \id) \circ \gD_{p,q,1}
        \\[4mm]
      &
      \stackrel{\scriptscriptstyle{}}{=}:
      &
(6) + (7),
  \end{array}
 \end{equation*}
and let us show that $(7)$, in turn, equals minus the coopposite cup coproduct:
\begin{equation*}
  \begin{array}{rcl}
     \cup^c_\coop &=&  \tau \circ \cup^c
\\
    &
      \stackrel{\scriptscriptstyle{\eqref{aceto}}}{=}
      &
    \Sum^{n+1}_{p=1} \tau \circ (\mu^c \otimes \id^{\otimes 2}) \circ (\gD_{2,p-1,1} \otimes \id) \circ \gD_{p,n+1-p,p}
  \\[1mm]
      &
      \stackrel{\scriptscriptstyle{}}{=}
      &
      \Sum^{n+1}_{p=1}
(-1)^{\foo (p-1)(n+1-p)}
      (\mu^c \otimes \id^{\otimes 2}) \circ (\id \otimes \gs) \circ (\gD_{2,p-1,1} \otimes \id) \circ \gD_{p,n+1-p,p}
  \\[1mm]
      &
      \stackrel{\scriptscriptstyle{\eqref{orzo}}}{=}
      &
      \Sum^{n+1}_{p=1}
(-1)^{\foo (p-1)q}
      (\mu^c \otimes \id^{\otimes 2}) \circ (\gD_{2,n+1-p,2} \otimes \id) \circ \gD_{n+2-p,p-1,1}
  \\[1mm]
      &
      \stackrel{\scriptscriptstyle{}}{=}
      &
     \Sum^{n+1}_{p=1} (-1)^{\foo (p-1)q} (\mu^c \otimes \id^{\otimes 2}) \circ (\gD_{2,p-1,2} \otimes \id) \circ \gD_{p,n+1-p,1}
  \\[1mm]
      &
      \stackrel{\scriptscriptstyle{}}{=}
      &
     \Sum_{p+q=n+1} (-1)^{\foo (p-1)q} (\mu^c \otimes \id^{\otimes 2}) \circ (\gD_{2,p-1,2} \otimes \id) \circ \gD_{p,q,1}
  \\[4mm]
      &
      \stackrel{\scriptscriptstyle{}}{=}
      &
     (7),
  \end{array}
 \end{equation*}
where in line three from bottom we re-indexed $p \mapsto n+2-p$, which, however, does not change the sign.

Consider now
\begin{equation*}
  \begin{array}{rcl}
      &&    
[-]_{n-1} \circ d  
      \\[1mm]
      &
      \stackrel{\scriptscriptstyle{\eqref{sale}, \eqref{olio}}}{=}
      &
      \Sum_{p+q=n} \Sum^{p}_{i=1}
(-1)^{\foo (q-1)(p-i) + n}
      (\mu^c \otimes \id^{\otimes 2}) \circ (\id \otimes \gD_{p,q,i}) \circ \gD_{2,n-1,1}
  \\[1mm]
      &&
\quad
+      \Sum_{p+q=n} \Sum^{p}_{i=1}
(-1)^{\foo (q-1)(p-i)}
(\mu^c \otimes \id^{\otimes 2}) \circ (\id \otimes \gD_{p,q,i}) \circ \gD_{2,n-1,2}
      \\[1mm]
      &&
\quad
+     \Sum_{p+q=n} \Sum^{p}_{i=1} \Sum^{n-1}_{k=1}
(-1)^{\foo (q-1)(p-i) + k}
(\id^{\otimes 2} \otimes \mu^c) \circ (\gD_{p,q,i} \otimes \id) \circ \gD_{n-1,2,k}
            \\[4mm]
      &
      \stackrel{\scriptscriptstyle{}}{=:}
      &
(8) + (9) + (10),
  \end{array}
 \end{equation*}
and one sees that
\begin{equation*}
  \begin{array}{rcl}
      &&    
(8) + (9)
      \\[1mm]
      &
      \stackrel{\scriptscriptstyle{\eqref{orzo}}}{=}
      &
      \Sum_{p+q=n} \Sum^{p}_{i=1}
(-1)^{\foo (q-1)(p-i) + n}
      (\mu^c \otimes \id^{\otimes 2}) \circ
\pig(
(\gD_{2,p,1} \otimes \id) \circ \gD_{p+1,q,i}
\\[1mm]
&& \hspace*{8cm}
+ (-1)^{\foo n} (\gD_{2,p,2} \otimes \id) \circ \gD_{p+1,q,i+1}\pig)
      \\[1mm]
      &
      \stackrel{\scriptscriptstyle{}}{=}
      &
      \Sum_{p+q=n+1} \Sum^{p-1}_{i=1}
(-1)^{\foo (q-1)i + np}
      (\mu^c \otimes \id^{\otimes 2}) \circ
\pig(
(\gD_{2,p-1,1} \otimes \id) \circ \gD_{p,q,i}
\\[1mm]
&& \hspace*{8cm}
+ (-1)^{\foo n}
(\gD_{2,p-1,2} \otimes \id) \circ \gD_{p,q,i+1}\pig)
      \\[4mm]
      &
      \stackrel{\scriptscriptstyle{}}{=}
      &
- (4) - (6),
  \end{array}
 \end{equation*}
where we re-indexed $p \mapsto p+1$, and where for the second summand re-indexing $ i \mapsto i+1$ is still needed.
The reader having carefully followed the computation will be certainly aware of the fact that the only terms left at this stage are $(3)$, $(10)$, as well as $(-1)^{n+1} (\id \otimes d) \circ [-]_n$ in the homotopy identity \eqref{radicchio}. Hence, let us continue by
\begin{equation*}
  \begin{array}{rcl}
      &&    
\tau \circ (d \otimes \id) \circ \tau \circ [-]_n
      \\[1mm]
      &
      \stackrel{\scriptscriptstyle{\eqref{sale}, \eqref{olio}}}{=}
      &
      \Sum_{p+q=n+1} \Sum^{p}_{i=1}
(-1)^{\foo (q-1)(p-i) + p}
      (\id \otimes \mu^c \otimes \id)
\\[1mm]
&& \hspace*{4cm}
      \circ
\pig(
(-1)^{\foo q}
(\id \otimes \gD_{2,q-1,1}) \circ \gD_{p,q,i}
+ (\id \otimes \gD_{2,q-1,2}) \circ \gD_{p,q,i}\pig)
      \\[1mm]
      &
      &
\quad + 
\Sum_{p+q=n+1} \Sum^{p}_{i=1} \Sum^{q-1}_{k=1}
(-1)^{\foo (q-1)(p-i) + p +k}
(\id^{\otimes 2} \otimes \mu) \circ
(\id \otimes \gD_{q-1, 2, k}) \circ \gD_{p,q,i}
      \\[1mm]
      &
      \stackrel{\scriptscriptstyle{\eqref{orzo}}}{=}
      &
      \Sum_{p+q=n+1} \Sum^{p}_{i=1}
(-1)^{\foo (q-1)(p-i) + p}
      (\id \otimes \mu^c \otimes \id)
\\[1mm]
&& \hspace*{3.7cm}
      \circ
      \pig(
      (-1)^{\foo q}
(\gD_{p,2,i} \otimes \id) \circ \gD_{p+1,q-1,i}
+ (\gD_{p,2,i} \otimes \id) \circ \gD_{p+1,q-1,i+1}\pig)
      \\[1mm]
      &
      &
\quad + 
\Sum_{p+q=n+1} \Sum^{p}_{i=1} \Sum^{q-1}_{k=1}
(-1)^{\foo (q-1)(p-i) + p +k}
(\id^{\otimes 2} \otimes \mu) \circ
(\gD_{p, q-1, i} \otimes \id) \circ \gD_{p+q-2,2,k+i-1}
      \\[1mm]
      &
      \stackrel{\scriptscriptstyle{}}{=}
      &
      \Sum_{p+q=n+1} \Sum^{p-1}_{i=1}
(-1)^{\foo q(p-1-i) + p +1}
      (\id \otimes \mu^c \otimes \id)
\\[1mm]
&& \hspace*{4cm}
      \circ
\pig(
(-1)^{\foo q}
(\gD_{p-1,2,i} \otimes \id) \circ \gD_{p,q,i}
+ (\gD_{p-1,2,i} \otimes \id) \circ \gD_{p,q,i+1}\pig)
      \\[1mm]
      &
      &
\quad + 
\Sum_{p+q=n} \Sum^{p}_{i=1} \Sum^{q}_{k=1}
(-1)^{\foo q(p-i) + p  +k}
(\id^{\otimes 2} \otimes \mu) \circ
(\gD_{p, q, i} \otimes \id) \circ \gD_{n-1,2,k+i-1}
      \\[4mm]
      &
      \stackrel{\scriptscriptstyle{}}{=}:
      &
      (11) + (12) + (13),
  \end{array}
\end{equation*}
where in the penultimate step we re-indexed $p \mapsto p+1$ in the first two summands, and $q \mapsto q-1$ in all three. 
On the other hand, the (negative of the) double sums $(11)$ and $(12)$ are contained in the triple sum of $(3)$, and hence cancel out with the respective terms. Decomposing therefore $(3)$ and turning the sums around yields:

\begin{equation*}
  \begin{array}{rcl}   
&&
    (3) + (11) + (12)
\\[2mm]
    &
      \stackrel{\scriptscriptstyle{}}{=}
      &
      \pig(
      \Sum_{p+q=n+1} \Sum^{p-1}_{k=1} \Sum^{k-1}_{i=1}
      +
      \Sum_{p+q=n+1} \Sum^{p-1}_{k=1} \Sum^{p}_{i=k+2}
      \pig)
\\[3mm]
&& \hspace*{4cm}
           (-1)^{\foo (q-1)(p-i) + k+1}
      (\id \otimes \mu^c \otimes \id) \circ (\gD_{p-1,2,k} \otimes \id) \circ \gD_{p,q,i}
      \\[4mm]
      &
      \stackrel{\scriptscriptstyle{}}{=}:
      &
  (14) + (15).
  \end{array}
\end{equation*}
Re-indexing $k \mapsto k+i-1$ in $(13)$, we realise that the sum $(10)$ contains the sum $(13)$:  
\begin{equation*}
  \begin{array}{rcl}
      &&    
(10) + (13) 
\\[1mm]
    &
      \stackrel{\scriptscriptstyle{}}{=}
      &
      \pig(
      \Sum_{p+q=n} \Sum^{p}_{i=1} \Sum^{n-1}_{k=1}
      -
      \Sum_{p+q=n} \Sum^{p}_{i=1} \Sum^{q+i-1}_{k=i}
      \pig)
                 (-1)^{\foo (q-1)(p-i) + k}
      (\id^{\otimes 2} \otimes \mu^c) \circ (\gD_{p,q,i} \otimes \id) \circ \gD_{n-1,2,k}
\\[1mm]
    &
      \stackrel{\scriptscriptstyle{}}{=}
      &
      \pig(
      \Sum_{p+q=n} \Sum^{p}_{i=2} \Sum^{i-1}_{k=1}
      +
      \Sum_{p+q=n} \Sum^{p}_{i=1} \Sum^{n-1}_{k=q+i}
      \pig)
                       (-1)^{\foo (q-1)(p-i) + k}
      (\id^{\otimes 2} \otimes \mu^c) \circ (\gD_{p,q,i} \otimes \id) \circ \gD_{n-1,2,k}
\\[1mm]
    &
      \stackrel{\scriptscriptstyle{\eqref{orzo}}}{=}
      &
      \Sum_{p+q=n} \Sum^{p}_{k=1} \Sum^{p}_{i=k+1}
                 (-1)^{\foo (q-1)(p-i) + k}
      (\id^{\otimes 2} \otimes \mu^c) \circ (\gD_{p,q,i} \otimes \id) \circ \gD_{n-1,2,k}
      \\[1mm]
      &&
 \quad      
+
\Sum_{p+q=n} \Sum^{p}_{i=1} \Sum^{n-1}_{k=q+i}
                 (-1)^{\foo (q-1)(p-i) + k}
          (\id^{\otimes 2} \otimes \mu^c) \circ (\id \otimes \gs) \circ (\gD_{p,2,k-q+1} \otimes \id) \circ \gD_{p+1,q,i}
\\[1mm]
      &
      \stackrel{\scriptscriptstyle{\eqref{orzo}}}{=}
      &
      \Sum_{p+q=n} \Sum^{p}_{k=1} \Sum^{p+1}_{i=k+2}
                       (-1)^{\foo (q-1)(p-1-i) + k}
      (\id \otimes \mu^c \otimes \id) \circ (\id \otimes \gs) \circ (\gD_{p,q,i-1} \otimes \id) \circ \gD_{n-1,2,k}
      \\[1mm]
      &&
 \quad      
+
      \Sum_{p+q=n} \Sum^{p}_{i=1} \Sum^{p}_{k=i+1}
                 (-1)^{\foo (q-1)(p-i) + k-q+1}
      (\id \otimes \mu^c \otimes \id) \circ (\gD_{p,2,k} \otimes \id) \circ \gD_{p+1,q,i}
\\[1mm]
      &
      \stackrel{\scriptscriptstyle{}}{=}
      &
      \Sum_{p+q=n} \Sum^{p}_{k=1} \Sum^{p+1}_{i=k+2}
       (-1)^{\foo (q-1)(p-1-i) + k}
      (\id \otimes \mu^c \otimes \id) \circ (\gD_{p,2,k} \otimes \id) \circ \gD_{p+1,q,i}
      \\[1mm]
      &&
 \quad      
+
      \Sum_{p+q=n} \Sum^{p}_{k=2} \Sum^{k-1}_{i=1}
 (-1)^{\foo (q-1)(p-i) + k-q+1}
      (\id \otimes \mu^c \otimes \id) \circ (\gD_{p,2,k} \otimes \id) \circ \gD_{p+1,q,i}
      \\[1mm]
&
      \stackrel{\scriptscriptstyle{}}{=}
      &
      \Sum_{p+q=n+1} \Sum^{p-1}_{k=1} \Sum^{p}_{i=k+2}
 (-1)^{\foo (q-1)(p-i) + k}
      (\id \otimes \mu^c \otimes \id) \circ (\gD_{p-1,2,k} \otimes \id) \circ \gD_{p,q,i}
      \\[1mm]
      &&
 \quad      
+
      \Sum_{p+q=n+1} \Sum^{p-1}_{k=2} \Sum^{k-1}_{i=1}
(-1)^{\foo (q-1)(p-i) + k}
      (\id \otimes \mu^c \otimes \id) \circ (\gD_{p-1,2,k} \otimes \id) \circ \gD_{p,q,i}
      \\[4mm]
      &
      \stackrel{\scriptscriptstyle{}}{=}
      &
  (15) + (14),
  \end{array}
\end{equation*}
re-indexing in the fourth step
$i \mapsto i-1$ 
in the first sum, $k \mapsto k-q+1$ in the second sum, $p \mapsto p-1$ in the penultimate step, 
plus enhanced triple sum yoga throughout.
Hence all terms are taken care of, therefore Eq.~\eqref{radicchio} is proven. This concludes the proof. 
  \end{proof}

As all terms in the homotopy relation \eqref{radicchio} involving the differential $d$ disappear on homology, we immediately obtain:

\begin{coro}
  The cup coproduct from Eq.~\eqref{aceto} is graded cocommutative on homology, and hence the groups $H_\bull(\cC)$ form a graded cocommutative counital coalgebra.
    \end{coro}

This coproduct, however, is only part of a bigger structure, which leads to a notion dual to that of a Gerstenhaber algebra, which we already mentioned in the Introduction:

\begin{deff}
  \label{bombadacqua}
  A {\em Gerstenhaber coalgebra} is a graded
  cocommutative $k$-coalgebra $\big(\! \bigoplus_{n \in \N} V_n, \cup^c\big)$
  endowed with a graded Lie cobracket
  $
  \{-\} \colon V_n \to \bigoplus_{p+q=n+1} V_p \otimes V_q
  $
  of degree $+1$ such that the {\em coLeibniz} identity
\begin{equation}
  \label{biscotto}
(\id \otimes \cup^c) \circ \{-\} = \big(\{-\} \otimes \id + (\gvr \otimes \id) \circ (\id \otimes \{-\})\big) \circ \cup^c
  \end{equation}
holds, where $\gvr$ denotes the mixed flip as in \eqref{piadina}.
\end{deff}

\begin{prop}
  \label{miglio}
  Let $\cC$ be a cooperad with comultiplication, and denote $[-]^\coop := \rho \circ [-]$, where $\rho$ denotes the suspended flip as in \eqref{piadina}. Then the coopposite cobrace yields a coderivation of the cup coproduct, and the cobrace does so only up to homotopy. More precisely, one has
  \begin{equation}
    \label{ciambella1}
    (\mathrm{id} \otimes \cup^c) \circ [-]^\coop = \big([-]^\coop \otimes \id
    + (\gvr \otimes \id) \circ (\id \otimes [-]^\coop )\big) \circ \cup^c.
  \end{equation}
  Moreover, let
  \begin{equation}
    \label{pici}
    F \coloneqq \Sum_{p+q+r = n+2} \, \Sum^{p}_{i=1} \Sum^{p+q-1}_{j=q+i}
(-1)^{\foo (p+q)j + (p+r-1)i+ n}
    (\gD_{p,q,i} \otimes \id) \circ \gD_{p+q-1, r, j}
  \end{equation}
  as a map $\cC(n) \to \cC(p) \otimes \cC(q) \otimes \cC(r)$. Then 
  \begin{equation}
    \begin{split}
    \label{ciambella2}
    (\id \otimes \cup^c) \circ [-] &= \big([-] \otimes \id + (\gvr \otimes \id) \circ (\id \otimes [-]) \big) \circ \cup^c
    \\
    & \quad + F \circ d + (d \otimes \id^{\otimes 2}) \circ F
\\
& \quad
-  (\tau \otimes \id) \circ (d \otimes \id^{\otimes 2}) \circ (\tau \otimes \id) \circ F
\\
& \quad
+
(\id \otimes \tau) \circ
(\tau \otimes \id) \circ
(d \otimes \id^{\otimes 2}) \circ (\tau \otimes \id) \circ (\id \otimes \tau) \circ F
\end{split}
  \end{equation}
  holds.
\end{prop}
 
 \begin{proof}
For an improved reading flow, the proof of this proposition has been moved to the Appendix~\ref{appendix} due to its truly technical nature and notable length.
 \end{proof}
 
We have now gathered all necessary ingredients to state (and prove) our main theorem:

\begin{theorem}
  \label{asparagi}
  The homology groups arising from a cooperad with comultiplication constitute a Gerstenhaber coalgebra.
  \end{theorem}

\begin{proof}
This follows by combining Corollary \ref{sushi}, Proposition \ref{farro}, and Proposition \ref{miglio}, taking into account that all terms involving the differential $d$ from Proposition \ref{sugo} disappear when descending to homology.
\end{proof}

\section{Examples}
In this section, we exhibit a couple of classical homology theories that are endowed with the structure of a Gerstenhaber coalgebra as a consequence of Theorem \ref{asparagi}, and give the necessary technical details along with all structure maps. To our knowledge, while some single pieces in some example contexts appeared before, the full respective structures have not been observed in the literature so far.

\subsection{Chain spaces arising from bialgebras as cooperads}
\label{gelato}

  Let $(B,\mu,\eta,\gd,\gve)$ be a bialgebra over a commutative ring $k$, where $\gd$ denotes here the customary coalgebra comultiplication, in order to avoid confusion with the (collection of) cooperadic decompositions $\gD$, see below.

\begin{construction}
  For any $p\in \mathbb{N}$, set $C_p(B) \coloneqq B^{\otimes p}$. The $k$-linear $\mathbb{N}$-module $C_\bullet(B)=\{C_p(B)\}_{p \in \N}$ can be endowed with the structure of a cooperad as follows: for any $n, k, n_1,\dots,n_k \in \mathbb{N}$ such that $n_1+\dots+n_k=n$, define the map
  \[
\Delta_{k,n_1,\dots, n_k}\colon C_n(B) \rightarrow C_k(B) \otimes \big(C_{n_1}(B)\otimes \dots \otimes C_{n_k}(B)\big)
\]
as the composition:
    \begin{center}
          \begin{tikzcd}
    C_n(B)=B^{\otimes n} \arrow[rr] \arrow[d, "\gd^{\otimes n}"] &  & C_k(B)\otimes \big(C_{n_1}(B)\otimes \dots \otimes C_{n_k}(B)\big)                                                                    \\
    (B^{\otimes 2})^{\otimes n} \arrow[rr, "\simeq "]               &  & (B^{\otimes n_1}\otimes \dots \otimes B^{\otimes n_k})\otimes (B^{\otimes n_1}\otimes \dots \otimes B^{\otimes n_k}),
    \arrow[u, "(\mu_{n_1}\otimes \dots \otimes \mu_{n_k})\otimes (\text{id}\otimes\dots\otimes \text{id})"']
    \end{tikzcd}
    \end{center}
    where we adopt the convention that, if some $n_j=0$, the map $\mu_{0}\colon B^{\otimes 0}\simeq k \to B$ is given by the
    unit $\eta\colon k\to B$.

    \pagebreak
    
   \noindent If  we denote an elementary tensor in $C_n(B)$ by $(x_1,\dots,x_n)$ and use
  Sweedler notation
 $\delta(x) = x^{(1)} \otimes x^{(2)}$
  with superscripts for the coalgebra coproduct, we can
  write
\begin{equation*}
  \begin{split}
 & \Delta_{k,n_1,\dots, n_k}(x_1,\ldots,x_n)
    \\
= \ & \big(\mu_{n_1}(x_1^{(1)},\ldots,x_{n_1}^{(1)}),\ldots , \mu_{n_k}(x_{n-n_k +1}^{(1)},\ldots,x_n^{(1)})\big) \otimes (x_1^{(2)},\ldots, x_{n_1}^{(2)})\otimes \ldots \otimes (x_{n-n_k+1}^{(2)},\ldots,x_n^{(2)}),
\end{split}
\end{equation*}
where $\mu_n$ denotes the algebra multiplication in $n$ elements.
Let $\Delta$ denote the collection of these (co-)operations, that is, $\Delta= \{\Delta_{k,n_1,\dots,n_k}\}_{k,n_i \in \N}$. Apart from the (crucial) difference that we also consider nonzero objects for $n=0$, this is the cooperadic structure on $C_\bull(B)$ mentioned in \cite{BerMoe:AHTFO} or \cite[Lem.~4.10]{VdL:OATHAOR}.
  \end{construction}

\begin{prop}
  Let $(B,\mu,\delta,\eta,\gve)$ be a $k$-bialgebra. Then $(C_\bullet(B), \Delta, \gve)$ is a counital cooperad.
\end{prop}

\noindent See {\em op.~cit.} for a proof. For our purposes let us, nevertheless, explicitly write down the partial decompositions: for any $p,q \in \mathbb{N}$ and $1 \leq i \leq p$, the $i^{\text{th}}$-partial decomposition is given by
\begin{eqnarray}
  \label{pandoro}
  \Delta_{p,q,i} \colon  \ C_{p+q-1}(B) & \rightarrow &  C_p(B) \otimes C_q(B),
  \\
  \nonumber
 (x_1,\dots,x_{p+q-1})
  &\mapsto & (x_1,\dots,x_{i-1}, x_i^{(1)} \cdots x_{i+q-1}^{(1)}, x_{i+q},\dots,x_{p+q-1})\otimes (x_i^{(2)},\dots,x_{i+q-1}^{(2)}),
\end{eqnarray}
where in the left tensor factor for compactness we denoted the multiplication in $B$ simply by juxtaposition.
If $q=0$, then the above simplifies to
$$
\Delta_{p,0,i}(x_1,\dots,x_{p-1})= (x_1,\dots,x_{i-1}, 1_B ,x_i,\dots,x_{p-1} )\otimes (1_k).
$$

\begin{construction}
 We define a map $\psi\colon C_\bullet(B) \to \ass^*$ of $k$-linear $\mathbb{N}$-modules. Since $\ass^*$ is defined by taking the linear dual arity-wise and $\ass(n)$ is the free $k$-module on the generator $\nu_n$, any element $f\in \ass(n)^*=\mathsf{Hom}(\ass(n),k)=\mathsf{Hom}(k\nu_n, k)$ may be identified with the scalar $f(\nu_n)$, and we use this observation to define:
    \begin{align*}
   & \psi_0(1_k)\coloneqq 1_k  & \\
   & \psi_p(x_1,\dots,x_p) \coloneqq \gve (\mu_p (x_1,\dots,x_p)) & \text{for any $p\geq 1$ and $(x_1,\dots,x_p) \in C_p(B)$.}
    \end{align*}     

\end{construction}

\begin{prop}
    The map $\psi\colon \ass^*\to C_\bullet(B)$ in $\mathbb{N}$-$\mathsf{mod}_k$ is a map of cooperads. In other words, $(C_\bullet(B),\psi)$ is a cooperad with comultiplication.
\end{prop}

\begin{proof}    Checking that $\psi$ is a morphism of cooperads boils down to
    observing that
    $$
    \gve\big(\mu_{n-1}(x_1,\dots,x_{n-1})\big)= \gve\big(\mu_n(x_1,\dots,x_{i-1}, 1_B, x_i,\dots, x_{n-1})\big) 
    $$
    for any $n\geq 1$ and $1 \leq i \leq n$, as well as considering that, if $m\geq 1$,  and $(x_1,\dots,x_{n+m-1})\in C_{n+m-1}(B)$ is an elementary tensor, the two scalars
    \[
    \gve \big(\mu_{n+m-1}(x_1,\dots,x_{n+m-1})\big)
    \]
    and
    \[
    \gve\big(\mu_n(x_1,\dots,x_{i-1}, \mu_m(x_i^{(1)},\dots,x_{i+m-1}^{(1)}), x_{i+m},\dots,x_{m+n-1})\big)\cdot\gve\big(\mu_m(x_i^{(2)},\dots,x_{i+m-1}^{(2)})\big)
    \]
    are equal.
This latter affirmation is true since $\gve$ is a morphism of algebras, the ground ring $k$ is assumed to be commutative, which implies that $\gve\circ \mu_2= \gve\otimes\gve$ as well as $(\gve\otimes\gve)\circ \Delta= \gve$.
    \end{proof}

\begin{rmk}
  \label{aglio}
  The morphism $\psi\colon C_\bullet(B)\to \ass^*$ just defined induces a morphism
  $$
  \psi^*\colon \ass^{**}=\ass \to C_\bullet(B)^*
  $$
  of operads. Explicitly, the three special operations which make $C_\bullet(B)^*$ an operad with multiplication are given by the $k$-linear maps
  \begin{equation}
    \label{panettone}
    \begin{array}{lr}
      \mu^c = \gve \circ \mu_2 \colon & B\otimes B \to k,
      \\
      \mathbbm{1}^c = \gve\colon & B \to k,
\\
      e^c = \id_k \colon  &k \to k.  
    \end{array}
    \end{equation}
    \end{rmk}

\noindent
For the homological part, 
recall that the
graded $k$-module $C_\bull(B) = \bigoplus_{p} C_p(B)$ can also be considered as a chain complex with standard differential
 given by 
 \begin{equation}
   \label{papavero}
  \begin{array}{rcl}
    d_n\colon  C_n(B)  &\to& C_{n-1}(B), \\ 
    (b_1,\dots,b_n) & \mapsto  &
    (\gve(b_1)b_2,b_3,\dots,b_n)
+ \Sum_{i=1}^{n-1}(-1)^i (b_1,\ldots, b_i b_{i+1}, \ldots, b_n)
\\
&&  + (-1)^n(b_1,\ldots,b_{n-1}\gve(b_n)).
\end{array}
 \end{equation}
 Even if this only depends on the bialgebra structure and no antipode is needed, to better distinguish this from the Gerstenhaber-Schack bialgebra homology \cite{GerSch:ABQGAAD} that appears in bialgebra deformation theory, we refer to the respective homology $H_\bull(B,k) := H\big(C_\bull(B), d \big)$ as {\em Hopf algebra homology} (with trivial coefficients), which, if $k$ is a field or, more general, if $B$ is $k$-projective, computes the groups $\Tor^B_\bull(k,k)$. We can then state:

\begin{lemma}
\label{sesamo}
The differential \eqref{olio} 
coincides with the canonical one in \eqref{papavero} arising in Hopf algebra homology. Hence, the homology of the cooperad $C_\bull(B)$ induced by its comultipli\-cation in the sense of Definition \ref{zenzero} coincides with the Hopf algebra homology of $C_\bull(B)$. 
  \end{lemma}

\begin{proof}
 This is a straightforward verification by simply comparing the map \eqref{papavero} to the differential arising from the alternating sum of the faces \eqref{riso1}--\eqref{riso3} by means of the decomposition maps \eqref{pandoro} and the (co)multiplication element \eqref{panettone}, and is therefore left to the reader.
  \end{proof}

As an immediate consequence of this and our main Theorem \ref{asparagi}, we have:

\begin{prop}
  \label{fave}
For any $k$-bialgebra $B$, the homology groups $H_\bull(B,k)$ yield a Gerstenhaber coalgebra. Hence, if $B$ is $k$-projective, then the homology groups $\Tor_\bull^B(k,k)$ constitute a Gerstenhaber coalgebra, and as such, in particular, carry the structure of a graded cocommutative $k$-coalgebra.
  \end{prop}

\begin{rmk}
As far as we know, only the structure of a graded cocommutative coproduct on $\Tor_\bull^B(k,k)$ has been obtained before from a different approach. This appears in, for example, \cite[Prop.~7.10 \& Cor.~7.12]{McC:AUGTSS} or elsewhere. A short computation
with the help of the special elements \eqref{panettone} from Remark \ref{aglio} and the general expression \eqref{aceto} of the cup coproduct leads to the expected expression:
\begin{equation*}
  \begin{split}
    \cup^c (x_1, \ldots, x_n)
& = \Sum^{n+1}_{j=1} (\mu^c \otimes \id^{\otimes 2}) \circ (\Delta_{2, j-1, 1} \otimes \id) \circ \Delta_{j, n+1-j, j} (x_1, \ldots, x_n)
    \\
    & =
    (\mu^c \otimes \id^{\otimes 2}) \circ (\Delta_{2, 0, 1} \otimes \id) \circ \Delta_{1, n, 1} (x_1, \ldots, x_n)
    \\
    &
\quad    +  \Sum^{n}_{j=2} (\mu^c \otimes \id^{\otimes 2}) \circ (\Delta_{2, j-1, 1} \otimes \id) \circ \Delta_{j, n+1-j, j} (x_1, \ldots, x_n)
    \\
    &
\quad +    (\mu^c \otimes \id^{\otimes 2}) \circ (\Delta_{2, n, 1} \otimes \id) \circ \Delta_{n+1, 0, n+1} (x_1, \ldots, x_n)
    \\
    & =
    \mu^c(1_B, x^{(1)}_1  \cdots x^{(1)}_n) \otimes  (1_k) \otimes (x^{(2)}_1,  \ldots, x^{(2)}_n) 
    \\
    &
    \quad    +  \Sum^{n}_{j=2}
 \mu^c(x^{(1)}_1 \cdots x^{(1)}_{j-1}, x^{(1)}_{j} \cdots x^{(1)}_{n}) \otimes (x^{(1)}_{1}, \ldots x^{(1)}_{j-1}) \otimes (x^{(2)}_{j}, \ldots, x^{(2)}_{n})
    \\
    &
    \quad +
   \mu^c(x^{(1)}_1  \cdots x^{(1)}_n, 1_B) \otimes (x^{(2)}_1,  \ldots, x^{(2)}_n) \otimes (1_k)
\\
    &= (1_k) \otimes (x_1, \ldots, x_n) +
\Sum^{n-1}_{j=1} (x_1, \ldots, x_j) \otimes (x_{j+1}, \ldots, x_n)
+
(x_1, \ldots, x_n) \otimes (1_k),
  \end{split}
  \end{equation*}
that is, the sum over all possible deconcatenations of $(x^1, \ldots, x^n)$, plus the trivial ones; compare {\em loc.~cit.} again.
\end{rmk}

\subsection{Cooperadic structures on Hochschild chains for Frobenius algebras}

In striking contrast to (homotopy) Gerstenhaber algebra structures on Hochschild cohomology (resp.\ cochains, with values in the algebra in question itself) as a result of a multiplicative operadic structure on the level of Hochschild cochains, it is not so clear how to directly dualise this to a cooperadic structure on Hochschild {\em chains} (resp.\ a Gerstenhaber coalgebra structure on Hochschild homology), at least not with values in the algebra in question itself. It seems that an extra structure is needed in order to decompose a Hochschild chain, seen as an upside-down tree with the coefficients on top, into two of the same kind (as becomes quite clear from a graphical presentation). Such an extra structure arises, for example, by the coproduct that comes with a Frobenius algebra, while the cooperadic comultiplication is induced by the Frobenius functional.

\subsubsection{Frobenius algebras}

Let us briefly recall one possible standard definition of Frobenius algebras that suits best our needs, plus a couple of technical details required for subsequent computations. Let $k$ be a field in what follows.

 \begin{deff}
 A {\em Frobenius algebra} is an algebra $A$ with a 
 functional 
 $ \varepsilon : A \rightarrow k$ 
 such that the map $A \rightarrow A^*, a \mapsto
 \varepsilon_a$  
 is bijective, where $ \varepsilon_a(b):=\varepsilon(ab)$.
 The functional $ \varepsilon $
 is called a {\em Frobenius functional}.
 \end{deff}

 This implies that $A$ is finite-dimensional and also a $k$-coalgebra with $A$-bilinear coproduct
$\gD \colon A \to A \otimes A$,
 the counit of which is precisely the Frobenius functional: let $\{e_i\}_{1 \leq i \leq n}  \in A$ be a basis and  $\{e^i\}_{1 \leq i \leq n}  \in A$ another basis defined via $\gve(e^j e_i) = \delta^j_i$.
Observe that, despite of its notation, the $e^j$ are {\em not} the dual basis but live in $A$ itself.
 Set
 $$
 	\Delta (1)= 1^{(1)} \otimes 1^{(2)} = \textstyle\sum_{i}  e_i \otimes e^i,
 $$
 and by $A$-bilinearity we have for all $a \in A$:
 \begin{equation}
 \label{sonne}
 	\Delta (a)= a^{(1)} \otimes a^{(2)} = \textstyle\sum_{i} ae_i \otimes e^i = \textstyle\sum_{i} e_i \otimes e^ia,
 \end{equation}
and then counitality comes out as
 \begin{equation}
 \label{mond1}
 	a=\textstyle\sum_{i} \varepsilon (ae_i)e^i=
 	\textstyle\sum_{i} \varepsilon (e^ia)e_i
 \end{equation}
 for all $a \in A$.

 The map $A \to A^*, \, a \mapsto \gve_a$ is right $A$-linear with respect to right multiplication on
 $A$ and right action on $A^*$ defined 
 by $\phi \otimes a \mapsto \phi(a(-))$. This way, one obtains an isomorphism $A \simeq A^*$ of right $A$-modules, which implies that $a \mapsto \gve_a$ becomes a morphism of left $A$-modules as well with respect
to the {\em twisted} left $A$-action
$a \otimes b \mapsto \gs(a)b$
 on $A$ for an automorphism $\gs \in \Aut(A)$, and left action on $A^*$ given by $a \otimes \phi \mapsto \phi((-)a)$:
in total,  
$A^* \simeq {}_\gs A$ as $A$-bimodules with respect to this twisted action.
This
leads  for all $a,b \in A$ to the identity
 \begin{equation}
   \label{fagiolini}
 \gve(ab) = \gve(\gs(b)a).
 \end{equation}
The automorphism $\gs$ is determined up to inner automorphisms and usually called the {\em Na\-ka\-ya\-ma} automorphism.

 \subsubsection{Hochschild chains and cochains}
  Denote by $\Ae := A \otimes \Aop$ the enveloping algebra of an associative $k$-algebra $A$, and let $M$ be an $A$-bimodule, considered, equivalently, as a (right) $\Ae$-module. For any $A$ and any $M$, one defines the Hochschild chain resp.\ cochain spaces by
 $$
  C_\bull(A, M) = M \otimes_\Ae \mathrm{Bar}_\bull(A), \qquad
  C^\bull(A, M) = \Hom_\Ae(\mathrm{Bar}_\bull(A), M),
  $$
  where $\mathrm{Bar}_\bull(A) := A \otimes A^{\otimes \bullet} \otimes A$ denotes the {\em bar resolution} of $A$. Equivalently,
 $$
  C_\bull(A, M) = M \otimes A^{\otimes \bullet}, \qquad
  C^\bull(A, M) = \Hom(A^{\otimes \bullet}, M)
  $$
  as $k$-modules.
 In case of a Frobenius algebra as above such that ${}_\gs \! A \simeq A^*$ as right $\Ae$-modules, Hochschild chains and cochains are in duality: one has
 \begin{equation*}
   \begin{split}
     C_\bull(A, {}_\gs \! A)^* &=
     \Hom({}_\gs \! A \otimes_\Ae \mathrm{Bar}_\bull(A), k)
     \\
& \simeq
     \Hom_\Ae\big( \mathrm{Bar}_\bull(A), ({}_\gs \! A)^* \big)
     \\
     & \simeq
     \Hom_\Ae\big( \mathrm{Bar}_\bull(A), A\big) = C^\bull(A,A),
\end{split}
   \end{equation*}
 where the second line is simply the adjunction, whereas the third line arises from being Frobenius.
 Explicitly, this isomorphism is induced by the Frobenius functional by means of
 \begin{equation}
   \label{polpo}
\hat{} \, \colon C^p(A,A)
\to  C_p(A, {}_\gs \! A)^*, \quad f \mapsto \big\{ \, (a_0, \ldots, a_p) \mapsto 
\gve( a_0 f(a_1, \ldots, a_p)) \, \big\}.
   \end{equation}
Its inverse is less simple to describe, but fortunately we are not going to need it explicitly.

\begin{lemma}
  For a Frobenius algebra $(A,\varepsilon)$, 
there exists a unique operadic structure on $C_\bullet(A, {}_\gs \! A)^*$ such that the isomorphism \eqref{polpo}
extends to an operad isomorphism.
This becomes an operad with multiplication by means of the composition $\ass\to C^\bullet(A, A) \xrightarrow{ \hat{ }} C_\bullet(A, {}_\gs \! A)^*.$
\end{lemma}

\noindent The actual construction happens, as expected, by transport of structure:
the requirement that the map $\hat{(-)}$ is an isomorphism of operads means that the partial composition maps
$
\circ_i \colon C_p(A, {}_\gs \! A)^* \otimes C_q(A, {}_\gs \! A)^*  \rightarrow C_{p+q-1}(A, {}_\gs \! A)^*
$
are uniquely determined by the equality
$$
\hat{f}\circ_i \hat{g} = \widehat{f\circ_i g}
$$ 
for all $f \in C^p(A,A)$, $g \in C^q(A,A)$, and $1 \leq i \leq p$.
 Explicitly, we can write 
 \begin{equation}
   \label{panna}
\begin{split}
  \hat{f} \circ_i \hat{g} (a_0,\ldots,a_{p+q-1})
  &= \widehat{f \circ_i g} (a_0,\dots,a_{p+q-1})
  \\
  &= \gve\big(a_0f\circ_i g (a_1,\dots,a_{p+q-1})\big)
  \\
  &= \gve\big(a_0f(a_1,\ldots,a_{i-1},g(a_i,\ldots,a_{i+q-1}),a_{i+q},\ldots,a_{p+q-1})\big)
  \\
  &= \textstyle\sum_i \gve\big(a_0f(a_1,\ldots,a_{i-1},e_i ,a_{i+q},\ldots,a_{p+q-1}))
\gve\big(e^i g(a_i,\ldots,a_{i+q-1})\big),
  \\
  &= \textstyle\sum_i \hat f(a_0, a_1,\ldots,a_{i-1},e_i ,a_{i+q},\ldots,a_{p+q-1}) \hat g
(e^i, a_i,\ldots,a_{i+q-1}),
  \\
  &= \textstyle\sum_i \hat f(a_0, a_1,\ldots,a_{i-1}, 1^{(1)} ,a_{i+q},\ldots,a_{p+q-1}) \hat g
(1^{(2)}, a_i,\ldots,a_{i+q-1}),
\end{split}
 \end{equation}
 where we used \eqref{mond1} in the penultimate step and
 $\gD(1) = 1^{(1)} \otimes 1^{(2)} = \sum_i e_i \otimes e^i$ in the last.
 The unit is given by
 $$ \mathbbm{1}(a_0,a_1)= \gve(a_0a_1),
 $$
 while the operad morphism
 $$
 \phi\colon \ass\rightarrow C^\bullet(A,A)
 $$
 sends the unique operation $\nu_n \in \ass(n)$ to the functional $\phi(\nu_n)\colon A\otimes A^n \to k$ acting as $\phi(\nu_n)(a_0,\ldots,a_n)= \gve(a_0\cdots a_n).$
 
 As $A$ is finite dimensional, the isomorphism $C_p(A, {}_\gs \! A)^{**} \simeq C_p(A, {}_\gs \! A)$ then yields:

\begin{coro}
  \label{ciliegie}
  For any Frobenius algebra $A$, the Hochschild chain complex $C_\bull(A, {}_\gs A)$ has the structure of a cooperad, 
where  
  the partial decomposition maps are defined as
  \begin{equation}
\label{dechoch}
    \begin{array}{rcl}
    \Delta_{p,q,i}\colon   C_{p+q-1}(A, {}_\gs \! A) & \to &  C_p(A, {}_\gs \! A)\otimes C_q(A, {}_\gs \! A),
\\
(a_0,\dots,a_{p+q-1})  & \mapsto & (a_0,\dots,a_{i-1},1^{(1)}, a_{i+q},\dots,a_{p+q-1}) \otimes (1^{(2)},a_i,\dots,a_{i+q-1}).
\end{array}
  \end{equation}
This cooperad is equipped with a comultiplication induced by the Frobenius functional and the multiplication in $A$, that is to say, 
\begin{equation}
  \label{merluzzo}
  \begin{array}{rclc}
  \mu^c \colon \ C_2(A, {}_\gs \! A) \to k, & (a_0, a_1, a_2) &\mapsto &\gve(a_0a_1a_2),
  \\
  \mathbbm{1}^c \colon \ C_1(A, {}_\gs \! A) \to k, & (a_0, a_1) &\mapsto & \gve(a_0a_1),
  \\
   e^c \colon \ C_0(A, {}_\gs \! A) \to k, & (a_0) &\mapsto & \gve(a_0),
\end{array}
\end{equation}
for any $a_0, a_1, a_2 \in A$.
\end{coro}

\begin{proof}
  We just need to check that the isomorphism $C_p(A, {}_\gs \! A)^{**} \simeq C_p(A, {}_\gs \! A)$ sends the dual of the composition product for $C_p(A, {}_\gs \! A)^*$ to the maps defined in \eqref{dechoch}, which is evident from looking at the arguments on the left and right hand side of \eqref{panna}, along with
  $
  \big( C_p(A, {}_\gs \! A)\otimes C_q(A, {}_\gs \! A)\big)^*
  \simeq
C_p(A, {}_\gs \! A)^* \otimes C_q(A, {}_\gs \! A)^*,
$
which is true since $A$ is finite dimensional. 
\end{proof}

\begin{lemma}
  \label{more}
  The differential \eqref{olio} resp.\ \eqref{olio1} coincides with the classical Hochschild differential. Hence, the homology of the cooperad $C_\bull(A, {}_\gs \! A)$ induced by its comultipli\-cation in the sense of Definition \ref{zenzero} coincides with the Hochschild homology of the Frobenius algebra $A$. 
  \end{lemma}

\begin{proof}
  Recall first that the Hochschild chain space $C_\bull(A, {}_\gs \! A)$, as a simplicial $k$-module, is equipped with the faces $d_i(a_0, \ldots, a_n) = (a_0, \ldots, a_ia_{i+1}, \ldots, a_n)$ for $0 \leq i \leq n-1$, along with $d_n(a_0, \ldots, a_n) = (\gs(a_n)a_0, a_1, \ldots, a_{n-1})$. A direct computation gives that these are precisely the ones defined in Eqs.~\eqref{riso1}--\eqref{riso3} with respect to the cooperadic decomposition \eqref{dechoch} and the (co)multiplication elements \eqref{merluzzo}: using $\gD(1) = \sum_j e_j \otimes e^j$, one has
  \begin{equation*}
    \begin{split}
      d_i(a_0, \ldots, a_n)
      &= (\id \otimes \mu^c) \circ \gD_{n-1,2,i}(a_0, \ldots, a_n) 
      \\
      &= (\id \otimes \mu^c) \big((a_0, \ldots, a_{i-1}, e_j, a_{i+2}, \ldots, a_n) \otimes (e^j, a_i, a_{i+1}) \big)
      \\
      &= (a_0, \ldots, a_{i-1}, \gve(e^ja_ia_{i+1}) e_j, a_{i+2}, \ldots, a_n) 
  \\
      &= (a_0, \ldots, a_{i-1}, a_ia_{i+1}, a_{i+2}, \ldots, a_n) 
        \end{split}
    \end{equation*}
for $1 \leq i \leq n-1$, where we used \eqref{mond1} in the last step.  Similarly,
  \begin{equation*}
    \begin{split}
      d_n(a_0, \ldots, a_n)
      &= (\mu^c \otimes \id) \circ \gD_{2,n-1,1}(a_0, \ldots, a_n) 
      \\
      &= (\mu^c \otimes \id)\big( (a_0, e_j , a_n) \otimes (e^j, a_1, \ldots, a_{n-1} )\big)
      \\
      &= \big(\gve(a_0e_ja_n) e^j, a_1, \ldots, a_{n-1} \big)
  \\
      &= \big(\gve(\gs(a_n) a_0 e_j) e^j, a_1, \ldots, a_{n-1} \big)
     \\
      &= (\gs(a_n) a_0, a_1, \ldots, a_{n-1}),
    \end{split}
    \end{equation*}
where we used \eqref{fagiolini} in the penultimate, and \eqref{mond1} in the last step. The case for $d_0$ is similar and left to the reader.
  \end{proof}

Wrapping up, combining Corollary \ref{ciliegie} \& Lemma \ref{more} with our main result in Theorem~\ref{asparagi}, we obtain a Gerstenhaber coalgebra structure on Hochschild homology for Frobenius algebras:

\begin{prop}
  \label{limone}
The Hochschild homology $H_\bull(A, {}_\gs \! A)$ of any Frobenius algebra $(A,\gve)$ carries the structure of a Gerstenhaber coalgebra. In particular, if $A$ is $k$-projective, then the groups $\Tor_\bull^\Ae({}_\gs \! A,A)$ form a Gerstenhaber coalgebra.  
  \end{prop}

\smallskip

\begin{rmk}
The structure of a (graded cocommutative) coproduct on Hochschild homology at least for symmetric Frobenius algebras, {\em i.e.}, those for which the Nakayama automorphism $\gs$ equals the identity has been observed before, for example, in \cite[\S6.5]{WahWes:HHOSA}. Computing \eqref{aceto} with the help of the decomposition \eqref{dechoch} and the special elements in \eqref{merluzzo} yields 
\begin{equation*}
\begin{split}
\cup^c(a_0, \ldots, a_n)
 = \ &
\Sum^{n+1}_{j=1}
  (\mu^c \otimes \id^{\otimes 2}) \circ (\gD_{2,j-1,1} \otimes \id) \circ \gD_{j,n+1-j,j} (a_0,\ldots, a_n)
  \\
  = \ &
\Sum_i \Sum^{n+1}_{j=1} 
(\mu^c \otimes \id^{\otimes 2}) \circ (\gD_{2,j-1,1} \otimes \id) \big( (a_0,\ldots, a_{j-1}, e_i) \otimes (e^i, a_{j}, \ldots, a_n)
\big)
  \\
  = \ &
  \Sum_{i,k} \Sum^{n+1}_{j=1} 
(\mu^c \otimes \id^{\otimes 2}) \big( (a_0, e_k, e_i) \otimes (e^k,a_1, \ldots, a_{j-1}) \otimes (e^i, a_{j}, \ldots, a_n)
\big)
  \\
  = \ &
   \Sum_{i,k} \Sum^{n+1}_{j=1} 
\big( e^k,a_1, \ldots, a_{j-1}\big) \otimes \big( \gve(a_0e_k e_i) e^i, a_{j}, \ldots, a_n\big )
  \\
  = \ &
  \Sum_{k} \Sum^{n+1}_{j=1}
  (e^k,a_1, \ldots, a_{j-1}\big) \otimes \big( a_0e_k, a_{j}, \ldots, a_n)
  \\
  = \ &
\Sum^{n+1}_{j=1} 
( a_0^{(2)} ,a_1, \ldots, a_{j-1}\big) \otimes \big( a_0^{(1)}, a_{j}, \ldots, a_n),
 \end{split}
\end{equation*}
where we also used \eqref{mond1} along with \eqref{sonne} in the last two steps, and 
which, up to a sign convention, coincides with the expression in {\em op.~cit.}
\end{rmk}

\begin{rmk}
 At the moment, it is unclear what a necessary criterion on an arbitrary associative algebra $A$ could be in order to obtain a cooperadic structure on its Hochschild chains with values in $A$ itself, but if there {\em are} any, it is immediately clear that the cooperadic structure found on bialgebra chains as in \S\ref{gelato} cannot be generalised to any {\em bialgebroid} $(U,A)$, as a sort of bialgebra over a noncommutative ring, since the Hochschild case would be an example that could be deduced from the {\em enveloping} bialgebroid $(\Ae, A)$.
  \end{rmk}

\appendix 
\section{Full proof of Proposition \ref{miglio}}\label{appendix}

\begin{proof}[Proof of Proposition \ref{miglio}]
To start with, recall the different signs of the various tensor flips $\gs, \rho$, and $\gvr$ in \eqref{piadina}.
  Proving first Eq.~\eqref{ciambella1} is not really difficult but still long enough. Indeed, its left hand side reads, by using the explicit formul\ae{} for the cup coproduct and the cobraces:
  \begin{equation*}
    \begin{array}{rcl}
      &&    
      (\id \otimes \cup^c) \circ [-]_n^\coop
     \\[1mm]
   &
      \stackrel{\scriptscriptstyle{}}{=}
      &
(\id \otimes \cup^c) \circ \rho \circ  [-]_n
      \\[1mm]
   &
      \stackrel{\scriptscriptstyle{}}{=}
      &
      (\gvr \otimes \id) \circ (\id \otimes \rho) \circ  (\cup^c \otimes \id) \circ  [-]_n
      \\[1mm]
   &
      \stackrel{\scriptscriptstyle{\eqref{sale}, \eqref{aceto}}}{=}
      &
      \Sum_{p+q=n+1} \Sum^p_{i=1} \Sum^{p+1}_{j=1}  (-1)^{\foo (q-1)(i-1)} (\gs \otimes \id) \circ (\id \otimes \gs) \circ (\mu^c \otimes \id^{\otimes 3}) \circ (\gD_{2,j-1, 1} \otimes \id^{\otimes 2})
      \\[1mm]
      && \hspace*{8.1cm}
      \circ (\gD_{j,p+1-j, j} \otimes \id) \circ \gD_{p,q,i}
      \\[1mm]
   &
      \stackrel{\scriptscriptstyle{\eqref{sale2}}}{=}
      &
      \Sum^{n+1}_{p=1} \Sum^{p+1}_{j=1} \Sum^{p}_{i=1}  (-1)^{\foo (n-p)(i-1)} (\mu^c \otimes \id^{\otimes 3}) \circ (\id \otimes \gs \otimes \id) \circ (\id^{\otimes 2} \otimes \gs) \circ  (\gD_{2,j-1, 1} \otimes \id^{\otimes 2})
      \\[1mm]
      && \hspace*{8.1cm}
      \circ (\gD_{j,p+1-j, j} \otimes \id) \circ \gD_{p,n+1-p,i},
\\[1mm]
      &
      \stackrel{\scriptscriptstyle{}}{=}
      &
      \pig(
\Sum^{n+1}_{p=1} \Sum^{p+1}_{j=2} \Sum^{j-1}_{i=1}
+
 \Sum^{n+1}_{p=1} \Sum^{p}_{j=1} \Sum^{p}_{i=j}  
\pig) (-1)^{\foo (n-p)(i-1)} (\mu^c \otimes \id^{\otimes 3}) \circ (\id \otimes \gs \otimes \id) \circ (\id^{\otimes 2} \otimes \gs) 
      \\[1mm]
      && \hspace*{5.1cm}
\circ  (\gD_{2,j-1, 1} \otimes \id^{\otimes 2})
      \circ (\gD_{j,p+1-j, j} \otimes \id) \circ \gD_{p,n+1-p,i},
    \end{array}
\end{equation*}
whereas the second summand on the right hand side of \eqref{ciambella1} comes out as:
  \begin{equation*}
    \begin{array}{rcl}
      &&    
(\gvr \otimes \id) \circ (\id \otimes [-]^\coop) \circ \cup^c
      \\[1mm]
   &
      \stackrel{\scriptscriptstyle{}}{=}
      &
     (\gvr \otimes \id) \circ (\id \otimes \rho) \circ (\id \otimes [-]) \circ \cup^c
      \\[1mm]
    &
      \stackrel{\scriptscriptstyle{\eqref{aceto}}}{=}
      &
      \Sum^{n+1}_{j=1}  (\gvr \otimes \id) \circ (\id \otimes \rho) \circ
(\id \otimes [-]_{n+1-j})
\circ       (\mu^c \otimes \id^{\otimes 2}) \circ (\gD_{2,j-1, 1} \otimes \id) \circ \gD_{j,n+1-j, j}
      \\[1mm]
   &
      \stackrel{\scriptscriptstyle{\eqref{sale}}}{=}
      &
      \Sum^{n+1}_{j=1}  \Sum_{p+q=n+2-j} \Sum^p_{i=1} (-1)^{\foo (q-1)(p-i)}
 (\mu^c \otimes \id^{\otimes 3}) \circ
      (\id \otimes \gvr \otimes \id) \circ (\id^{\otimes 2} \otimes \rho)
      \circ
      (\id^{\otimes 2} \otimes \gD_{p,q,i})
      \\[1mm]
      && \hspace*{8.1cm}
\circ       (\gD_{2,j-1, 1} \otimes \id) \circ \gD_{j,n+1-j, j}
      \\[1mm]
   &
      \stackrel{\scriptscriptstyle{\eqref{sale2}}}{=}
      &
      \Sum^{n+1}_{j=1}  \Sum^{n+2-j}_{p=1} \Sum^p_{i=1} (-1)^{\foo (n+1-j-p)(p-i)}
 (\mu^c \otimes \id^{\otimes 3}) \circ
      (\id \otimes \gvr \otimes \id) \circ (\id^{\otimes 2} \otimes \rho)
            \circ
      (\gD_{2,j-1, 1} \otimes \id^{\otimes 2})
      \\[1mm]
      && \hspace*{8.1cm}
      \circ       
(\id \otimes \gD_{p,n+2-j-p,i})
      \circ \gD_{j,n+1-j, j}
      \\[1mm]
   &
      \stackrel{\scriptscriptstyle{\eqref{orzo}}}{=}
      &
      \Sum^{n+1}_{j=1}  \Sum^{n+2-j}_{p=1} \Sum^p_{i=1} (-1)^{\foo (n+1-j-p)(p-i)}
 (\mu^c \otimes \id^{\otimes 3}) \circ
(\id \otimes \gvr \otimes \id) \circ (\id^{\otimes 2} \otimes \rho)
      \circ
      (\gD_{2,j-1, 1} \otimes \id^{\otimes 2})
      \\[1mm]
      && \hspace*{8.1cm}
      \circ       
(\gD_{j,p,j} \otimes \id)
      \circ \gD_{p+j-1,n+2-j-p, i+j-1}
      \\[1mm]
 &
      \stackrel{\scriptscriptstyle{}}{=}
      &
      \Sum^{n+1}_{j=1}  \Sum^{n+1}_{p=j} \Sum^{p-j+1}_{i=1} (-1)^{\foo (n-p)(p-i-j+1)}
      (\mu^c \otimes \id^{\otimes 3}) \circ
      (\id \otimes \gvr \otimes \id) \circ (\id^{\otimes 2} \otimes \rho)
            \circ
      (\gD_{2,j-1, 1} \otimes \id^{\otimes 2})
      \\[1mm]
      && \hspace*{8.1cm}
      \circ       
(\gD_{j,p+1-j,j} \otimes \id)
      \circ \gD_{p,n+1-p, i+j-1}
      \\[1mm]
 &
      \stackrel{\scriptscriptstyle{}}{=}
      &
      \Sum^{n+1}_{j=1}  \Sum^{n+1}_{p=j} \Sum^{p}_{i=j} (-1)^{\foo (n-p)(p-i)}
 (\mu^c \otimes \id^{\otimes 3}) \circ
(\id \otimes \gvr \otimes \id) \circ (\id^{\otimes 2} \otimes \rho)
      \circ
      (\gD_{2,j-1, 1} \otimes \id^{\otimes 2})
      \\[1mm]
      && \hspace*{8.1cm}
      \circ       
(\gD_{j,p+1-j,j} \otimes \id)
      \circ \gD_{p,n+1-p, i}
      \\[1mm]
       &
      \stackrel{\scriptscriptstyle{}}{=}
      &
      \Sum^{n+1}_{p=1}  \Sum^{p}_{j=1} \Sum^{p}_{i=j} (-1)^{\foo (n-p)(i-1)}
 (\mu^c \otimes \id^{\otimes 3}) \circ
      (\id \otimes \gs \otimes \id) \circ (\id^{\otimes 2} \otimes \gs) \circ
      (\gD_{2,j-1, 1} \otimes \id^{\otimes 2})
      \\[1mm]
      && \hspace*{8.1cm}
      \circ       
(\gD_{j,p+1-j,j} \otimes \id)
      \circ \gD_{p,n+1-p, i},
      \end{array}
\end{equation*}
where we re-indexed $p \mapsto p+j-1$ in the sixth step and $i \mapsto i+j-1$ in the seventh, and which is the second triple sum in the above expression for 
$(\id \otimes \cup^c) \circ [-]_n^\coop$.
On the other hand, the first summand on the right hand side of \eqref{ciambella1} results into:
  \begin{equation*}
    \begin{array}{rcl}
      &&    
([-]^\coop \otimes \id) \circ \cup^c
      \\[1mm]
   &
      \stackrel{\scriptscriptstyle{}}{=}
      &
      (\rho \otimes \id) \circ ([-] \otimes \id) \circ \cup^c
      \\[1mm]
    &
      \stackrel{\scriptscriptstyle{\eqref{aceto}}}{=}
      &
      \Sum^{n+1}_{j=1}  (\rho \otimes \id) \circ
([-]_{j-1} \otimes \id)
\circ       (\mu^c \otimes \id^{\otimes 2}) \circ (\gD_{2,j-1, 1} \otimes \id) \circ \gD_{j,n+1-j, j}
      \\[1mm]
   &
      \stackrel{\scriptscriptstyle{\eqref{sale}}}{=}
      &
      \Sum^{n+1}_{j=1}  \Sum_{p+q=j} \Sum^p_{i=1} (-1)^{\foo (q-1)(i-1)}
 (\mu^c \otimes \id^{\otimes 3}) \circ
      (\id \otimes \gs \otimes \id) \circ
      (\id \otimes \gD_{p,q,i} \otimes \id)
      \\[1mm]
      && \hspace*{8.5cm}
\circ       (\gD_{2,j-1, 1} \otimes \id) \circ \gD_{j,n+1-j, j}
      \\[1mm]
    &
      \stackrel{\scriptscriptstyle{\eqref{sale2}}}{=}
      &
      \Sum^{n+1}_{j=1}  \Sum^{j}_{p=1} \Sum^p_{i=1} (-1)^{\foo (j-p-1)(i-1)}
 (\mu^c \otimes \id^{\otimes 3}) \circ
      (\id \otimes \gs \otimes \id) \circ
      (\id \otimes \gD_{p,j-p,i} \otimes \id)
      \\[1mm]
     && \hspace*{8.5cm}
 \circ       (\gD_{2,j-1, 1} \otimes \id) \circ \gD_{j,n+1-j, j}
      \\[1mm]
    &
      \stackrel{\scriptscriptstyle{\eqref{orzo}}}{=}
      &
            \Sum^{n+1}_{j=1}  \Sum^{j}_{p=1} \Sum^p_{i=1} (-1)^{\foo (j-p-1)(i-1)}
 (\mu^c \otimes \id^{\otimes 3}) \circ
      (\id \otimes \gs \otimes \id) \circ
      (\gD_{2,p,1} \otimes \id^{\otimes 2})
    \\[1mm]
     && \hspace*{8.5cm}
\circ       (\gD_{p+1,j-p, i} \otimes \id) \circ \gD_{j,n+1-j, j}
      \\[1mm]
    &
      \stackrel{\scriptscriptstyle{\eqref{orzo}}}{=}
      &
       \Sum^{n+1}_{j=1}  \Sum^{j}_{p=1} \Sum^p_{i=1} (-1)^{\foo (j-p-1)(i-1)}
 (\mu^c \otimes \id^{\otimes 3}) \circ
      (\id \otimes \gs \otimes \id) \circ (\id^{\otimes 2} \otimes \gs) \circ
      (\gD_{2,p,1} \otimes \id^{\otimes 2})
      \\[1mm]
     && \hspace*{7.1cm}
\circ       (\gD_{p+1,n+1-j, p+1} \otimes \id) \circ \gD_{n+p+1-j,j-p, i}
      \\[1mm]
    &
      \stackrel{\scriptscriptstyle{}}{=}
      &
       \Sum^{n+1}_{p=1}  \Sum^{p}_{j=1} \Sum^j_{i=1} (-1)^{\foo (p-j-1)(i-1)}
 (\mu^c \otimes \id^{\otimes 3}) \circ
      (\id \otimes \gs \otimes \id) \circ (\id^{\otimes 2} \otimes \gs) \circ
      (\gD_{2,j,1} \otimes \id^{\otimes 2})
      \\[1mm]
     && \hspace*{7.1cm}
\circ       (\gD_{j+1,n+1-p, j+1} \otimes \id) \circ \gD_{n+j+1-p,p-j, i}
      \\[1mm]
    &
      \stackrel{\scriptscriptstyle{}}{=}
      &
       \Sum^{n+1}_{j=1}  \Sum^{n+1}_{p=j} \Sum^j_{i=1} (-1)^{\foo (p-j-1)(i-1)}
 (\mu^c \otimes \id^{\otimes 3}) \circ
      (\id \otimes \gs \otimes \id) \circ (\id^{\otimes 2} \otimes \gs) \circ
      (\gD_{2,j,1} \otimes \id^{\otimes 2})
      \\[1mm]
     && \hspace*{7.1cm}
\circ       (\gD_{j+1,n+1-p, j+1} \otimes \id) \circ \gD_{n+j+1-p,p-j, i}
\\[1mm]
    &
      \stackrel{\scriptscriptstyle{}}{=}
      &
       \Sum^{n+1}_{j=1}  \Sum^{n+1}_{p=j} \Sum^j_{i=1} (-1)^{\foo (n-p)(i-1)}
 (\mu^c \otimes \id^{\otimes 3}) \circ
      (\id \otimes \gs \otimes \id) \circ (\id^{\otimes 2} \otimes \gs) \circ
      (\gD_{2,j,1} \otimes \id^{\otimes 2})
      \\[1mm]
     && \hspace*{8.1cm}
\circ       (\gD_{j+1,p-j, j+1} \otimes \id) \circ \gD_{p,n-p+1, i}
      \\[1mm]
    &
      \stackrel{\scriptscriptstyle{}}{=}
      &
       \Sum^{n+1}_{p=1}  \Sum^{p+1}_{j=2} \Sum^{j-1}_{i=1} (-1)^{\foo (n-p)(i-1)}
 (\mu^c \otimes \id^{\otimes 3}) \circ
      (\id \otimes \gs \otimes \id) \circ (\id^{\otimes 2} \otimes \gs) \circ
      (\gD_{2,j-1,1} \otimes \id^{\otimes 2})
      \\[1mm]
     && \hspace*{8.1cm}
\circ       (\gD_{j,p+1-j, j} \otimes \id) \circ \gD_{p,n+1-p, i},
      \\[1mm]
    \end{array}
\end{equation*}
where in the seventh step we simply exchanged the r\^oles of $j$ and $p$, in the ninth re-indexed $p \mapsto n-p+j+1$, and in the last step re-indexed $j \mapsto j-1$, along with the customary double sum yoga. We thus obtain the first triple sum in the expression of 
$(\id \otimes \cup^c) \circ [-]_n^\coop$ as well,
and hence Eq.~\eqref{ciambella1} is proven.
As for proving the identity \eqref{ciambella2}, let us start by writing its first line, {\em i.e.}, the terms that constitute the coLeibniz rule. 
The first one on the right hand side, by using the relations \eqref{orzo}--\eqref{orzo4}, can be transformed into:
\begin{equation*}
    \begin{array}{rcl}
      &&    
([-] \otimes \id) \circ \cup^c
      \\[1mm]
   &
      \stackrel{\scriptscriptstyle{\eqref{aceto}, \eqref{sale2}}}{=}
      &
      \Sum^{n+1}_{j=1} \Sum^{j}_{p=1} \Sum^{p}_{i=1} 
      (-1)^{\foo (j-p-1)(p-i)}
      (\mu^c \otimes \id^{\otimes 3}) \circ (\id \otimes \Delta_{p, j-p, i} \otimes \id) \circ (\Delta_{2, j-1, 1} \otimes \id) \circ \gD_{j,n+1-j,j}
      \\[1mm]
   &
      \stackrel{\scriptscriptstyle{\eqref{orzo}}}{=}
      &
      \Sum^{n+1}_{p=1} \Sum^{n+1}_{j=p} \Sum^{p}_{i=1} 
      (-1)^{\foo (j-p-1)(p-i)}
      (\mu^c \otimes \id^{\otimes 3}) \circ (\Delta_{2, p, 1} \otimes \id^{\otimes 2}) \circ (\Delta_{p+1, j-p, i} \otimes \id) \circ \gD_{j,n+1-j,j}
\\[1mm]
      &
      \stackrel{\scriptscriptstyle{}}{=:}
      &
       (1)
    \end{array}
  \end{equation*}
on $\cC(n)$,
while the second term on the right hand side becomes:
\begin{equation*}
    \begin{array}{rcl}
      &&    
(\gvr \otimes \id) \circ (\id \otimes [-]) \circ \cup^c
      \\[1mm]
       &
      \stackrel{\scriptscriptstyle{\eqref{aceto}, \eqref{sale2}}}{=}
      &
      \Sum^{n+1}_{j=1} \Sum^{n+2-j}_{p=1} \Sum^{p}_{i=1} 
      (-1)^{\foo (n+1-j-p)(p-i)}
      (\mu^c \otimes \id^{\otimes 3}) \circ
( \id \otimes \gvr \otimes \id)
      \circ
      (\id^{\otimes 2} \otimes \Delta_{p, n+2-j-p, i})
      \\[1mm]
      &&
      \hspace*{8cm}
      \circ (\Delta_{2, j-1, 1} \otimes \id) \circ \gD_{j,n+1-j,j}
                 \end{array}
    \end{equation*}

    
       \begin{equation*}
    \begin{array}{rcl}
      &
      \stackrel{\scriptscriptstyle{}}{=}
      &
      \Sum^{n+1}_{j=1} \Sum^{n+2-j}_{p=1} \Sum^{p}_{i=1} 
      (-1)^{\foo (n+1-j-p)(p-i)}
      (\mu^c \otimes \id^{\otimes 3}) \circ
( \id \otimes \gvr \otimes \id)
      \circ (\Delta_{2, j-1, 1} \otimes \id^{\otimes 2})
      \\[1mm]
      &&
      \hspace*{8cm}
            \circ
      (\id \otimes \Delta_{p, n+2-j-p, i}) \circ 
      \gD_{j,n+1-j,j}
      \\[1mm]
   &
      \stackrel{\scriptscriptstyle{\eqref{orzo}}}{=}
      &
      \Sum^{n+1}_{j=1} \Sum^{n+2-j}_{p=1} \Sum^{p}_{i=1} 
      (-1)^{\foo (n+1-j-p)(p-i)+(j-1)(p-1)}
      (\mu^c \otimes \id^{\otimes 3}) \circ
( \id \otimes \gs \otimes \id)
      \circ (\Delta_{2, j-1, 1} \otimes \id^{\otimes 2})
      \\[1mm]
      &&
      \hspace*{7.3cm}
            \circ
      (\Delta_{j, p, j} \otimes \id) \circ 
      \gD_{p+j-1,n+2-j-p,i+j-1}
      \\[1mm]
   &
      \stackrel{\scriptscriptstyle{\eqref{orzo}}}{=}
      &
      \Sum^{n+1}_{j=1} \Sum^{n+2-j}_{p=1} \Sum^{p}_{i=1} 
      (-1)^{\foo (n+1-j-p)(p-i)+(j-1)(p-1)}
      (\mu^c \otimes \id^{\otimes 3}) 
      \circ (\Delta_{2, p, 2} \otimes \id^{\otimes 2})
\\[1mm]
      &&
      \hspace*{6.3cm}
       \circ
            (\Delta_{p+1, j-1, 1} \otimes \id)
                 \circ 
      \gD_{p+j-1,n+2-j-p,i+j-1}
      \\[1mm]
   &
      \stackrel{\scriptscriptstyle{}}{=}
      &
      \Sum^{n+1}_{p=1} \Sum^{n+1}_{j=p} \Sum^{p}_{i=1} 
      (-1)^{\foo (n-j)(i-1)+j(p-1)}
      (\mu^c \otimes \id^{\otimes 3}) \circ
       (\Delta_{2, p, 2} \otimes \id^{\otimes 2})
            \circ
            (\Delta_{p+1, j-p, 1} \otimes \id)
 \\[1mm]
      &&
      \hspace*{10.3cm}
                 \circ 
            \gD_{j,n+1-j,j+1-i}
\\[1mm]
            &
      \stackrel{\scriptscriptstyle{}}{=:}
      &
             (2)
          \end{array}
  \end{equation*}
on $\cC(n)$,
where we re-indexed $j \mapsto j+1-p$ and $i \mapsto p+1-i$ in the last step (plus the standard double sum operations).
On the left hand side in the first row of \eqref{ciambella2}, we have that
\begin{equation*}
    \begin{array}{rcl}
      &&    
(\id \otimes \cup^c) \circ [-]_n
      \\[1mm]
   &
      \stackrel{\scriptscriptstyle{\eqref{aceto}, \eqref{sale2}}}{=}
      &
      \Sum^{n+1}_{p=1} \Sum^{p}_{i=1} \Sum^{n+2-p}_{j=1} 
      (-1)^{\foo (n-p)(p-i)}
      (\id \otimes \mu^c \otimes \id^{\otimes 2}) \circ (\id \otimes \Delta_{2, j-1, 1} \otimes \id)
      \circ (\id \otimes \Delta_{j, n+2-p-j, j})
    \\[1mm]
      &&
      \hspace*{11.2cm}
     \circ \gD_{p,n+1-p,i}
      \\[1mm]
   &
      \stackrel{\scriptscriptstyle{\eqref{orzo}}}{=}
      &
      \Sum^{n+1}_{p=1} \Sum^{p}_{i=1} \Sum^{n+2-p}_{j=1} 
      (-1)^{\foo (n-p)(p-i)}
      (\id \otimes \mu^c \otimes \id^{\otimes 2}) \circ (\id \otimes \Delta_{2, j-1, 1} \otimes \id)
        \circ (\Delta_{p, j, i} \otimes \id)
    \\[1mm]
      &&
      \hspace*{9.8cm}
      \circ \gD_{p+j-1,n+2-p-j,j+i-1}
      \\[1mm]
   &
      \stackrel{\scriptscriptstyle{\eqref{orzo}}}{=}
      &
      \Sum^{n+1}_{p=1} \Sum^{p}_{i=1} \Sum^{n+2-p}_{j=1} 
      (-1)^{\foo (n-p)(p-i)}
      (\id \otimes \mu^c \otimes \id^{\otimes 2}) \circ (\Delta_{p, 2, i} \otimes \id^{\otimes 2})
      \circ (\Delta_{p+1, j-1, i} \otimes \id)
    \\[1mm]
      &&
      \hspace*{9.8cm}
      \circ \gD_{p+j-1,n+2-p-j,j+i-1}
      \\[1mm]
   &
      \stackrel{\scriptscriptstyle{}}{=}
      &
      \Sum^{n+1}_{p=1} \Sum^{n+1}_{j=p} \Sum^{p}_{i=1}  
      (-1)^{\foo (n-p)(p-i)}
      (\id \otimes \mu^c \otimes \id^{\otimes 2}) \circ (\Delta_{p, 2, i} \otimes \id^{\otimes 2}) \circ (\Delta_{p+1, j-p, i} \otimes \id)
  \\[1mm]
      &&
      \hspace*{10.8cm}
        \circ \gD_{j,n+1-j,j-p+i}
    \\[1mm]
   &
      \stackrel{\scriptscriptstyle{}}{=:}
      &
  (3)
    \end{array}
  \end{equation*}
on $\cC(n)$ again, where here we re-indexed $j \mapsto p+j-1$ in the last step. Next, we will show that these terms correspond to certain summands of the term $(d \otimes \id^{\otimes 2}) \circ F$. To this end, note first that
for this computation $F$ can be more conveniently rewritten the following way:
   \begin{equation}
    \label{pici2}
    F = \Sum^{n+1}_{p=1} \Sum^{n+1}_{r=p} \Sum^{p}_{i=1} \Sum^{i}_{j=1}
    (-1)^{\foo (-1)^{\foo (p+q)j + (p+r-1)i+ n}}
    (\gD_{p+1,r-p,p-i+1} \otimes \id) \circ \gD_{r, n-r+1, r-j+1},
    \end{equation}
 seen  as a map $\cC(n) \to \cC(p+1) \otimes \cC(r-p) \otimes \cC(n-r+1)$, as
a
   little multiple sum yoga reveals.
   Note that the sign is the same as in the slightly longer expression
   $(-1)^{\foo (n-r)(j-1) + (r-p-1)(i-1)+ p}$, which, however, is sometimes more useful to consider in subsequent calculations: with respect to the faces from \eqref{riso1}--\eqref{riso3}, we have on $\cC(n)$:
\begin{equation*}
    \begin{array}{rcl}
      &&    
(d_{\rm last} \otimes \id^{\otimes 2}) \circ F
      \\[1mm]
   &
      \stackrel{\scriptscriptstyle{\eqref{pici2}, \eqref{riso3}}}{=}
      &
      \Sum^{n+1}_{p=1} \Sum^{n+1}_{r=p} \Sum^{p}_{i=1} \Sum^{i}_{j=1} 
      (-1)^{\foo (n-r)(j-1) + (r-p-1)(i-1)+ 1}
      (\mu^c \otimes \id^{\otimes 3}) \circ (\Delta_{2, p, 1} \otimes \id^{\otimes 2})
  \\[1mm]
      &&
      \hspace*{6.8cm}
      \circ (\Delta_{p+1, r-p, p+1-i} \otimes \id) \circ \gD_{r,n+1-r,r+1-j}
                \end{array}
    \end{equation*}

    
       \begin{equation*}
    \begin{array}{rcl}
       &
      \stackrel{\scriptscriptstyle{}}{=}
      &
      \Sum^{n+1}_{p=1} \Sum^{n+1}_{r=p} \Sum^{p}_{i=1}
 (-1)^{\foo (r-p-1)(p-i)+ 1}
      (\mu^c \otimes \id^{\otimes 3}) \circ (\Delta_{2, p, 1} \otimes \id^{\otimes 2})
 \\[1mm]
      &&
      \hspace*{7.8cm}
       \circ (\Delta_{p+1, r-p, i} \otimes \id) \circ \gD_{r,n+1-r,r}
      \\[1mm]
      &&
       +
\Sum^{n+1}_{p=2} \Sum^{n+1}_{r=p} \Sum^{p}_{i=2} \Sum^{i}_{j=2}
      (-1)^{\foo (n-r)(j-1) + (r-p-1)(i-1)+ 1}
(\mu^c \otimes \id^{\otimes 3}) \circ (\Delta_{2, p, 1} \otimes \id^{\otimes 2})
 \\[1mm]
      &&
      \hspace*{6.8cm}
 \circ (\Delta_{p+1, r-p, p+1-i} \otimes \id) \circ \gD_{r,n+1-r,r+1-j}
\\[1mm]
&
      \stackrel{\scriptscriptstyle{}}{=:}
      &
       (1') + (4),
    \end{array}
\end{equation*}
where we re-indexed $i \mapsto p+1-i$ in the second step in the first sum. One notes that $(1')$ equals minus $(1)$ (by changing names of summation indices), and hence cancels out in \eqref{ciambella2}. Next, we have
\begin{equation*}
    \begin{array}{rcl}
      &&    
(d_{0} \otimes \id^{\otimes 2}) \circ F
      \\[1mm]
   &
      \stackrel{\scriptscriptstyle{\eqref{pici2}, \eqref{riso1}}}{=}
      &
      \Sum^{n+1}_{p=1} \Sum^{n+1}_{r=p} \Sum^{p}_{i=1} \Sum^{i}_{j=1} 
      (-1)^{\foo (n-r)(j-1) + (r-p-1)(i-1)+ p}
      (\mu^c \otimes \id^{\otimes 3}) \circ (\Delta_{2, p, 2} \otimes \id^{\otimes 2})
 \\[1mm]
      &&
      \hspace*{6.8cm}
      \circ (\Delta_{p+1, r-p, p-i+1} \otimes \id) \circ \gD_{r,n+1-r,r+1-j}
\\[1mm]
      &
      \stackrel{\scriptscriptstyle{}}{=}
      &
      \Sum^{n+1}_{p=1} \Sum^{n+1}_{r=p} \Sum^{p}_{j=1}
 (-1)^{\foo (n-r)(j-1) + r(p-1) + 1}
      (\mu^c \otimes \id^{\otimes 3}) \circ (\Delta_{2, p, 2} \otimes \id^{\otimes 2})
 \\[1mm]
      &&
      \hspace*{6.8cm}
      \circ (\Delta_{p+1, r-p, 1} \otimes \id) \circ \gD_{r,n+1-r,r+1-j}
\\[1mm]
      &&
       +
\Sum^{n+1}_{p=2} \Sum^{n+1}_{r=p} \Sum^{p-1}_{i=1} \Sum^{i}_{j=1}
      (-1)^{\foo (n-r)(j-1) + (r-p-1)(i-1)+ p}
(\mu^c \otimes \id^{\otimes 3}) \circ (\Delta_{2, p, 2} \otimes \id^{\otimes 2})
 \\[1mm]
      &&
      \hspace*{6.8cm}
\circ (\Delta_{p+1, r-p, p-i+1} \otimes \id) \circ \gD_{r,n+1-r,r+1-j}
      \\[1mm]
&
      \stackrel{\scriptscriptstyle{}}{=:}
      &
      (2') + (5),
    \end{array}
\end{equation*}
and one notes that $(2')$ equals minus $(2)$, again by changing names of summation indices. To end this part with respect to the intermediate faces, let us compute:
\begin{equation*}
    \begin{array}{rcl}
      &&    
\Sum^{{\rm max}-1}_{k=1} (-1)^k (d_k \otimes \id^{\otimes 2}) \circ F
      \\[1mm]
   &
      \stackrel{\scriptscriptstyle{\eqref{pici2}, \eqref{riso2}}}{=}
      &
      \Sum^{n+1}_{p=1} \Sum^{n+1}_{r=p} \Sum^{p}_{i=1} \Sum^{i}_{j=1} \Sum^{p}_{k=1} 
      (-1)^{\foo (n-r)(j-1) + (r-p-1)(i-1)+ p+k}
      (\id \otimes \mu^c \otimes \id^{\otimes 2}) \circ (\Delta_{p, 2, k} \otimes \id^{\otimes 2}) 
      \\[1mm]
      &&
      \hspace*{6.5cm}
 \circ (\Delta_{p+1, r-p, p-i+1} \otimes \id)     \circ \gD_{r,n+1-r,r+1-j}
\\[1mm]
&
      \stackrel{\scriptscriptstyle{}}{=}
      &
      \Sum^{n+1}_{p=1} \Sum^{n+1}_{r=p}  \Sum^{p}_{i=1} \Sum^{p}_{j=i} \Sum^{p}_{k=1}  
      (-1)^{\foo (n-r)(p-j) + (r-p-1)(p-i)+ p+k}
      (\id \otimes \mu^c \otimes \id^{\otimes 2}) \circ (\Delta_{p, 2, k} \otimes \id^{\otimes 2}) 
      \\[1mm]
      &&
      \hspace*{7cm}
   \circ (\Delta_{p+1, r-p, i} \otimes \id)   \circ \gD_{r,n+1-r,r+j-p}
\\[1mm]
&
      \stackrel{\scriptscriptstyle{}}{=}
      &
      \pig(
      \Sum^{n+1}_{p=2} \Sum^{n+1}_{r=p}  \Sum^{p}_{k=2} \Sum^{k-1}_{i=1} \Sum^{p}_{j=i} 
      +
\Sum^{n+1}_{p=2} \Sum^{n+1}_{r=p}  \Sum^{p}_{k=1} \Sum^{p}_{i=k+1} \Sum^{p}_{j=i} 
      \pig)
      (-1)^{\foo (n-r)(p-j) + (r-p-1)(p-i)+ p+k}
   \\[3mm]
      &&
      \hspace*{2.5cm}
      (\id \otimes \mu^c \otimes \id^{\otimes 2}) \circ (\Delta_{p, 2, k} \otimes \id^{\otimes 2})
      \circ (\Delta_{p+1, r-p, i} \otimes \id)
      \circ \gD_{r,n+1-r,r+j-p}
   \\[1mm]
      &&
   +
\Sum^{n+1}_{p=1} \Sum^{n+1}_{r=p}  \Sum^{p}_{k=1} \Sum^{p}_{j=k}      
      (-1)^{\foo (n-r)(p-j) + (r-p)(p-k)}
      (\id \otimes \mu^c \otimes \id^{\otimes 2}) \circ (\Delta_{p, 2, k} \otimes \id^{\otimes 2})
   \\[1mm]
      &&
      \hspace*{7.5cm}
         \circ (\Delta_{p+1, r-p, k} \otimes \id)
         \circ \gD_{r,n+1-r,r+j-p}
         \\[1mm]
         &
         \stackrel{\scriptscriptstyle{}}{=}
         &
      \pig(
      \Sum^{n+1}_{p=2} \Sum^{n+1}_{r=p}  \Sum^{p}_{k=2} \Sum^{k-1}_{i=1} \Sum^{p}_{j=i} 
      +
\Sum^{n+1}_{p=2} \Sum^{n+1}_{r=p}  \Sum^{p-1}_{k=1} \Sum^{p}_{i=k+1} \Sum^{p}_{j=i} 
      \pig)
      (-1)^{\foo (n-r)(p-j) + (r-p-1)(p-i)+ p+k}
   \\[3mm]
      &&
   \hspace*{2.5cm}
     (\id \otimes \mu^c \otimes \id^{\otimes 2}) \circ (\Delta_{p, 2, k} \otimes \id^{\otimes 2})
         \circ (\Delta_{p+1, r-p, i} \otimes \id)
      \circ \gD_{r,n+1-r,r+j-p}
 \\[1mm]
       &&
   +
\Sum^{n+1}_{p=2} \Sum^{n+1}_{r=p}  \Sum^{p-1}_{k=1} \Sum^{p}_{j=k+1}      
      (-1)^{\foo (n-r)(p-j) + (r-p)(p-k)}
      (\id \otimes \mu^c \otimes \id^{\otimes 2}) \circ (\Delta_{p, 2, k} \otimes \id^{\otimes 2})
   \\[1mm]
      &&
      \hspace*{7.5cm}
         \circ (\Delta_{p+1, r-p, k} \otimes \id)
      \circ \gD_{r,n+1-r,r+j-p}   
             \end{array}
    \end{equation*}

    
       \begin{equation*}
    \begin{array}{rcl}
       &&
   +
\Sum^{n+1}_{p=1} \Sum^{n+1}_{r=p}  \Sum^{p}_{k=1}       
      (-1)^{\foo (n-p)(p-k)}
      (\id \otimes \mu^c \otimes \id^{\otimes 2}) \circ (\Delta_{p, 2, k} \otimes \id^{\otimes 2})
   \\[1mm]
      &&
      \hspace*{7.5cm}
         \circ (\Delta_{p+1, r-p, k} \otimes \id)
      \circ \gD_{r,n+1-r,r+k-p}   
         \\[1mm]
         &
      \stackrel{\scriptscriptstyle{}}{=:}
      &
       (6) + (7) + (8) + (3'),
    \end{array}
\end{equation*}
where we re-indexed $i \mapsto p+1-i$ and $j \mapsto p+1-j$ in the second step. Here, one observes that $(3)$ and $(3')$ are the same on both sides of Eq.~\eqref{ciambella2}.

At this point, we are left with showing that the terms $(4), (5), (6), (7)$, and $(8)$ cancel with $(\id \otimes d \otimes \id) \circ F$, along with $(\id^{\otimes 2} \otimes d) \circ F$ and $F \circ d$ in \eqref{ciambella2}, with their respective signs.
Indeed, by repeatedly applying \eqref{orzo}, let us compute for the last face $d_n$ on $\cC(n)$:
\begin{equation*}
    \begin{array}{rcl}
      &&    
F \circ d_n
      \\[1mm]
   &
      \stackrel{\scriptscriptstyle{\eqref{pici2}, \eqref{riso3}}}{=}
      &
      \Sum^{n}_{p=1} \Sum^{n}_{r=p} \Sum^{p}_{i=1} \Sum^{i}_{j=1} 
      (-1)^{\foo (n-1-r)(j-1) + (r-p-1)(i-1)+ p+n}
      (\mu^c \otimes \id^{\otimes 3}) \circ (\id \otimes \Delta_{p+1, r-p, p+1-i} \otimes \id) 
  \\[1mm]
      &&
      \hspace*{8.5cm}
    \circ (\id \otimes \Delta_{r, n-r, r+1-j})   \circ \gD_{2,n-1,1}
\\[1mm]
      &
      \stackrel{\scriptscriptstyle{\eqref{orzo}}}{=}
      &
      \Sum^{n}_{p=1} \Sum^{n}_{r=p} \Sum^{p}_{i=1} \Sum^{i}_{j=1} 
      (-1)^{\foo (n-1-r)(j-1) + (r-p-1)(i-1)+ p+n}
      (\mu^c \otimes \id^{\otimes 3}) \circ (\Delta_{2, p+1, 1} \otimes \id^{\otimes 2}) 
  \\[1mm]
      &&
  \hspace*{6.5cm}
  \circ (\Delta_{p+2, r-p, p+1-i} \otimes \id)
       \circ \gD_{r+1,n-r,r+1-j}
\\[1mm]
      &
      \stackrel{\scriptscriptstyle{}}{=}
&
      \Sum^{n+1}_{p=2} \Sum^{n+1}_{r=p} \Sum^{p}_{i=2} \Sum^{i}_{j=2} 
      (-1)^{\foo (n-r)j + (r-p-1)i+ p+n+1}
      (\mu^c \otimes \id^{\otimes 3}) \circ (\Delta_{2, p, 1} \otimes \id^{\otimes 2}) 
  \\[1mm]
      &&
      \hspace*{6.5cm}
\circ (\Delta_{p+1, r-p, p+1-i} \otimes \id)
      \circ \gD_{r,n+1-r,r+1-j},
\end{array}
\end{equation*} 
where we re-indexed $p \mapsto p+1$, $r \mapsto r+1$, $i \mapsto i+1$, and $j \mapsto j+1$ in the last step. One directly notices that this equals minus the term $(4)$ obtained before.
    The same way, one shows that
$
F \circ d_0
$
equals minus $(5)$, which we omit.

In order to continue, consider
\begin{equation*}
    \begin{array}{rcl}
      &&    
- (\tau \otimes \id) \circ (d_{\rm last} \otimes \id^{\otimes 2}) \circ (\tau \otimes \id) \circ F
      \\[1mm]
   &
      \stackrel{\scriptscriptstyle{\eqref{pici2}, \eqref{riso3}}}{=}
      &
      \Sum^{n}_{p=1} \Sum^{n+1}_{r=p+1} \Sum^{p}_{i=1} \Sum^{i}_{j=1} 
      (-1)^{\foo (n-r)(j-1)+ (r-p-1)i + 1}
      (\id \otimes \mu^c \otimes \id^{\otimes 2}) \circ (\id \otimes \Delta_{2, r-p-1, 1} \otimes \id)
  \\[1mm]
      &&
      \hspace*{7cm}
     \circ (\Delta_{p+1, r-p, p-i+1} \otimes \id)
       \circ \gD_{r,n+1-r,r+1-j}
\\[1mm]
      &
      \stackrel{\scriptscriptstyle{\eqref{orzo}}}{=}
      &
      \Sum^{n}_{p=1} \Sum^{n+1}_{r=p+1} \Sum^{p}_{i=1} \Sum^{i}_{j=1} 
      (-1)^{\foo (n-r)(j-1)+ (r-p-1)i +1}
      (\id \otimes \mu^c \otimes \id^{\otimes 2}) \circ (\Delta_{p+1, 2, p+1-i} \otimes \id^{\otimes 2})
  \\[1mm]
      &&
      \hspace*{6.7cm}
      \circ (\Delta_{p+2, r-p-1, p+1-i} \otimes \id)
       \circ \gD_{r,n+1-r,r+1-j}
\\[1mm]
      &
      \stackrel{\scriptscriptstyle{}}{=}
      &
      \Sum^{n+1}_{p=2} \Sum^{n+1}_{r=p} \Sum^{p-1}_{i=1} \Sum^{p-i}_{j=1} 
      (-1)^{\foo (n-r)(j-1)+ (r-p)(p-i) +1}
      (\id \otimes \mu^c \otimes \id^{\otimes 2}) \circ (\Delta_{p, 2, i} \otimes \id^{\otimes 2})
 \\[1mm]
      &&
     \hspace*{7.5cm}
      \circ (\Delta_{p+1, r-p, i} \otimes \id)
       \circ \gD_{r,n+1-r,r+1-j}
\\[1mm]
      &
      \stackrel{\scriptscriptstyle{}}{=}
      &
      \Sum^{n+1}_{p=2} \Sum^{n+1}_{r=p} \Sum^{p-1}_{i=1} \Sum^{p}_{j=i+1} 
      (-1)^{\foo (n-r)(p-j)+ (r-p)(p-i) +1}
      (\id \otimes \mu^c \otimes \id^{\otimes 2}) \circ (\Delta_{p, 2, i} \otimes \id^{\otimes 2})
 \\[1mm]
      &&
     \hspace*{7.5cm}
      \circ (\Delta_{p+1, r-p, i} \otimes \id)
      \circ \gD_{r,n+1-r,r+j-p}
         \\[1mm]
         &
      \stackrel{\scriptscriptstyle{}}{=:}
      &
       (8'),
      \end{array}
\end{equation*} 
where in the third step we re-indexed first $p \mapsto p+1$ and afterwards $i \mapsto p-i$, whereas in the fourth step we re-indexed $j \mapsto p-j+1$. One notices that this is minus the term $(8)$ above, by renominating summation indices. With analogous steps just performed, we obtain:
\begin{equation*}
    \begin{array}{rcl}
      &&
-      (\tau \otimes \id) \circ (d_0 \otimes \id^{\otimes 2}) \circ (\tau \otimes \id) \circ F
      \\[1mm]
   &
      \stackrel{\scriptscriptstyle{\eqref{pici2}, \eqref{riso1}, \eqref{orzo}}}{=}
      &
      \Sum^{n+1}_{p=2} \Sum^{n+1}_{r=p} \Sum^{p-1}_{i=1} \Sum^{p}_{j=i+1} 
      (-1)^{\foo (n-r)(p-j)+ (r-p)(p-i-1)}
      (\id \otimes \mu^c \otimes \id^{\otimes 2}) \circ (\Delta_{p, 2, i} \otimes \id^{\otimes 2})
 \\[1mm]
      &&
     \hspace*{6.5cm}
      \circ (\Delta_{p+1, r-p, i+1} \otimes \id)
      \circ \gD_{r,n+1-r,r+j-p}
         \\[1mm]
         &
      \stackrel{\scriptscriptstyle{}}{=:}
      &
       (7'''),
      \end{array}
\end{equation*} 
while we can further decompose:
\begin{equation*}
    \begin{array}{rcl}
 (7)
&
      \stackrel{}{=}
      &
      \Sum^{n+1}_{p=3} \Sum^{n+1}_{r=p} \Sum^{p-2}_{k=1} \Sum^{p}_{i=k+2} \Sum^{p}_{j=i} 
      (-1)^{\foo (n-r)(p-j)+ (r-p-1)(p-i) + p + k}
      (\id \otimes \mu^c \otimes \id^{\otimes 2}) \circ (\Delta_{p, 2, k} \otimes \id^{\otimes 2})
 \\[1mm]
      &&
     \hspace*{7.5cm}
      \circ (\Delta_{p+1, r-p, i} \otimes \id)
      \circ \gD_{r,n+1-r,r+j-p}
 \\[1mm]
      &&
+ 
    \Sum^{n+1}_{p=2} \Sum^{n+1}_{r=p} \Sum^{p-1}_{k=1} \Sum^{p}_{j=k+1} 
      (-1)^{\foo (n-r)(p-j)+ (r-p)(p-k-1) + 1}
    (\id \otimes \mu^c \otimes \id^{\otimes 2}) \circ (\Delta_{p, 2, k} \otimes \id^{\otimes 2})
\\[1mm]
      &&
     \hspace*{7.5cm}
    \circ (\Delta_{p+1, r-p, k+1} \otimes \id)
      \circ \gD_{r,n+1-r,r+j-p}
     \\[1mm]
         &
      \stackrel{\scriptscriptstyle{}}{=:}
      &
(7') + (7'').
      \end{array}
\end{equation*} 
It is clear by inspection exchanging index names that $(7'')$ equals minus $(7''')$. Moreover, after a couple of operations similar to those before, and by re-indexing $i \mapsto p-i+1$ and $j \mapsto p-j+1$,
\begin{equation*}
    \begin{array}{rcl}
      &&
      (\id \otimes \tau) \circ
(\tau \otimes \id) \circ
(d_{\rm last} \otimes \id^{\otimes 2}) \circ (\tau \otimes \id) \circ (\id \otimes \tau) \circ F
      \\[1mm]
   &
      \stackrel{\scriptscriptstyle{\eqref{pici2}, \eqref{riso3}}}{=}
      &
      \Sum^{n}_{p=1} \Sum^{n}_{r=p} \Sum^{p}_{i=1} \Sum^{i}_{j=1} 
      (-1)^{\foo (n-r)(p-j) + (r-p-1)(p-i) + p + n}
      (\id^{\otimes 2} \otimes \mu^c \otimes \id) \circ (\id^{\otimes 2} \otimes \Delta_{2, n-r, 1})
 \\[1mm]
      &&
     \hspace*{6.5cm}
      \circ (\Delta_{p+1, r-p, p+1-i} \otimes \id)
      \circ \gD_{r,n+1-r,r+j-p}
     \\[1mm]
   &
      \stackrel{\scriptscriptstyle{\eqref{orzo}}}{=}
      &
      \Sum^{n}_{p=1} \Sum^{n}_{r=p} \Sum^{p}_{j=1} \Sum^{j}_{i=1} 
      (-1)^{\foo (n-r)(p-j) + (r-p-1)(p-i) + p + n}
      (\id \otimes \mu^c \otimes \id^{\otimes 2}) \circ (\Delta_{p+1, 2, j+1} \otimes \id^{\otimes 2})
\\[1mm]
      &&
     \hspace*{6.5cm}
       \circ (\Delta_{p+2, r-p, i} \otimes \id)
      \circ \gD_{r+1,n-r,r+j-p}
       \\[1mm]
         &
      \stackrel{\scriptscriptstyle{}}{=:}
      &
       (6'''),
      \end{array}
\end{equation*} 
holds true, 
and likewise
\begin{equation*}
    \begin{array}{rcl}
      &&
      (\id \otimes \tau) \circ
(\tau \otimes \id) \circ
(d_0 \otimes \id^{\otimes 2}) \circ (\tau \otimes \id) \circ (\id \otimes \tau) \circ F
      \\[1mm]
   &
     \stackrel{\scriptscriptstyle{}}{=}
     &
    \Sum^{n}_{p=1} \Sum^{n}_{r=p} \Sum^{p}_{j=1} \Sum^{j}_{i=1} 
      (-1)^{\foo (n-r)(p-j) + (r-p-1)(p-i-1)+1}
    (\id \otimes \mu^c \otimes \id^{\otimes 2}) \circ (\Delta_{p+1, 2, j+1} \otimes \id^{\otimes 2})
 \\[1mm]
      &&
     \hspace*{6.5cm}
    \circ (\Delta_{p+2, r-p, i} \otimes \id)
      \circ \gD_{r+1,n-r,r+j+1-p}
        \\[1mm]
         &
      \stackrel{\scriptscriptstyle{}}{=:}
      &
       (6''''),
      \end{array}
\end{equation*}
can be obtained. On the other hand, decomposing and re-indexing $p \mapsto p-1$, $r \mapsto r-1$, and $k \mapsto k-1$, and calling $j$ by $k$ and vice versa, 
\begin{equation*}
    \begin{array}{rcl}
 (6)
&
      \stackrel{}{=}
      &
      \Sum^{n}_{p=1} \Sum^{n}_{r=p} \Sum^{p}_{j=1} \Sum^{j}_{i=1} \Sum^{p+1}_{k=i} 
      (-1)^{\foo (n-r-1)(p-1-k) + (r-p-1)(p-i-1) + p+j}
      (\id \otimes \mu^c \otimes \id^{\otimes 2})
\\[1mm]
      &&
     \hspace*{4.5cm}
       \circ (\Delta_{p+1, 2, j+1} \otimes \id^{\otimes 2}) \circ (\Delta_{p+2, r-p, i} \otimes \id)
      \circ \gD_{r+1,n-r,r+k-p}
      \\[1mm]
      &
      \stackrel{}{=}
      &
      \pig(\Sum^{n}_{p=2} \Sum^{n}_{r=p} \Sum^{p}_{j=2} \Sum^{j-1}_{i=1} \Sum^{j-1}_{k=i}
      +
        \Sum^{n}_{p=2} \Sum^{n}_{r=p} \Sum^{p-1}_{j=1} \Sum^{j}_{i=1} \Sum^{p+1}_{k=j+2} \pig) (-1)^{\foo (n-r-1)(p-1-k) + (r-p-1)(p-i-1) + p+j} 
\\[4mm]
      &&
     \hspace*{2.0cm}
       (\id \otimes \mu^c \otimes \id^{\otimes 2})   \circ (\Delta_{p+1, 2, j+1} \otimes \id^{\otimes 2}) \circ (\Delta_{p+2, r-p, i} \otimes \id)
     \circ \gD_{r+1,n-r,r+k-p}
     \\[1mm]
     &&
     +  \Sum^{n}_{p=1} \Sum^{n}_{r=p} \Sum^{p}_{j=1} \Sum^{j}_{i=1} 
     (-1)^{\foo (n-r)(p-1-j) + (r-p-1)(p-i-1) + 1} (\id \otimes \mu^c \otimes \id^{\otimes 2})
     \circ (\Delta_{p+1, 2, j+1} \otimes \id^{\otimes 2})
 \\[1mm]
      &&
     \hspace*{7.5cm}
     \circ (\Delta_{p+2, r-p, i} \otimes \id)
           \circ \gD_{r+1,n-r,r+j-p}
     \\[1mm]
     &&
     +  \Sum^{n}_{p=1} \Sum^{n}_{r=p} \Sum^{p}_{j=1} \Sum^{j}_{i=1} 
     (-1)^{\foo (n-r)(p-j) + (r-p-1)(p-i-1)}
     (\id \otimes \mu^c \otimes \id^{\otimes 2})
     \circ (\Delta_{p+1, 2, j+1} \otimes \id^{\otimes 2})
 \\[1mm]
      &&
     \hspace*{7.5cm}
     \circ (\Delta_{p+2, r-p, i} \otimes \id)
          \circ \gD_{r+1,n-r,r+j+1-p}
          \\[1mm]
          &
      \stackrel{\scriptscriptstyle{}}{=:}
      &
(9) + (10) + (6') + (6''),
      \end{array}
\end{equation*} 
we see that $(6')$ equals minus the term $(6''')$, while $(6'')$ is equal to minus $(6'''')$.

At this stage, we are left with the terms $(7')$, $(9)$, as well as $(10)$, and still have to compute the terms $\sum^{n-1}_{k=1} (-1)^k F \circ d_k$, and also  $\sum^{{\rm max}-1}_{k=1} (-1)^k (\id \otimes d_k \otimes \id) \circ F$ as well as  $\sum^{{\rm max}-1}_{k=1} (-1)^k (\id^{\otimes 2} \otimes d_k) \circ F$, all with their respective signs.
To start with,
\begin{equation*}
    \begin{array}{rcl}
&& \Sum^{n-1}_{k=1} (-1)^k F \circ d_k
\\[1mm]
      &
      \stackrel{\scriptscriptstyle{\eqref{pici2}, \eqref{riso2}}}{=}
      &
      \Sum^{n}_{p=1} \Sum^{n}_{r=p} \Sum^{p}_{i=1} \Sum^{i}_{j=1} \Sum^{n-1}_{k=1} 
      (-1)^{\foo (n-1-r)(j-1)+ (r-p-1)(i-1) + p + k}
      (\id^{\otimes 3} \otimes \mu^c) \circ (\Delta_{p+1, r-p, p-i+1} \otimes \id^{\otimes 2})
\\[1mm]
 &&
 \hspace*{8.5cm}
       \circ (\Delta_{r, n-r, r+1-j} \otimes \id)
      \circ \gD_{n-1,2,k}
      \\[1mm]
         &
      \stackrel{\scriptscriptstyle{}}{=:}
      &
(11),
      \end{array}
\end{equation*} 
while
\begin{equation*}
    \begin{array}{rcl}
      &&
- \Sum^{{\rm max}-1}_{k=1} (-1)^k (\tau \otimes \id) \circ (d_k \otimes \id^{\otimes 2}) \circ (\tau \otimes \id) \circ F
\\[1mm]
      &
      \stackrel{\scriptscriptstyle{\eqref{pici2}, \eqref{riso2}, \eqref{orzo}}}{=}
      &
      \Sum^{n+1}_{p=1} \Sum^{n+1}_{r=p} \Sum^{p}_{i=1} \Sum^{i}_{j=1} \Sum^{r-p-1}_{k=1} 
      (-1)^{\foo (n-r)(j-1)+ (r-p-1)(i-1) + k}
      (\id^{\otimes 2} \otimes \mu^c \otimes \id)
\\[1mm]
     &&     \hspace*{2.5cm}
      \circ (\Delta_{p+1, r-p-1, p-i+1} \otimes \id^{\otimes 2})
       \circ (\Delta_{r-1, 2, k+p-i} \otimes \id)
      \circ \gD_{r,n-r+1,r-j+1}
\\[1mm]
      &
      \stackrel{\scriptscriptstyle{}}{=}
      &
      \Sum^{n-1}_{p=1} \Sum^{n+1}_{r=p+2} \Sum^{p}_{i=1} \Sum^{i}_{j=1} \Sum^{r-p-1}_{k=1} 
      (-1)^{\foo (n-r)(j-1)+ (r-p-1)(i-1) + k}
      (\id^{\otimes 2} \otimes \mu^c \otimes \id)
\\[1mm]
     &&     \hspace*{2.5cm}
      \circ (\Delta_{p+1, r-p-1, p-i+1} \otimes \id^{\otimes 2})
       \circ (\Delta_{r-1, 2, k+p-i} \otimes \id)
      \circ \gD_{r,n-r+1,r-j+1}
\\[1mm]
      &
      \stackrel{\scriptscriptstyle{\eqref{orzo}}}{=}
      &
      \Sum^{n-1}_{p=1} \Sum^{n+1}_{r=p+2} \Sum^{p}_{i=1} \Sum^{i}_{j=1} \Sum^{r-p-1}_{k=1} 
      (-1)^{\foo (n-r)(j-1)+ (r-p-1)(i-1) + k}
      (\id^{\otimes 3} \otimes \mu^c)
\\[1mm]
     &&     \hspace*{2.5cm}
      \circ (\Delta_{p+1, r-p-1, p-i+1} \otimes \id^{\otimes 2})
       \circ (\Delta_{r-1, n-r+1, r-j} \otimes \id)
       \circ \gD_{n-1,2,k+p-i}
       \\[1mm]
      &
      \stackrel{\scriptscriptstyle{}}{=}
      &
      \Sum^{n-1}_{p=1} \Sum^{n}_{r=p+1} \Sum^{p}_{i=1} \Sum^{i}_{j=1} \Sum^{r-i}_{k=p-i+1} 
      (-1)^{\foo (n-r-1)(j-1)+ (r-p-1)(i-1) + k + p +1}
      (\id^{\otimes 3} \otimes \mu^c)
\\[1mm]
     &&     \hspace*{4.0cm}
      \circ (\Delta_{p+1, r-p, p-i+1} \otimes \id^{\otimes 2})
       \circ (\Delta_{r, n-r, r-j+1} \otimes \id)
      \circ \gD_{n-1,2,k}     
      \\[1mm]
         &
      \stackrel{\scriptscriptstyle{}}{=:}
      &
(12),
      \end{array}
\end{equation*} 
where we re-indexed $r \mapsto r-1$ and $k \mapsto k+p-i$ in the fourth step.
Likewise, 
\begin{equation*}
    \begin{array}{rcl}
      &&
\Sum^{{\rm max}-1}_{k=1} (-1)^k
      (\id \otimes \tau) \circ
(\tau \otimes \id) \circ
       (d_k
      \otimes \id^{\otimes 2}) \circ (\tau \otimes \id) \circ (\id \otimes \tau) \circ F
\\[1mm]
      &
      \stackrel{\scriptscriptstyle{\eqref{pici2}, \eqref{riso2}}}{=}
      &
      \Sum^{n+1}_{p=1} \Sum^{n+1}_{r=p} \Sum^{p}_{i=1} \Sum^{i}_{j=1} \Sum^{n-r}_{k=1} 
      (-1)^{\foo (n-r)(j-1)+ (r-p-1)(i-1) + r+p+k + 1}
      (\id^{\otimes 3} \otimes \mu^c) \circ (\id^{\otimes 2} \otimes \Delta_{n-r, 2, k})
\\[1mm]
     &&     \hspace*{7.1cm}
       \circ (\Delta_{p+1, r-p, p-i+1} \otimes \id)
      \circ \gD_{r,n-r+1,r-j+1}
\\[1mm]
      &
      \stackrel{\scriptscriptstyle{}}{=}
      &
      \Sum^{n-1}_{p=1} \Sum^{n-1}_{r=p} \Sum^{p}_{i=1} \Sum^{i}_{j=1} \Sum^{n-r}_{k=1} 
      (-1)^{\foo (n-r)(j-1)+ (r-p-1)(i-1) + r+p+k + 1}
      (\id^{\otimes 3} \otimes \mu^c)
\\[1mm]
     &&     \hspace*{4.2cm}
      \circ (\Delta_{p+1, r-p, p-i+1} \otimes \id^{\otimes 2})
       \circ (\id \otimes \Delta_{n-r, 2, k})
      \circ \gD_{r,n-r+1,r-j+1}
\\[1mm]
      &
      \stackrel{\scriptscriptstyle{\eqref{orzo}}}{=}
      &
      \Sum^{n-1}_{p=1} \Sum^{n-1}_{r=p} \Sum^{p}_{i=1} \Sum^{i}_{j=1} \Sum^{n-r}_{k=1} 
      (-1)^{\foo (n-r)(j-1)+ (r-p-1)(i-1) + r+p+k + 1}
      (\id^{\otimes 3} \otimes \mu^c)
\\[1mm]
     &&     \hspace*{3.9cm}
      \circ (\Delta_{p+1, r-p, p-i+1} \otimes \id^{\otimes 2})
       \circ (\Delta_{r, n-r, r-j+1} \otimes \id)
       \circ \gD_{n-1,2,k+r-j}
       \\[1mm]
      &
      \stackrel{\scriptscriptstyle{}}{=}
      &
   \Sum^{n-1}_{p=1} \Sum^{n-1}_{r=p} \Sum^{p}_{i=1} \Sum^{i}_{j=1} \Sum^{n-j}_{k=r-j+1} 
      (-1)^{\foo (n-r-1)(j-1)+ (r-p-1)(i-1) + p+k + 1}
   (\id^{\otimes 3} \otimes \mu^c)
\\[1mm]
     &&     \hspace*{4.5cm}
   \circ (\Delta_{p+1, r-p, p-i+1} \otimes \id^{\otimes 2})
       \circ (\Delta_{r, n-r, r-j+1} \otimes \id)
       \circ \gD_{n-1,2,k}
         \\[1mm]
         &
      \stackrel{\scriptscriptstyle{}}{=:}
      &
(13),
      \end{array}
\end{equation*} 
by re-indexing $k \mapsto k+r-j$ in the last step. Moreover, the term $(7')$ is of the form as $(11)$--$(13)$.

\noindent Indeed, 
\begin{equation*}
    \begin{array}{rcl}
 (7')
&
      \stackrel{\scriptscriptstyle{\eqref{orzo}}}{=}
      &
      \Sum^{n+1}_{p=3} \Sum^{n+1}_{r=p} \Sum^{p-2}_{k=1} \Sum^{p}_{i=k+2} \Sum^{p}_{j=i} 
      (-1)^{\foo (n-r)(p-j)+ (r-p-1)(p-i) + p+k}
      (\id^{\otimes 2} \otimes \mu^c \otimes \id)
 \\[1mm]
      &&
 \hspace*{4.8cm}
 \circ (\Delta_{p, r-p, i-1} \otimes \id^{\otimes 2})
      \circ (\Delta_{r-1, 2, k} \otimes \id)
     \circ \gD_{r,n-r+1,r+j-p}
                 \end{array}
    \end{equation*}

    
       \begin{equation*}
    \begin{array}{rcl}
 &
      \stackrel{\scriptscriptstyle{}}{=}
      &
      \Sum^{n}_{p=2} \Sum^{n}_{r=p} \Sum^{p-1}_{i=1} \Sum^{i}_{j=1} \Sum^{p-i}_{k=1} 
      (-1)^{\foo (n-r-1)(j-1)+ (r-p-1)(i-1) + p+k+1}
      (\id^{\otimes 2} \otimes \mu^c \otimes \id)
\\[1mm]
      &&
     \hspace*{4.8cm}
       \circ (\Delta_{p+1, r-p, p-i+1} \otimes \id^{\otimes 2})
           \circ (\Delta_{r, 2, k} \otimes \id) \circ \gD_{r+1,n-r,r-j+2}
\\[1mm]
      &
      \stackrel{\scriptscriptstyle{\eqref{orzo}}}{=}
      &
      \Sum^{n}_{p=2} \Sum^{n}_{r=p} \Sum^{p-1}_{i=1} \Sum^{i}_{j=1} \Sum^{p-i}_{k=1} 
      (-1)^{\foo (n-r-1)(j-1)+ (r-p-1)(i-1) + p+k+1}
      (\id^{\otimes 3} \otimes \mu^c)
 \\[1mm]
      &&
     \hspace*{4.8cm}
      \circ (\Delta_{p+1, r-p, p-i+1} \otimes \id^{\otimes 2})
      \circ (\Delta_{r, n-r, r-j+1} \otimes \id)
\circ \Delta_{n-1, 2, k}
\\[1mm]
         &
      \stackrel{\scriptscriptstyle{}}{=:}
      &
(14),
      \end{array}
\end{equation*} 
where we re-indexed $p \mapsto p-1$, $r \mapsto r-1$, $j \mapsto p-j +1$, as well as $i \mapsto p-i+2$, and also reordered the sums in the second step. By mere inspection, one observes:
\begin{equation*}
    \begin{array}{rcl}
      &&
      (11)+(14)+(12)+(13)
            \\[1mm]
      &
      \stackrel{\scriptscriptstyle{}}{=}
      &
      \Sum^{n}_{p=2} \Sum^{n}_{r=p} \Sum^{p}_{i=1} \Sum^{i}_{j=1}
      \pig(
\Sum^{n-1}_{k=1}
-  \   \Sum^{p-i}_{k=1}
- \!\! \Sum^{r-i}_{k=p-i+1}
  -  \!\!  \Sum^{n-j}_{k=r-j+1}
  \pig)
      (-1)^{\foo (n-r-1)(j-1)+ (r-p-1)(i-1) + p+k}
       \\[3mm]
  &&
     \hspace*{3.1cm}
(\id^{\otimes 3} \otimes \mu^c)
     \circ (\Delta_{p+1, r-p, p-i+1} \otimes \id^{\otimes 2})
     \circ (\Delta_{r, n-r, r-j+1} \otimes \id)
\circ \Delta_{n-1, 2, k}
            \\[3mm]
      &
      \stackrel{\scriptscriptstyle{}}{=}
      &
      \Sum^{n}_{p=2} \Sum^{n}_{r=p} \Sum^{p}_{i=1} \Sum^{i}_{j=1}
      \pig(
\Sum^{n-1}_{k=r-i+1}
  -    \Sum^{n-j}_{k=r-j+1}
  \pig)
      (-1)^{\foo (n-r-1)(j-1)+ (r-p-1)(i-1) + p+k}
      (\id^{\otimes 3} \otimes \mu^c)
 \\[3mm]
      &&
     \hspace*{4.8cm}
  \circ (\Delta_{p+1, r-p, p-i+1} \otimes \id^{\otimes 2})
     \circ (\Delta_{r, n-r, r-j+1} \otimes \id)
\circ \Delta_{n-1, 2, k}
\\[1mm]
     &
      \stackrel{\scriptscriptstyle{}}{=}
      &
      \Sum^{n}_{p=2} \Sum^{n}_{r=p} \Sum^{p}_{i=1} \Sum^{i}_{j=1}
      \pig(
\Sum^{r-j}_{k=r-i+1}
  +    \Sum^{n-1}_{k=n-j+1}
  \pig)
      (-1)^{\foo (n-r-1)(j-1)+ (r-p-1)(i-1) + p+k}
      (\id^{\otimes 3} \otimes \mu^c)
 \\[3mm]
      &&
     \hspace*{4.8cm}
  \circ (\Delta_{p+1, r-p, p-i+1} \otimes \id^{\otimes 2})
     \circ (\Delta_{r, n-r, r-j+1} \otimes \id)
\circ \Delta_{n-1, 2, k}
\\[1mm]
&
      \stackrel{\scriptscriptstyle{}}{=}
      &
   \pig( \Sum^{n}_{p=2} \Sum^{n}_{r=p} \Sum^{p}_{i=2} \Sum^{i-1}_{j=1}
      \Sum^{r-j}_{k=r-i+1}
      +
\Sum^{n}_{p=2} \Sum^{n}_{r=p} \Sum^{p}_{i=2} \Sum^{i}_{j=2}
      \Sum^{n-1}_{k=n-j+1}
  \pig)
      (-1)^{\foo (n-r-1)(j-1)+ (r-p-1)(i-1) + p+k}
     \\[3mm]
      &&
     \hspace*{3.1cm}
 (\id^{\otimes 3} \otimes \mu^c)
     \circ (\Delta_{p+1, r-p, p-i+1} \otimes \id^{\otimes 2})
     \circ (\Delta_{r, n-r, r-j+1} \otimes \id)
\circ \Delta_{n-1, 2, k}
\\[1mm]
         &
      \stackrel{\scriptscriptstyle{}}{=:}
      &
(15) + (16).
      \end{array}
\end{equation*} 
The alert reader will have noticed that the only thing left to show is that these two terms correspond to $(9)$ and $(10)$. Indeed, let us show this to finalise this admittedly long proof:
\begin{equation*}
    \begin{array}{rcl}
  (9)
&
      \stackrel{\scriptscriptstyle{}}{=}
      &
    \Sum^{n}_{p=2} \Sum^{n}_{r=p} \Sum^{p}_{j=2} \Sum^{j-1}_{i=1} \Sum^{j-1}_{k=i}
     (-1)^{\foo (n-r-1)(p-k-1)+ (r-p-1)(p-i-1) + p+j} (\id \otimes \mu^c \otimes \id^{\otimes 2})
   \\[1mm]
      &&
     \hspace*{4.5cm}
    \circ (\Delta_{p+1, 2, j+1} \otimes \id^{\otimes 2})
 \circ (\Delta_{p+2, r-p, i} \otimes \id)
     \circ \gD_{r+1,n-r,r+k-p}
     \\[1mm]
     &
      \stackrel{\scriptscriptstyle{}}{=}
      &
    \Sum^{n}_{p=2} \Sum^{n}_{r=p} \Sum^{p}_{j=2} \Sum^{j-1}_{i=1} \Sum^{j-1}_{k=i}
     (-1)^{\foo (n-r-1)(p-k-1)+ (r-p-1)(p-i-1) + p+j} (\id^{\otimes 3} \otimes \mu^c)
\circ ( \id^{\otimes 2} \otimes \gs) 
    \\[1mm]
      &&
     \hspace*{2.2cm}
\circ (\id \otimes \gs \otimes \id)
     \circ (\Delta_{p+1, 2, j+1} \otimes \id^{\otimes 2}) \circ (\Delta_{p+2, r-p, i} \otimes \id)
     \circ \gD_{r+1,n-r,r+k-p}
     \\[1mm]
        &
      \stackrel{\scriptscriptstyle{\eqref{orzo}}}{=}
      &
    \Sum^{n}_{p=2} \Sum^{n}_{r=p} \Sum^{p}_{j=2} \Sum^{j-1}_{i=1} \Sum^{j-1}_{k=i}
     (-1)^{\foo (n-r-1)(p-k-1)+ (r-p-1)(p-i-1) + p+j} (\id^{\otimes 3} \otimes \mu^c)
 \circ  (\Delta_{p+1, r-p, i} \otimes \id^{\otimes 2})
    \\[1mm]
      &&
     \hspace*{6.5cm}
 \circ ( \id \otimes \gs)
     \circ (\Delta_{r, 2, j+r-p} \otimes \id)
     \circ \gD_{r+1,n-r,r+k-p}
     \\[1mm]
             &
      \stackrel{\scriptscriptstyle{\eqref{orzo}}}{=}
      &
    \Sum^{n}_{p=2} \Sum^{n}_{r=p} \Sum^{p}_{j=2} \Sum^{j-1}_{i=1} \Sum^{j-1}_{k=i}
     (-1)^{\foo (n-r-1)(p-k-1)+ (r-p-1)(p-i-1) + p+j} (\id^{\otimes 3} \otimes \mu^c)
 \circ  (\Delta_{p+1, r-p, i} \otimes \id^{\otimes 2})
 \\[1mm]
      &&
     \hspace*{7.7cm}
\circ (\Delta_{r, n-r, r+k-p} \otimes \id)
     \circ \gD_{n-1,2,n+j-p-1}
     \\[1mm]
     &
      \stackrel{\scriptscriptstyle{}}{=}
      &
    \Sum^{n}_{p=2} \Sum^{n}_{r=p} \Sum^{p-1}_{i=1} \Sum^{p-1}_{k=i} \Sum^{k}_{j=i}
     (-1)^{\foo (n-r-1)(p-j-1)+ (r-p-1)(p-i-1) + p+k+1} (\id^{\otimes 3} \otimes \mu^c)
 \\[1mm]
      &&
     \hspace*{4.5cm}
    \circ  (\Delta_{p+1, r-p, i} \otimes \id^{\otimes 2})
\circ (\Delta_{r, n-r, r+j-p} \otimes \id)
     \circ \gD_{n-1,2,n+k-p}
\\[1mm]
&
      \stackrel{\scriptscriptstyle{}}{=}
      &
    \Sum^{n}_{p=2} \Sum^{n}_{r=p} \Sum^{p}_{i=2} \Sum^{p-1}_{j=p-i+1} \Sum^{p-1}_{k=j}
     (-1)^{\foo (n-r-1)(p-j-1)+ (r-p-1)i + p+k+1} (\id^{\otimes 3} \otimes \mu^c)
 \\[1mm]
      &&
     \hspace*{4.0cm}
    \circ  (\Delta_{p+1, r-p, p-i+1} \otimes \id^{\otimes 2})
\circ (\Delta_{r, n-r, r+j-p} \otimes \id)
     \circ \gD_{n-1,2,n+k-p}
        \end{array}
     \end{equation*}

    
        \begin{equation*}
     \begin{array}{rcl}
&
      \stackrel{\scriptscriptstyle{}}{=}
      &
    \Sum^{n}_{p=2} \Sum^{n}_{r=p} \Sum^{p}_{i=2} \Sum^{i}_{j=2} \Sum^{p-1}_{k=p-j+1}
     (-1)^{\foo (n-r-1)j + (r-p-1)i + p+k+1} (\id^{\otimes 3} \otimes \mu^c)
 \circ  (\Delta_{p+1, r-p, p-i+1} \otimes \id^{\otimes 2})
 \\[1mm]
      &&
     \hspace*{7.7cm}
\circ (\Delta_{r, n-r, r-j+1} \otimes \id)
     \circ \gD_{n-1,2,n+k-p}
\\[1mm]
&
      \stackrel{\scriptscriptstyle{}}{=}
      &
    \Sum^{n}_{p=2} \Sum^{n}_{r=p} \Sum^{p}_{i=2} \Sum^{i}_{j=2} \Sum^{n-1}_{k=n-j+1}
     (-1)^{\foo (n-r-1)(j-1) + (r-p-1)(i-1) + p+k+1} (\id^{\otimes 3} \otimes \mu^c)
\\[1mm]
      &&
     \hspace*{4.2cm}
     \circ  (\Delta_{p+1, r-p, p-i+1} \otimes \id^{\otimes 2})
\circ (\Delta_{r, n-r, r-j+1} \otimes \id)
     \circ \gD_{n-1,2,k},
          \end{array}
\end{equation*}
where we re-indexed $j \mapsto j-1$ and afterwards exchanged the names of $j$ and $k$ in line five, reordered the sums as well as re-indexed $i \mapsto p-i+1$ in line six, re-indexed $j \mapsto p-j+1$ in line seven, and finally re-indexed $k \mapsto n+k-p$ in line eight. This is now directly seen to equal minus $(16)$.
In a similar fashion, one shows that $(10)$ and $(15)$ cancel out, the main difference in the computation being that for showing $(9) = - (16)$ the first line in the relations \eqref{orzo} read from right to left was used twice, while now one needs first the first line of the relations \eqref{orzo} and afterwards its last line, again in both cases from right to left; we omit this.
This concludes the proof of Proposition \ref{miglio}.
\end{proof}

\addtocontents{toc}{\SkipTocEntry}
\section*{Acknowledgements}

N.K.\ wishes to thank Sophie Chemla and the {\em Institut de Math\'ematiques de Jussieu -- Paris Rive Gauche}, where part of this work has been achieved, for hospitality and support.

F.P.\ wishes to thank the {\em Università di Roma Tor Vergata} for the invitation to speak at the Algebra Seminar, thanks to which this project took shape. She was supported by the European Commission under the Horizon 2020 research and innovation program, Marie Skłodowska-Curie grant agreement no.~945332.

  This paper is part of the activities of the MIUR Excellence Department
Project MatMod@TOV (CUP:E83C23000330006) and has been partially supported by the {\em National
  Group for Algebraic and Geometric Structures, and their Applications (GNSAGA-INdAM)}.

\providecommand{\bysame}{\leavevmode\hbox to3em{\hrulefill}\thinspace}
\providecommand{\MR}{\relax\ifhmode\unskip\space\fi M`R }
\providecommand{\MRhref}[2]{%
  \href{http://www.ams.org/mathscinet-getitem?mr=#1}{#2}}
\providecommand{\href}[2]{#2}

\end{document}